\documentclass{article}
\usepackage[utf8]{inputenc}
\usepackage{amsmath}
\usepackage{amssymb}
\usepackage{amsthm}
\usepackage{tikz-cd}
\usepackage{enumerate}
\usepackage[margin=1in]{geometry}
\usepackage{comment}
\usepackage{hyperref}

\newcommand{\N}{\mathbb{N}}

\newcommand{\overbar}[1]{\mkern 1.5mu\overline{\mkern-1.5mu#1\mkern-1.5mu}\mkern 1.5mu}

\newcommand{\Rm}{\mbox{Rm}}

\newtheorem{thm}{Theorem}[section]
\newtheorem{cor}[thm]{Corollary}

\newtheorem{prop}[thm]{Proposition}
\newtheorem{lem}[thm]{Lemma}

\newtheorem{defin}[thm]{Definition}

\numberwithin{equation}{section}

\title{Canonical surgeries in rotationally invariant Ricci flow}
\author{Timothy Buttsworth, Maximilien Hallgren and Yongjia Zhang}
\date{}
\newcommand{\Addresses}{{
  \bigskip
  \footnotesize

  T.~Buttsworth, \textsc{School of Mathematics and Physics, The University of Queensland, St Lucia 4067}\par\nopagebreak
  \textit{E-mail address}: \texttt{t.buttsworth@uq.edu.au}

  \medskip

  M.~Hallgren, \textsc{Department of Mathematics, Cornell University, Ithaca, NY 14850}\par\nopagebreak
  \textit{E-mail address}: \texttt{meh249@cornell.edu}

  \medskip

 Y. Zhang, \textsc{School of Mathematics, University of Minnesota, Minneapolis, MN 55455}\par\nopagebreak
  \textit{E-mail address}: \texttt{zhan7298@umn.edu}

}}

\begin{document}

\maketitle
\abstract{We construct a rotationally invariant Ricci flow through surgery starting at any closed rotationally invariant Riemannian manifold. We demonstrate that a sequence of such Ricci flows with surgery converges to a Ricci flow spacetime in the sense of \cite{KleinerLott17}. Results of Bamler-Kleiner \cite{BamlerKleiner17} and Haslhofer \cite{Haslhofer} then guarantee the uniqueness and stability of these spacetimes given initial data. We simplify aspects of this proof in our setting, and show that for rotationally invariant Ricci flows, the closeness of spacetimes can be measured by equivariant comparison maps. Finally we show that the blowup rate of the curvature near a singular time for these Ricci flows is bounded by the inverse of remaining time squared.}
\tableofcontents

\section{Introduction}

The Ricci flow is one of the most widely-studied geometric evolution equations. The substantial interest in Ricci flow is derived from its tendency to smooth out irregularities in a given geometry, deforming generic shapes into more canonical ones.

The utility of the Ricci flow is obstructed by the fact that singularities tend to develop in the geometry before the flow has completed the regularising process. One successful approach to resolving this issue involves removing singularities as they occur so that the flow can continue. 
Perhaps the most notable use of this is in Perelman's celebrated resolution of the Poincar\'e and Geometrisation conjectures. This process is known as \textit{Ricci flow with surgery}, and becomes useful if the topology and geometry of the removed singularity is well-understood. The surgery process typically involves obtaining a comprehensive understanding of regions of the Ricci flow with large curvature.

One particularly unsatisfying aspect of Ricci flow with surgery is that it is not especially canonical, in the sense that the geometric alterations to the flow depend on a number of parameters, including a choice of curvature `tolerance' which determines how much of the geometry is removed before restarting the Ricci flow. One would hope that by taking a limit of these parameters, one could obtain a Ricci flow with surgery that does not depend on any parameters. The resulting \textit{Ricci flow space-time} that arises from this limiting procedure can be thought of as a weak solution of the Ricci flow. 
The rigorous theory of these weak solutions of Ricci flow was introduced in \cite{KleinerLott17} where the authors focused their study on Ricci flow for three-dimensional manifolds. 

Later, the authors of \cite{BamlerKleiner17} demonstrated that these weak solutions were uniquely determined by their initial data, thus demonstrating that these weak solutions were \textit{canonical} solutions of Ricci flow with surgery. This proof involved demonstrating that two weak solutions that are initially close to each other must remain that way for all future times. This observation is achieved primarily through analysing the linearised Ricci flow equations, but care is needed since the flows can often encounter singularities. Indeed, the analysis is carried out on certain \textit{comparison domains}, which provide the tools needed to rigorously compare the geometry of two weak solutions \textit{even at singular times}. 

In this paper, we consider Ricci flows in dimensions greater than three, but specialize to the case where the initial metric is rotationally invariant. Such Ricci flows have been extensively studied in the literature, generally through the use of barrier/maximum principle arguments. In \cite{AngenentKnopf04} and \cite{AngenentKnopf05}, the authors construct a rotationally-symmetric Ricci flow developing a family of neckpinch singularities, and describe the precise convergence rates and asymptotic profiles of these solutions after a Type-I rescaling.  In \cite{AngenentCaputoKnopf}, it is shown that the Ricci flow can be continued canonically through these singularities, and asymptotic profiles for the flow recovering from surgery are also derived. In \cite{AngenentIsenbergKnopf}, the authors construct a family of Ricci flows developing degenerate neckpinch singularities, and in \cite{Carson16} it is again shown that there is a canonical flow through singularities. These results all give detailed information about the singularity formation and recovery, but only apply to carefully constructed classes of initial metrics. 

Throughout this paper, we forego the strict assumptions of the aforementioned results, and demonstrate in general the well-posedness of weak solutions of Ricci flow (in the sense of \cite{KleinerLott17}) for rotationally invariant metrics in any dimension. Given our weaker assumptions, we are not able to obtain precise asymptotics for how the singularities form, so our results may be viewed as complementary to the precise special cases considered by previous works. 

Our first main result describes the construction of rotationally invariant Ricci flow spacetimes by establishing long-time existence of a rotationally invariant Ricci flow with surgery, taking surgery parameters to zero, and taking a limit of spacetimes.

\begin{thm}(Existence) Given any closed, rotationally invariant Riemannian manifold $(M^{n+1},g_0)$, there exists a Ricci flow with surgery $(M,(g_t)_{t\in [0,\infty)})$ which is $\kappa_0(n)$-noncollapsed below scale 1 for all $t\in [0,\infty)$. The surgery procedure is described in Section 4.1 ; taking surgery parameters to zero gives a rotationally invariant Ricci flow spacetime (in the sense of \cite{KleinerLott17}, see Section 4.2) which becomes extinct in finite time. Moreover, for any $\epsilon>0$, there exists $r>0$ such that any such spacetime satisfies the equivariant $\epsilon$-canonical neighborhood assumptions below the scale $r>0$.
\end{thm}

The proof of this theorem follows primarily by adapting several classical theories of three-dimensional Ricci flow to the higher-dimensional rotationally invariant Ricci flow case. Firstly, we prove a simplified Hamilton-Ivey curvature estimate as well as a no local collapsing theorem for rotationally invariant Ricci flows. Secondly, we establish a version of the Cheeger-Gromov-Hamilton convergence theorem for rotationally invariant flows. Although it is reasonable to expect that the limit of a sequence of rotationally invariant Ricci flows must bear some sort of symmetry,  it not clear from the original theory whether the sequence of diffeomorphisms in the definition of the Cheeger-Gromov-Hamilton convergence can be chosen to be equivariant. The novelty of our compactness theorem is to show that one can always choose the diffeomorphisms to be equivariant. In view of the complete classification of rotationally symmetric $\kappa$-solutions \cite{LiZhang18}, we thus establish an equivariant version of the canonical neighborhood theorem. The combination of all of these elements essentially allows us to reconstruct Perelman's surgery process \cite{Perelman03} for rotationally invariant flows, and the existence of singular Ricci flow follows step-by-step from \cite{KleinerLott17}. For the relevant definitions, see Sections 2-4. 

In \cite{Haslhofer}, Haslhofer showed the uniqueness and stability (in terms of initial time slices) of Ricci flow spacetimes satisfying the $\epsilon$-canonical neighborhood assumptions, generalizing Bamler-Kleiner's result. Together with Theorem 1.1, this proves the well-posedness of singular Ricci flow starting at any closed rotationally invariant Riemannian manifold. In our setting, we can refine the well-posedness statement by showing that the comparison maps between Ricci flow spacetimes can be taken to be $O(n+1)$-equivariant. 

\begin{thm} (Equivariant stability) If $(M,g_j)_{j\in \mathbb{N}}$ is a sequence of Riemannian manifolds converging in the smooth Cheeger-Gromov sense to $(M,g_{\infty})$, then the corresponding spacetimes $\mathcal{M}_{j}^{\infty}$ associated to $(M_j,g_j)$ converge in the sense of Ricci flow spacetimes to the Ricci flow spacetime $\mathcal{M}_{\infty}^{\infty}$ associated to $(M_{\infty},g_{\infty})$ via equivariant diffeomorphisms. 
\end{thm}

For an effective version of this theorem, see Section 10. The idea of the proof (as in \cite{BamlerKleiner17,Haslhofer}), roughly speaking, is to establish the stability of the Ricci-DeTurck perturbation as in Kotschwar's energy approach to the uniqueness of the Ricci flow \cite{Kotschwar14}. Specifically, for two Ricci flows on the same underlying manifold, their difference can be measured by a time-dependent diffeomorphism evolving by the harmonic map heat flow, and this difference, a time-dependent symmetric $(0,2)$-tensor field, is the Ricci-DeTurck perturbation. Therefore, if one can prove that for a Ricci-DeTurck perturbation, the initial smallness implies smallness for a long time, then the uniqueness of the Ricci flow is established.

Treating the singularities of our Ricci flow space-times is the greatest challenge in establishing this smallness propagation result. Indeed, the existence of singularities implies that one can never expect to connect two singular Ricci flows with a long-lasting time-dependent diffeomorphism evolving by the harmonic map heat flow, because the underlying manifolds of the two singular Ricci flows may possibly undergo wildly different alterations. Therefore, as shown in \cite{BamlerKleiner17,Haslhofer}, it is natural to construct a comparison domain excluding all the `almost singular points', and compare the Ricci flows on this domain through a time-dependent diffeomorphism called a comparison (see Sections 8-9 for more details). Once the stability of the Ricci-DeTurck perturbation is established in this case, one immediately obtains that two singular Ricci flows with identical initial data must coincide completely at all regular points. 

The comparsion domain and the comparison themselves are also objects that are difficult to construct. Since the singular points can possibly occur anywhere on the manifold, the comparison domain, being designed to exclude all the singular points, is not expected to be a space-time product, and this in turn introduces difficulties of imposing boundary conditions for the solving of the harmonic map heat flow. The idea to settle this issue is to create the comparison domain in the piecewise space-time-product fashion, namely, to create the comparison domain as consecutive time-slabs, each one of which is a space-time product. The canonical neighborhood theorem is useful in this construction, because it implies that a high-curvature region is either an almost round neck or close to a piece of scaled Bryant soliton; therefore, one may always arrange the boundary of a space-time-product slab of the comparison domain to be the central sphere of a shrinking neck. As we move to making alterations between two adjacent time-slabs of the comparison domain, one either discards a region of the first or adds to it, so as to obtain the second slab; in the former case, the harmonic map heat flow and the Ricci-DeTurck perturbation on the first slab is obviously inherited by the second; in the latter case, the canonical neighborhood assumption indicates that that which is added is nothing but a piece of a rescaled Bryant soliton, for which the Bryant extension theorem (Section 7) enables the continuation of the harmonic map heat flow and the Ricci-DeTurck perturbation.

Our techniques are mostly derived from those presented in \cite{KleinerLott17,BamlerKleiner17}, but there are a number of key ingredients that need altering. We conclude the introduction by summarising the new arguments we produce in the course of demonstrating existence and uniqueness of weak solutions in the context of rotationally-invariant Ricci flow:
\begin{itemize}
    \item we give an easy proof of the Hamilton-Ivey pinching estimate on higher-dimensional manifolds with rotational invariance, avoiding the difficult computation for general Weyl-flat Ricci flows done in \cite{Zhang15};
    \item we show that any simply-connected rotationally-invariant metric is automatically $\kappa$-noncollapsed for some universal $\kappa>0$, thus bypassing the difficult adaptation of Perelman's no local collapsing theorem to the Ricci flow with surgery;
    \item we develop equivariant compactness theorems for manifolds and spacetimes with respect to Cheeger-Gromov convergence;
    \item we give a new proof of the Anderson-Chow pinching estimates for higher dimensional rotationally invariant flows (the result is already known from Lemma 5.2 of \cite{Carson16}, and the more general \cite{ChenWu16}); 
    \item we obtain a simpler proof of the induction step in the construction of the Bamler-Kleiner comparison domain; 
    \item we show that the number of `bumps' does not increase, even at singular times; and
    \item we show that for any isolated singular time $t^{\ast}$, the curvature can only blow up at most at the rate $(t^{\ast}-t)^{-2}$ as $t\nearrow t^{\ast}$.
\end{itemize}
We observe that, in full generality, many of these arguments fail without imposing rotational invariance. In fact, even for slightly more complicated cohomogeneity one Ricci flows, it is possible to have nonuniqueness through singularities \cite{AngenentKnopf19}. 

Finally, we also note that our proof related to the Bamler-Kleiner comparison domain can be combined with the classification of three-dimensional noncompact $\kappa$-solutions presented in \cite{Brendle20} to produce a simpler proof in the three-dimensional case \textit{without} imposing rotational symmetry.

\section{Preliminaries, definitions and notation}
In this section, we recall some standard definitions from the theory of Ricci flow, and also describe the specialized theory of rotationally-invariant Ricci flow. 

\subsection{Notation and Conventions}

Given a solution $(M,(g_t)_{t\in [a,b]})$ of Ricci flow, we let $B_{g_t}(x,r)$ denote the geodesic ball measured at time $t$, and $d_{g_t}(x,y)$ the distance measured at time $t$. We let $P(x,t,r,\tau)$ denote the parabolic cylinder $B_{g_t}(x,r)\times [t,t+\tau]$ if $\tau\geq 0$, and the cylinder $B_{g_t}(x,r)\times [t+\tau,t]$ otherwise. We also set $P(x,t,r):=P(x,t,r,-r^2)$. Let $\mathcal{O}(y)$ denote the $O(n+1)$-orbit through $y$, and define
$$B_{\mathcal{O}}(x,r):=\cup_{y\in B(x,r)} \mathcal{O}(y), \hspace{6 mm} P_{\mathcal{O}}(x,t,r):=\cup_{y\in P(x,t,r)} \mathcal{O}(y).$$
We let $\overline{g}$ denote the standard metric on $\mathbb{R}^{n+1}$, and let $g_{\mathbb{S}^n}$ denote the standard metric on $\mathbb{S}^n$ with scalar curvature 1. We use the following convention for the curvature tensor
$$R(X,Y,Z,W)=\langle \nabla_X (\nabla_Y Z)-\nabla_Y (\nabla_X Z)-\nabla_{[X,Y]}Z,W\rangle$$
and we define the curvature operator acting on $\wedge^2 TM$ by 
$$\langle \mathcal{R}(X\wedge Y),Z\wedge W\rangle =2R(X,Y,W,Z).$$

\subsection{Rotationally-invariant metrics}
Consider the standard action of $O(n+1)$ on the two manifolds $\mathbb{S}^{n+1}\subset \mathbb{R}\times \mathbb{R}^{n+1} $ and $\mathbb{S}^1\times \mathbb{S}^{n}\subset \mathbb{S}^1\times \mathbb{R}^{n+1}$. 
\begin{defin}\label{defrotational}
We say that a Riemannian metric $g$ on $\mathbb{S}^{n+1}$ or $\mathbb{S}^1\times \mathbb{S}^n$ is \textit{rotationally-invariant} if the $O(n+1)$ group acts via isometries. We will also consider a Riemannian manifold $(M,g)$ to be rotationally-invariant if $M$ is equivariantly diffeomorphic to an $O(n+1)$-invariant regular subdomain of either $\mathbb{S}^1\times \mathbb{S}^n$ or $\mathbb{S}^{n+1}$, and $g$ is $O(n+1)$-invariant. For example, rotationally-invariant metrics on $\mathbb{R}^{n+1}$ and the cylinder $\mathbb{R}\times \mathbb{S}^{n}$ arise from $\mathbb{S}^{n+1}$ in this way. We let $M_{\operatorname{sing}}$ denote the union of all singular orbits of $M$, and  $M_{\operatorname{reg}}:=M\setminus M_{\operatorname{sing}}$.
\end{defin}

We next observe an equivalent formulation of rotationally-invariant Riemannian manifolds. 

\begin{prop}
A connected Riemannian manifold $(M^{n+1},g)$ is rotationally invariant if and only if it is orientable, and
admits a proper and faithful cohomogeneity-one isometric action by $O(n+1)$ which has $O(n)$ as the principal isotropy subgroup.
The singular orbits of the $O(n+1)$ action are fixed points, and if $M$ is connected, the orbit space
$M/O(n+1)$ is diffeomorphic to $\mathbb{R}$, $[0,\infty)$,
$[0,1]$, or $\mathbb{S}^{1},$ which corresponds to the cases where
$M$ is diffeomorphic to $\mathbb{R}\times\mathbb{S}^{n}$,  $\mathbb{R}^{n+1}$,
$\mathbb{S}^{n+1}$, or $\mathbb{S}^{1}\times\mathbb{S}^{n}$, respectively. In particular, $M_{\operatorname{sing}}$ consists of at most two points.
\end{prop}
\begin{proof} 
It is clear that the manifolds described in Definition \ref{defrotational} are all orientable, and that the actions of $O(n+1)$ on these manifolds are all faithful and proper, with principal isotropy subgroup given by $O(n)$.  

On the other hand, suppose that the connected Riemannian manifold $(M^{n+1},g)$ admits a proper and faithful cohomogeneity one action of the group $G=O(n+1)$, which acts via isometries, and has $H=O(n)$ as the principal isometry subgroup. Let $M_{reg}$ denote the set of all points whose isotropy subgroups are (conjugacy equivalent to) $H$. This is an open subset of $M$, and $M_{reg}^*=M_{reg}/O(n+1)$ is a smooth one-dimensional manifold, and comes equipped with a Riemannian metric induced by $g$. The space $M_{reg}^*$ is open and dense in $M^*=M/O(n+1)$, which is therefore a smooth one-dimensional manifold with boundary. We split into cases according to the topological type of $M^*$. 

\textbf{Case One:} $M^*=\mathbb{R}$. In this case, $M_{reg}^*=M^*$ and since $\mathbb{R}$ is contractible, $M$ must appear as the trivial bundle $\mathbb{R}\times G/H=\mathbb{R}\times \mathbb{S}^n$. 

\textbf{Case Two:} $M^*=\mathbb{S}^1$. In this case, we again have $M_{reg}^*=M^*$. Choose any point $p^*\in M/O(n+1)$, and consider $M\setminus \pi^{-1}(p^*)$. Then like in Case One, this manifold appears as $\mathbb{R}\times \mathbb{S}^n$. The original manifold is then found by identifying  the copies of $\mathbb{S}^n$ at their two ends. The are only two $O(n+1)$-invariant diffeomorphisms on $\mathbb{S}^n$, namely, $\pm I$, so there are only two possibilities: $\mathbb{S}^1\times \mathbb{S}^n$, or a non-trivial $\mathbb{S}^1$ bundle, but this second example is not oriented. 

\textbf{Case Three:} $M^*=[0,\infty)$. Let $K$ be the isotropy of the single non-principal orbit. It is well known that $K/H$ is a sphere, say, $\mathbb{S}^d$, and $K$ acts linearly on $\mathbb{R}^{d+1}$, and transitively on $\mathbb{S}^d$. Since $H=O(n)$, it is clear that $K=O(n+1)=G$, so that the singular orbit is just a point. The resulting manifold is $\mathbb{R}^{n+1}$.

\textbf{Case Four:} $M^*=[0,1]$. This is just a gluing of two of the manifolds from Case Three along a principal orbit; the result is the compact $\mathbb{S}^{n+1}$. 
\end{proof}

It is convenient to note that $g|M_{\operatorname{reg}}$ can always be expressed as a warped Riemannian product. 

\begin{prop} \label{rotdifgen} Suppose $(M^{n+1},g)$ is a connected, simply connected rotationally invariant Riemannian manifold and $M_{\operatorname{reg}}$ denotes the dense collection of principal  ($\mathbb{S}^n$) orbits. Then there is an $O(n+1)$-equivariant diffeomorphism $\Phi: (\alpha, \beta) \times \mathbb{S}^n \to M_{\operatorname{reg}}$
satisfying 
$$\Phi^{\ast}g=dr \otimes dr +\psi^2 (r)\overline{g},$$
where $\psi \in C^{\infty}((\alpha,\beta))$ is strictly positive. If instead $M \cong \mathbb{S}^1 \times \mathbb{S}^n$, then there is an $O(n+1)$-equivariant diffeomorphism $\Phi:(\mathbb{R}/\ell \mathbb{Z}) \times \mathbb{S}^n \to M$ satisfying 
$$\Phi^{\ast}g = g_{\frac{\ell}{2\pi} \mathbb{S}^1} + \psi^2(r)\overline{g},$$
where $\ell$ is the length of the $\mathbb{S}^1$ factor, and $(\mathbb{R}/\ell \mathbb{Z}, g_{\frac{\ell}{2\pi}\mathbb{S}^1})$ is the circle of length $\ell$.
\end{prop}
\begin{proof}
In the simply connected case, we let $\Phi$ be the exponential map of $M$ restricted to the normal bundle of some principal orbit $\mathcal{O}$. Otherwise, we note that the map $\Phi:\mathbb{R}\times \mathbb{S}^n \to \mathbb{S}^1\times \mathbb{S}^n$ passes to the quotient.
\end{proof}

At a point $p=(r_0,\omega)\in (0,1)\times \mathbb{S}^n$, let $e_1,\cdots,e_n$ denote an orthonormal basis for $g$ restricted to $\mathbb{S}^n$. The Riemann curvature operator of the Riemannian metric $dr\otimes dr+f(r)^2 \overline{g}\vert_{\mathbb{S}^{n}}$ at the point $p$ is diagonal in the basis $\{\partial_r\wedge e_1,\cdots, \partial_r\wedge e_n, e_1\wedge e_2,\cdots, e_{n-1}\wedge e_n\}$, and we have 
\begin{align*}
    \mathcal{R}(\partial_r\wedge e_i)=-2\frac{f''}{f} \partial_r\wedge e_i, \qquad \mathcal{R}(e_i\wedge e_j)=\frac{2(1-(f')^2)}{f^2} e_i\wedge e_j.
\end{align*}

\begin{prop}\label{uniform_noncollapsing}
There exists $\kappa=\kappa(n)>0$ such that any complete, rotationally
symmetric Riemannian manifold $(M^{n+1},g)$ not diffeomorphic to
$\mathbb{S}^{1}\times\mathbb{S}^{n}$ is $\kappa(n)$-noncollapsed
on all scales.
\end{prop}

\begin{proof}
We know that $M$ is equivariantly diffeomorphic to one of $\mathbb{R}^{n+1},\mathbb{S}^{n+1},\mathbb{R}\times\mathbb{S}^{n}$ equipped with their standard actions. Using Proposition \ref{rotdifgen}, we can find an $O(n+1)$-equivariant diffeomorphism $\Phi:(a,b)\times \mathbb{S}^{n}\to M_{\text{reg}}$ so that 
\[
\Phi^* g=dr^{2}+\psi^{2}(r)g_{\mathbb{S}^{n}},
\]
where $\psi>0$. If $a>-\infty$,
then the classical smoothness conditions (see section 1.4.4 of \cite{Petersen}) for rotationally-invariant metrics imply that $\psi(a)=0$ and $\psi'(a)=1$, and if $b<\infty$, then $\psi'(b)=-1$
and $\psi(b)=0$. 
Now assume $x_{0}\in M$ satisfies $|Rm|\leq 1$
on $B(x_{0},1)$, and set $r_0=r(x_0)$. The Riemann curvature bounds imply that \[
\frac{|1-(\psi')^{2}|}{\psi^2}+\frac{|\psi''|}{\psi}\leq1
\]
for $x\in\mathbb{S}^{n+1}$ with $r(x)\in(a,b)\cap[r_{0}-1,r_{0}+1]$. We aim to find a uniform lower bound for the volume of $B(x_0,1)$ which does not depend on $x_0$ or the choice of Riemannian metric; to do this, we will split into two cases.

\textbf{Case One:} $\psi(r_0)<\frac{1}{10}$. We assume without loss of generality that $\psi'(r_0)\le 0$; then the differential inequality $\psi'(r)\le -\sqrt{ 1-\psi(r)^2}$ implies that $\psi'(r)\le -\frac{1}{2}$ for all $r\in \left[r_0,\min\{b,r_0+1\}\right]$, so that $b-r_0\le \frac{1}{2}$. Therefore, $B(r^{-1}(b),\frac{1}{2})\subset B(x_0,1)$, where $r^{-1}(b)$ is the singular orbit at $r=b$.  Now, the differential inequality $\left|\psi''(r)\right|\le \psi(r)$, coupled with $\psi(b)=0$ and $\psi'(b)=-1$ implies that $b-a\ge 1$ and $\frac{\psi(b-r)}{\sin(r)}\geq 1$ for all $r\in (0,\frac{1}{10})$, so that
$\frac{\psi(b-r)}{r}\in [\frac{1}{2},2]$. 
Therefore 
\begin{align*}
    \operatorname{Vol}(B(x_0,1))&\ge \operatorname{Vol}(B(r^{-1}(b),\frac{1}{2}))=(n+1)\omega_{n+1}\int_{b-\frac{1}{2}}^{b} \psi(r)^n dr\ge (n+1)\omega_{n+1}\int_{b-\frac{1}{10}}^{b}\frac{(b-r)^n}{2^n}dr \\ &=(n+1)\omega_{n+1}\int_0^{\frac{1}{10}}\frac{r^n dr}{2^n}. 
\end{align*}

\textbf{Case Two:} $\psi(r_0)\ge \frac{1}{10}$. Then $\left|1-(\psi'(r))^2\right|\le \psi(r)^2$ implies that $\psi'(r)\ge -1-\psi(r)$ for all $r\in \left[r_0,\min\{b,r_0+1\}\right]$, so that $b-r_0>\frac{1}{100}$, and $\psi(r)\ge \frac{1}{100}$ for $r\in [r_0,r_0+\frac{1}{100}]$. Therefore, if we let $\theta_0$ be the element of $\mathbb{S}^n$ corresponding to $x_0$, then $B(x_0,1)$ contains the set 
$$\{(r,\theta)\in ([-\frac{1}{100}+r_0,\frac{1}{100}+r_0]\times \mathbb{S}^n \ \vert \  d_{\overline{g}|\mathbb{S}^n}(\theta,\theta_0)<10^{-6}\}. $$ 
The uniform lower bound on $\psi$ on this set yields the required volume lower bound via Fubini's theorem. 
\end{proof}
We note that the above proposition fails in the case $M=\mathbb{S}^{1}\times\mathbb{S}^{n}$,
since for example $(\mathbb{S}^{1}\times\mathbb{S}^{n},k^{-1}g_{\mathbb{S}^{1}}+\overline{g}\vert_{\mathbb{S}^{n}})$
is a sequence of metrics with volume converging to zero as $k\to\infty$,
but uniformly bounded curvature. We can still say that in this case
$M$ is $\kappa(n)$-noncollapsed up to the scale of the length of
the $\mathbb{S}^{1}$ factor, using the reasoning of the above proposition.
Moreover, we know that if $M$ is a rotationally invariant metric
which is some time slice of a Ricci flow with surgery after the first
surgery time, then $M$ cannot be diffeomorphic to $\mathbb{S}^{1}\times\mathbb{S}^{n}$.
Thus $\kappa(t)$-noncollapsing for a rotationally invariant Ricci
flow with surgery follows from the usual $\kappa$-noncollapsing theorem
up to the first singular time, while the need to prove $\kappa$-noncollapsing
through surgery is removed by the above proposition.

\subsection{Rotationally-invariant Ricci flow}
Since the Ricci curvature operator commutes with pullbacks by diffeomorphisms, the uniqueness of solutions to the initial value problem implies that a Ricci flow on $\mathbb{S}^{n+1}$ or $\mathbb{S}^1\times \mathbb{S}^n$ which is initially rotationally-invariant stays that way as long as the Ricci flow is defined. 
Proposition \ref{rotdifgen} can be used to find a convenient form of the initial Riemannian metric. However, this form is not preserved by the Ricci flow. Instead, the time-varying metric can be expressed as 
\begin{equation}\label{Formg}
g_t= \varphi(x,t)^2 dx\otimes dx+\psi(x,t)^2 \overline{g}\vert _{\mathbb{S}^{n}}.
\end{equation}

The Ricci curvature of such a metric is 
\begin{equation}\label{FormRicci}
\text{Ric}(g_t)=n\left(\frac{\psi'\varphi'}{\psi \varphi^3}-\frac{\psi''}{\varphi^2 \psi}\right)\varphi^2 dx\otimes dx+\left(-\frac{\psi''}{\varphi^2 \psi} -(n-1)\frac{(\psi')^2}{\psi^2 \varphi^2} +\frac{\varphi '\psi'}{\varphi ^3 \psi}+\frac{n-1}{\psi^2} \right)\psi^2 \overline{g}\vert_{\mathbb{S}^{n}}.
\end{equation}

For each $t\in[0,T)$, define $s_{t}(x):=\int_{0}^{x}\varphi(y,t)dy$,
which we view as a 1-parameter family of functions on $(-1,1)$ rather
than a function on spacetime, so that $\partial_{s}=\frac{1}{\varphi}\partial_{x}$
has no time-like component. Note that $s_{t}'(x)=\varphi(t,x) >0$ implies
that $s_{t}$ is invertible. Then
$$g_t = ds_t \otimes ds_t +\psi^2(s_t,t)\overline{g},$$
and we can write the evolution equation satisfied by $\psi$ as
$$\partial_t \psi = \partial_s (\partial_s \psi) -(n-1)\left(\frac{1-(\partial_s \psi)^2}{\psi}\right),$$
where we view $\partial_s$ as a time-dependent vector field on spacetime.

\section{Geometry of high curvature regions}\label{geometryhighcurvature}

\subsection{Hamilton-Ivey pinching}
One of the major reasons why the study of Ricci flow has been so successful in three-dimensions is the Hamilton-Ivey pinching phenomenon, which ensures that near areas of unbounded curvature, the positive eigenvalues of the curvature operator dominate the negative eigenvalues in the sense of the Hamilton-Ivey estimate. This estimate ensures that blow-up limits near finite-time singularities have non-negative Riemann curvature. In four dimensions and higher, this pinching fails in full generality. Indeed, numerous authors have constructed higher-dimensional Ricci flows with finite-time singularities whose blow-up limits do not have non-negative Riemann curvature (see, for example, \cite{Appleton19}, \cite{GuoSong16} and \cite{Maximo14}). However, if we make the powerful assumption that our Ricci flow is rotationally-invariant, then we can recover the Hamilton-Ivey pinching property. 

\begin{defin}[Hamilton-Ivey pinching] We say that the rotationally-invariant Ricci flow $(M^{n+1},(g_t)_{t\in [0,T)})$ satisfies the Hamilton-Ivey estimate if for all $(x,t)\in M\times [0,T)$, we have 

\begin{align*}R(x,t)\geq -\nu(x,t) \left( \log(-\nu(x,t))+\log(1+t)-\frac{n(n+1)}{2} \right),
\end{align*}
where $\nu(x,t)$ denotes the smallest eigenvalue of $\mathcal{R}_{g_t}(x)$ with respect to $g_t$.
\end{defin}

\begin{thm}\label{rotpinch}
Let $(M,(g_t)_{t\in [0,T)})$ be an $n+1$-dimensional rotationally-invariant Ricci flow with $n\ge 2$. If $\nu(\cdot,0) \geq -1$ everywhere, then the flow satisfies the Hamilton-Ivey pinching estimate.
\end{thm}
Since rotationally-invariant Ricci flows have uniformly-vanishing Weyl tensors, Theorem \ref{rotpinch} is an immediate consequence of Theorem 1.1 in \cite{Zhang15}, which states that Hamilton-Ivey pinching occurs if the Weyl tensor vanishes for all times. However, the proof simplifies substantially in the case of rotationally-invariant Ricci flows. Since there does not seem to be many well-understood examples of Ricci flows with vanishing Weyl tensor that are \textit{not} rotationally-invariant, we find it appropriate to include the simpler proof in this section.

Naturally, the main ingredient in the proof of Theorem \ref{rotpinch} is Hamilton's tensor maximum principle:
\begin{thm}\label{hammax}
Let $V$ be a tensor bundle over $M$ with time-varying Riemannian metric $(g_t)_{t\in [0,T]}$. A smooth one-parameter family of sections $u_t:M\to V$ satisfying the diffusion equation 
\begin{align}\label{hamequation}
    \frac{\partial u}{\partial t}=\Delta_{g_t}u+F(x,u),
\end{align}
with $F$ a Lipschitz-continuous fiber-preserving map has the following maximum principle: for any closed, parallel, fiber-wise convex time-varying subset $K(t)\subset V$ which is invariant under the vector field $F(x,\cdot )$, we can conclude that $u(t,x)\in K_x(t)$ for all $(t,x)\in [0,T]\times M$ if $u(0,x)\in K_x(0)$ for all $x\in M$.
\end{thm}
As usual, the goal is to apply this tensor maximum principle to the Riemann curvature tensor $\mathcal{R}$, which is well-known to satisfy the following reaction-diffusion equation:
\begin{align*}
    \frac{d \mathcal{R}}{dt}=\Delta \mathcal{R}+\mathcal{R}^2+\mathcal{R}^{\#}
\end{align*}
after applying the Uhlenbeck trick (see Section 2.7 of \cite{ChowNi}, which also discusses the Lie algebra square $\mathcal{R}^{\#}$). Studying this evolution equation is substantially simplified in the rotationally-invariant setting because at any point on the manifold, there are only two eigenvalues of $\mathcal{R}$, namely  $-2\frac{\psi''}{\psi}=\lambda$ and $2\left(\frac{1-\psi'^2}{\psi^2}\right)=\mu$ (assuming arc-length parametrisation) so that the scalar curvature is given by $R=n\lambda+\frac{n(n-1)}{2}\mu$. The associated ODE $\frac{d\mathcal{R}}{dt}=\mathcal{R}^2+\mathcal{R}^{\#}$ then becomes 
\begin{align}\label{curvatureODEs}
\begin{split}
    \frac{d\lambda }{dt}&=\lambda^2+(n-1)\lambda\mu,\\
    \frac{d \mu}{dt}&=(n-1)\mu^2+\lambda^2.
    \end{split}
\end{align}
\begin{lem}\label{rotinvarode}
Assume that $n\ge 2$ and define $\nu=\min\{\lambda,\mu\}$. Then solutions of \eqref{curvatureODEs} satisfy
$$\frac{dR}{dt}\geq \frac{2R^2}{n+1},$$
$$(-\nu)\frac{d R}{dt}+\frac{d \nu}{dt}(R-\nu)\ge (-\nu)^3,$$ whenever $\nu\le 0$. 
\end{lem}
\begin{proof}
First we compute 
\begin{align*}
    \frac{ dR}{dt}&=n\frac{d\lambda}{dt}+\frac{n(n-1)}{2}\frac{d\mu}{dt}\\
    &=n(\lambda^2+(n-1)\lambda\mu)+\frac{n(n-1)}{2}((n-1)\mu^2+\lambda^2)\\
    &=n\left(\frac{(n+1)}{2}\lambda^2+(n-1)\lambda\mu+\frac{(n-1)^2}{2}\mu^2\right).
\end{align*}
We then conclude that
\begin{align*}
    \frac{dR}{dt}-\frac{2R^2}{n+1}&=
    n\left(\frac{(n+1)}{2}\lambda^2+(n-1)\lambda\mu+\frac{(n-1)^2}{2}\mu^2\right)-\frac{2}{n+1}\left(n^2\lambda^2+n^2(n-1)\lambda\mu+\frac{n^2(n-1)^2}{4}\mu^2\right)\\
    &=\lambda^2n\left(\frac{(n+1)}{2}-\frac{2n}{n+1}\right)+\left(n\frac{(n-1)^2}{2}-\frac{n^2(n-1)^2}{2(n+1)}\right)\mu^2+\left(n(n-1)-\frac{2n^2(n-1)}{n+1}\right)\lambda\mu\\
    &=\frac{n(n-1)^2}{n+1}(\lambda -\mu)^2\ge 0. 
\end{align*}
Now the desired estimate obviously holds if $\nu=0$ because this implies that $R\ge 0$ and $\frac{d\nu}{dt}\ge 0$. We now assume that $\nu<0$, and break the analysis into cases depending on which of $\lambda$ or $\mu$ coincides with $\nu$. \\

\noindent 
\noindent\textbf{Case One:} $\mu=\nu<0$.
We want to show that $-\mu \frac{dR}{dt}+\frac{d\mu}{dt}(R-\mu)\ge (-\mu)^3$. 
We see that 
\begin{align*}
  &-\mu \frac{dR}{dt}+\frac{d\mu}{dt}(R-\mu)+\mu^3\\
  &=-\mu n\left(\frac{(n+1)}{2}\lambda^2+(n-1)\lambda\mu+\frac{(n-1)^2}{2}\mu^2\right)+\left((n-1)\mu^2+\lambda^2\right)\left(n\lambda+\left(\frac{n^2-n-2}{2}\right)\mu\right)+\mu^3\\
  &=n\lambda^3-(n+1)\mu\lambda^2-(n-2)\mu^3\\
&= n\lambda^2(\lambda-\mu)-\lambda^2 \mu -(n-2)\mu^3\ge 0
\end{align*}
since $n\geq 2$, $\mu<0$ and $\lambda\ge \mu$. \\

\noindent\textbf{Case Two:} $\lambda=\nu<0$. 
We want to show that $-\lambda \frac{dR}{dt} + \frac{d\lambda}{dt}(R - \lambda) \geq (-\lambda)^3$.  
\begin{align*}
    & - \lambda \frac{dR}{dt} + \frac{d\lambda}{dt} (R - \lambda) + \lambda^3 \\&= -\lambda n \left(\frac{(n+1)}{2}\lambda^2+(n-1)\lambda\mu+\frac{(n-1)^2}{2}\mu^2\right) + (n-1)(\lambda^2 + (n-1)\lambda \mu )\left(\lambda + \frac{n}{2}\mu\right) + \lambda^3\\
    &= -n \left(\frac{n-1}{2}\right)\lambda^3
+(n-1)\left(\frac{n}{2}-1\right)\lambda^2\mu     \\
&=\lambda^2(n-1)\left(-\lambda+\left(\frac{n}{2}-1\right)(\mu-\lambda)\right)\ge 0,
\end{align*}
since $n\geq 2$, $\mu \geq \lambda$ and $\lambda<0$.
\end{proof}

\begin{proof}[Proof of Theorem \ref{rotpinch}] Let $K(t)$ denote the set of curvature operators $\mathcal{R}$ satisfying $\text{tr}(\mathcal{R})\geq \frac{n-1}{2(1+t)}$, and also
$$\text{tr}(\mathcal{R})\geq -\nu(\mathcal{R})\left( \log(-\nu (\mathcal{R}))+\log(1+t)-\frac{n(n+1)}{2}\right)$$
whenever $\nu(\mathcal{R})\leq -\frac{1}{1+t}$. It is straightforward to show that this defines a closed, parallel, fiber-wise convex time-varying subset $K(t)\subseteq V$. Moreover, Lemma \ref{rotinvarode} gives 
\begin{align*}
    \frac{d}{dt}\left(\frac{R}{-\nu}-\log(-\nu)\right)\ge -\nu, \hspace{6mm} R\geq \frac{n-1}{2(1+t)},
\end{align*} which are precisely the key estimates required to show that $K(t)$ is invariant under associated ODE
$$\frac{d\mathcal{R}}{dt}=\mathcal{R}^2+\mathcal{R}^{\#}$$
(see, for example, Theorem 10.17 of \cite{ChowII}). We can thus apply Theorem \ref{hammax} to conclude the Hamilton-Ivey estimate at any point $(x,t)\in M\times [0,T)$ where $\nu(x,t)\leq -\frac{1}{1+t}$. If instead $\nu(x,t)\geq -\frac{1}{1+t}$, we then have
$$\text{tr}(\mathcal{R})\geq \frac{n(n+1)}{2}\nu(x,t) \geq -\nu(\mathcal{R})\left( \log(-\nu(\mathcal{R}))+\log(1+t)-\frac{n(n+1)}{2}\right),$$
so the claim follows in this case as well
\end{proof}

\subsection{Equivariant Cheeger-Gromov compactness}

In this section, we discuss the Cheeger-Gromov limits of rotationally-invariant Riemannian manifolds. 

Let $\theta:O(n+1)\times M\to M$ denote the isometric action
of a rotationally invariant metric, so that $\theta_{h}:=\theta(h,\cdot)$
satisfies $\theta_{h}^{\ast}g=g$ for any $h\in O(n+1).$

\begin{prop} \label{equivariantcompactness}
Suppose $(M_{i}^{n+1},g_{i},p_{i})$ is a sequence of rotationally
invariant, complete, pointed Riemannian manifolds. Let $\theta_{i}:O(n+1)\times M_{i}\to M_{i}$
be the action by isometries, and suppose 
\[
\sup_{i\in\mathbb{N}}\sup_{M_{i}}|\nabla^{k}Rm|_{g_{i}}\leq C_{k},\hfill\operatorname{Vol}_{g_{i}}(B_{g_{i}}(p_{i},1))\geq v_{0}>0.
\]
If the instrinsic diameter of $(\mathcal{O}(p_{i}),g_{i}|_{\mathcal{O}(p_{i})})$
is uniformly bounded, then after passing to a subsequence, there is
a rotationally invariant, complete pointed Riemannian manifold $(M_{\infty}^{n+1},g_{\infty},p_{\infty})$,
an exhaustion $U_{i}\subseteq M_{\infty}$ of $O(n+1)$-invariant
open sets, and $O(n+1)$-equivariant embeddings $\varphi_{i}:U_{i}\to M_{i}$
such that $\varphi_{i}^{-1}(p_{i})\to p_{\infty}$ and 
\[
\lim_{i\to\infty}\sup_{K}|\nabla_{g_{\infty}}^{k}(\varphi_{i}^{\ast}g_{i}-g_{\infty})|_{g_{\infty}}=0
\]
for all $k\in\mathbb{N}$ and every compact subset $K\subseteq M_{\infty}$. 
\end{prop}

\begin{proof}
\textbf{Step 1: Uniform bounds for the warping function away from
singular orbits. }Suppose $(M^{n+1},g,p)$ is a rotationally invariant,
complete, pointed Riemannian manifold satisfying
\[ \label{eq:orbitdiameter}
\sup_{M}|\nabla^{k}Rm|_{g}\leq C_{k},\qquad\text{diam}(\mathcal{O}(p),g|_{\mathcal{O}(p)})\leq\pi D,\qquad d_{g}(\mathcal{O}(p),M_{\text{sing}})\geq\alpha,
\]
for each $k\in \mathbb{N}$. 
Let $\psi:(-\alpha,\alpha)\to(0,\infty)$ be a warping as in Proposition \ref{rotdifgen}
such that $s(\mathcal{O}(p))=0$, $g=dr^{2}+\psi^{2}(r)\overline{g}$, and $\psi(0)\leq D$. In this step, we show the existence of constants $C(C_0,\cdots,C_k,\alpha,D,k)$ so that
\begin{align}\label{derivativeestimateofpsi}
    \left|\psi^{(k)}(s)\right|\leq C(C_{0},\cdots,C_k,\alpha,D,k)  \ \text{
for all}  \ s\in[-\frac{\alpha}{2},\frac{\alpha}{2}].
\end{align}
To find these constants, first note that
\begin{align}\label{C0CurvatureBound}
\frac{|\psi''(s)|}{\psi(s)}\leq C_{0},\quad \frac{|1-(\psi'(s))^{2}|}{\psi^{2}(s)}\leq C_{0}
\end{align}
for all $s\in(-\alpha,\alpha)$. The second of these inequalities implies that $\psi(s)\ge \min\{\frac{\alpha}{4},\frac{1}{2\sqrt{C_0}}\}$ for all $s\in [-\frac{\alpha}{2},\frac{\alpha}{2}]$. Indeed, if this were not the case, then there would be some $s_0\in [-\frac{\alpha}{2},\frac{\alpha}{2}]$ at which $\psi(s_{0})<\min\{\frac{\alpha}{4},\frac{1}{2\sqrt{C_0}}\}$;
without loss of generality, we can assume $\psi'(s_{0})\le 0$, for otherwise we may reverse the sign of $s$. We have $\psi'(s_0)<-\frac{3}{4}$ by the second inequality of \eqref{C0CurvatureBound} and the assumption that $\psi(s_0)<\frac{1}{2\sqrt{C_0}}$. Now define $$s_1=\max\left\{ s\in\left[s_{0},s_{0}+\tfrac{\alpha}{2}\right];\psi'(r)\le -\frac{1}{2}\mbox{ for all } r\in [s_0,s)\right\}.$$ 
Then, on the interval $[s_0,s_1]$ we have $\psi'\le 0$, so that 
$\psi(s)\le \psi(s_0)\leq\frac{1}{2\sqrt{C_0}}$ for all $s\in [s_0,s_1]$. Using the second inequality of \eqref{C0CurvatureBound} again, we conclude 
\begin{align*}
    \psi'(s)\le -\frac{3}{4} \ \text{for all} \ s\in [s_0,s_1].
\end{align*}
Thus, $s_1=s_0+\frac{\alpha}{2}$, and $\psi'(s)\le -\frac{1}{2}$ for all $s\in [s_0,s_0+\frac{\alpha}{2}]$ which implies that $\psi(s_0+\frac{\alpha}{2})\le \psi(s_0)-\frac{\alpha}{4}<0$, a contradiction. 
We thus obtain $\psi(s)\geq\min\{\frac{\alpha}{4},\frac{1}{2\sqrt{C_{0}}}\}$
for all $s\in[-\frac{\alpha}{2},\frac{\alpha}{2}]$. We therefore
also have 
\[
\left|\frac{d}{ds}\log\psi(s)\right|^2\leq C_{0}+\frac{1}{\psi^{2}(s)}\leq C(C_{0},\alpha),
\]
which we can integrate, using the fact that $\psi(0)\leq \frac{1}{2}D$, to obtain (\ref{derivativeestimateofpsi}) for $k=0$. Noting that 
\[
\left|\frac{d^{k}}{ds^{k}}\left(\frac{\psi''(s)}{\psi(s)}\right)\right|\leq C_{k},
\]
we obtain by induction on $k\ge 1$ that $|\psi^{(k)}(s)|\leq C(C_{0},...,C_{k},\alpha,D,k)$
for all $s\in[-\frac{\alpha}{2},\frac{\alpha}{2}]$.

\noindent \textbf{Step 2: Bounds for the metric in exponential coordinates
based at a singular orbit. }Suppose $o_{S}\in M_{i}$ is a singular
orbit, and set $\widetilde{g}_{i}:=(\exp_{o_{S}}^{g_{i}})^{\ast}g_{i}$.
Then the standard injectivity radius estimate  and Corollary 4.11 of \cite{Hamilton95} give $r_{0}=r_{0}(C_{0},v_{0})>0$
and $C_{k}'=C_{k}'(C_{0},...,C_{k},v_{0})<\infty$ such that 
\[
\frac{1}{2}g_{\mathbb{R}^{n+1}}\leq\widetilde{g}_{i}\leq2g_{\mathbb{R}^{n+1}},\qquad|(\nabla^{g_{\mathbb{R}^{n+1}}})^{k}\widetilde{g}_{i}|_{g_{\mathbb{R}^{n+1}}}\leq C_{k}
\]
on $B(0,r_{0})$ for all $k\in\mathbb{N}$. If $M_{i}$ has another
singular orbit $o_{N}$, then $\text{inj}_{g_{i}}(M_{i})\geq c(C_{0},v_{0})$
implies $d_{g_{i}}(o_{S},o_{N})<\infty$. On the other hand, we have
$\widetilde{g}_{i}=dr^{2}+\psi_{i}^{2}(r)g_{\mathbb{S}^{n}}$ in spherical
coordinates, where $\psi_{i}$ is the warping function using the distance
from $o_{S}$, so Step 1 gives, for any $L<\infty$ 
\[
c(C_{0},v_{0})\leq\psi_{i}\leq C(C_{0},v_{0},L),\qquad|\psi_{i}^{(k)}|\leq C_{k}''(C_{0},...,C_{k},v_{0},L)
\]
on $B(0,\min\{L,d_{g_{i}}(o_{N},o_{S})-r_{0}\})\setminus B(0,r_{0})$
for all $i\in\mathbb{N}$. Combining these estimates gives uniform
bounds for $\widetilde{g}_{i}$ on $B(0,\min\{L,d_{g_{i}}(o_{N},o_{S})-r_{0}\})$. $\square$

\noindent \textbf{Step 3: Compactness when $(M_{i})_{\operatorname{sing}}=\emptyset$.}
First assume $M_{i}=\mathbb{R}\times\mathbb{S}^{n}$. Letting $\psi_{i}:\mathbb{R}\to(0,\infty)$
be as in \ref{rotdifgen} with reference orbit $\mathcal{O}(p_{i})$, the bounds in Step
1 allow us to pass to a subsequence so that $\psi_{i}\to\psi_{\infty}$
in $C_{loc}^{\infty}(\mathbb{R})$, where $\psi_{\infty}\in C^{\infty}(\mathbb{R})$
is positive. Setting $g_{\infty}:=dr^{2}+\psi_{\infty}^{2}(r)g_{\mathbb{S}^{n}}$,
and letting $\varphi_{i}$ be the diffeomorphisms found in Proposition \ref{rotdifgen} so that $\mathcal{O}(p_{i})$ corresponds to $r=0$,
we have
\[
\varphi_{i}^{\ast}g_{i}=dr^{2}+\psi_{i}^{2}(r)g_{\mathbb{S}^n}\to g_{\infty}
\]
in $C_{loc}^{\infty}(\mathbb{R}\times\mathbb{S}^n)$. Because $\varphi_i$
is $O(n+1)$-equivariant, the claim follows when $M_{i}=\mathbb{R}\times\mathbb{S}^{n}$.
If $M_{i}=\mathbb{S}^{1}\times\mathbb{S}^{n}$, and if $\ell_{i}>0$
denotes the length of the $\mathbb{S}^{1}$ factor, we can assume either $\ell_i \geq 1$ or
\[
v_{0}\leq\text{Vol}_{g_{i}}(M_{i})\leq \ell_{i}C(C_{0},D)
\]
gives $\ell_{i}\geq c(v_{0},C_{0},D)$, where $D$ is as in \eqref{eq:orbitdiameter}, so we can pass to a subsequence
so that $\ell_{i}\to\ell_{\infty}\in(0,\infty]$. If $\ell_{\infty}=\infty$,
we can use the diffeomorphisms from Proposition \ref{rotdifgen} restricted to $(-\frac{-\ell}{2},\frac{\ell}{2})\times \mathbb{S}^n$ and argue as in the case $M_{i}=\mathbb{R}\times\mathbb{S}^{n}$.
Otherwise, let $\pi:\mathbb{R}\times\mathbb{S}^{n}\to([-\frac{\ell_{\infty}}{2},\frac{\ell_{\infty}}{2}]/\sim)\times\mathbb{S}^{n}$
be the quotient map, let $g_{\infty}$ be the Riemannian metric on
the image of $\pi$ induced by $dr^{2}+\psi_{\infty}^{2}(r)g_{\mathbb{S}^{n}}$. Let $N_{\zeta}\in T_{\zeta}M_i$ be the chosen unit normal vector of $\mathcal{O}(p_i)$ at $\zeta \in \mathbb{S}^n\cong \mathcal{O}(p_i)$, so that the rescaled normal exponential map $\widetilde\varphi_{i}:\mathbb{R}\times\mathbb{S}^{n}\to\mathbb{S}^{1}\times\mathbb{S}^{n},(s,\zeta)\mapsto\exp_{\mathcal{O}(p_{i})}^{g_{i}}(\frac{\ell_{i}}{\ell_{\infty}}sN_{\zeta})$
passes to the quotient to give diffeomorphisms $\varphi_{i}:\mathbb{S}^{1}\times\mathbb{S}^{n}\to\mathbb{S}^{1}\times\mathbb{S}^{n}$
satisfying 
\[
\pi^{\ast}(\varphi_{i}^{\ast}g_{i}-g_{\infty})=\left(\frac{\ell_{i}^{2}}{\ell_{\infty}^{2}}-1\right)dr^{2}+\left(\psi_{i}^{2}\left(\frac{\ell_{i}}{\ell_{\infty}}r\right)-\psi_{\infty}^{2}(r)\right)g_{\mathbb{S}^{n}},
\]
converging to 0 in $C_{loc}^{\infty}(\mathbb{R}\times\mathbb{S}^{n})$.
Thus $\varphi_{i}^{\ast}g_{i}\to g_{\infty}$ in $C^{\infty}(\mathbb{R}\times\mathbb{S}^{n})$.
In either case, we have that $\varphi_{i}^{-1}(p_{i})\in\{s_{0}\}\times\mathbb{S}^{n}$
for some $s_{0}\in\mathbb{R}$ or $s_{0}\in\mathbb{S}^{1}$, so we
can pass to a further subsequence so that $\varphi_{i}^{-1}(p_{i})\to p_{\infty}.$ 

\noindent \textbf{Step 4: Compactness when $M_{i}=\mathbb{R}^{n+1}$.
}If $d_{g_{i}}(p_{i},0^{n+1})\to\infty$, we can argue as in Step
3, so assume $\limsup_{i\to\infty}d_{g_{i}}(p_{i},0^{n+1})<\infty$.
Step 2 implies that we can pass to a subsequence so that $(\exp_{0^{n+1}}^{g_{i}})^{\ast}g_{i}\to g_{\infty}$
in $C_{loc}^{\infty}(\mathbb{R}^{n+1})$, where $g_{\infty}$ is a
smooth Riemannian metric. Then pass to a subsequence so that $(\exp_{0^{n+1}}^{g_{i}})^{-1}(p_{i})\to p_{\infty}$. 

\noindent \textbf{Step 5: Compactness when $M_{i}=\mathbb{S}^{n+1}$.
}If $\max\{d_{g_{i}}(p_{i},o_{S}),d_{g_{i}}(p_{i},o_{N})\}\to\infty$, we can argue as
in Step 3 or Step 4, so assume $\limsup_{i\to\infty}\max\{d_{g_{i}}(p_{i},o_{S}),d_{g_{i}}(p_{i},o_{N})\}<\infty$.
We can therefore pass to a subsequence so that $r_{i}:=d_{g_{i}}(o_{S},o_{N})\to r_{\infty}\in(0,\infty)$.
By Step 2, we can pass to further subsequences so that $g_{i}^{S}:=(\exp_{o_{S}}^{g_{i}})^{\ast}g_{i}\to g_{\infty}^{S}$
and $g_{i}^{N}:=(\exp_{o_{N}}^{g_{i}})^{\ast}g_{i}\to g_{\infty}^{N}$
in $C_{loc}^{\infty}(B(0^{n+1},r_{\infty}))$. We note that the transition
maps $\eta_{i}:=(\exp_{o_{N}}^{g_{i}})^{-1}\circ\exp_{o_{S}}^{g_{i}}:B(0^{n+1},r_{i})\setminus\{0\}\to B(0^{n+1},r_{i})\setminus\{0\}$
are given by $x\mapsto\left(r_{i}-|x|\right)\frac{x}{|x|}$. Now define
\[
M_{\infty}:=\left(B(0_{S}^{n+1},r_{\infty})\sqcup B(0_{N}^{n+1},r_{\infty})\right)/\sim,
\]
where we identify $x\in B(0_{S}^{n+1},r_{\infty})$ with $y\in B(0_{N}^{n+1},r_{\infty})$
if and only if $y=(r_{\infty}-|x|)\frac{x}{|x|}$. Equip $B(0_{S}^{n+1},r_{\infty})$
with the metric $g_{\infty}^{S}$, and $B(0_{N}^{n+1},r_{\infty})$
with $g_{\infty}^{N}$. Setting $\eta_{\infty}(x):=(r_{\infty}-|x|)\frac{x}{|x|}$
for $x\in B(0^{n+1},r_{\infty})\setminus\{0\}$, we have $\eta_{i}\to\eta_{\infty}$
in $C_{loc}^{\infty}(B(0^{n+1},r_{\infty})\setminus\{0\})$. Thus
$\eta_{i}^{\ast}g_{i}^{N}=g_{i}^{S}\to g_{\infty}^{S}$ and $\eta_{i}^{\ast}g_{i}^{N}\to\eta_{\infty}^{\ast}g_{\infty}^{N}$
in $C_{loc}^{\infty}$, which implies $g_{\infty}^{S}=\eta_{\infty}^{\ast}g_{\infty}^{N}$,
so that the metric on $B(0_{S}^{n+1},r_{\infty})\sqcup B(0_{N}^{n+1},r_{\infty})$
passes to the quotient to define a metric $g_{\infty}$ on $M_{\infty}$
satisfying $g_{\infty}|B(0_{S}^{n+1},r_{\infty})=g_{\infty}^{S}$
and $g_{\infty}|B(0_{N}^{n+1},r_{\infty})=g_{\infty}^{N}$. Define
$\varphi_{i}(x):=\exp_{0_{S}}^{g_{i}}(\frac{r_{i}}{r_{\infty}}x)$
for $x\in B(0_{S}^{n+1},r_{\infty})$ and $\varphi_{i}(y):=\exp_{0_{N}}^{g_{i}}(\frac{r_{i}}{r_{\infty}}y)$
for $y\in B(0_{N}^{n+1},r_{\infty})$, so that if $x\sim y$, then
\begin{align*}
\varphi_{i}(y)= & \left(\exp_{0_{N}}^{g_{i}}\circ\eta_{i}\right)(\frac{r_{i}}{r_{\infty}}y)=\exp_{0_{N}}^{g_{i}}\left(\left(r_{i}-\frac{r_{i}}{r_{\infty}}|y|\right)\frac{y}{|y|}\right)\\
= & \exp_{0_{N}}^{g_{i}}\left(\frac{r_{i}}{r_{\infty}}(r_{\infty}-|y|)\frac{y}{|y|}\right)=\exp_{0_{N}}^{g_{i}}(\frac{r_{i}}{r_{\infty}}x)=\varphi_{i}(x).
\end{align*}
By similar reasoning, we see that $x\sim y$ only if $\varphi_{i}(x)=\varphi_{i}(y)$,
so that $\varphi_{i}$ induces diffeomorphisms $\varphi_{i}:M_{\infty}\to M_{i}$
whose restriction to $B(0_{S}^{n+1},r_{\infty})$ satisfy
\[
\varphi_{i}^{\ast}g_{i}-g_{\infty}=\varphi_{i}^{\ast}g_{i}^{S}-g_{\infty}^{S}=(\tau_{r_{i}^{2}r_{\infty}^{-2}})^{\ast}g_{i}^{S}-g_{\infty}^S,
\]
where $\tau_{\lambda}:\mathbb{R}^{n+1}\to\mathbb{R}^{n+1}$ is the
dilation $x\mapsto\lambda x$. Because $\tau_{r_{i}^{2}r_{\infty}^{-2}}\to\text{id}$
in $C_{loc}^{\infty}$, we obtain $\varphi_{i}^{\ast}g_{i}-g_{\infty}\to0$
in $C_{loc}^{\infty}(B(0_{S}^{n+1},r_{\infty}))$. Arguing similarly
on $B(0_{N}^{n+1},r_{\infty})$, we have $\varphi_{i}^{\ast}g_{i}\to g_{\infty}$
in $C^{\infty}(M_{\infty})$, and the claim follows after passing
to a further subsequence so that $\varphi_{i}^{-1}(p_{i})\to p_{\infty}$. 
\end{proof}

\begin{lem} \label{flat}
Suppose we have a sequence $(M_{i}^{n},g_{i},p_{i})$ of rotationally invariant
Riemannian manifolds with $$\lim_{i}\inf_{M_{i}}Rm_{g_{i}}\geq0$$ which converges
in the $C^{\infty}$ Cheeger-Gromov sense to the complete, pointed Riemannian manifold $(M_{\infty},g_{\infty},p_{\infty})$, such that the intrinsic diameters of $(\mathcal{O}(p_{i}),g|_{\mathcal{O}(p_i)})$ approach $\infty$ as $i\to\infty$. Then
$g_{\infty}$ is flat and  $M_{\infty}=\mathbb{R}^{n}$. 
\end{lem}

\begin{proof}
By assumption, we have an increasing exhaustion of $M_{\infty}$ by relatively compact open sets $\{U_{i}\}_{i\in \mathbb{N}}$, and embeddings $\varphi_{i}:U_{i}\to M_{i}$
satisfying $\varphi_{i}(p_{\infty})=p_{i}$ and
\[
\lim_{i\to\infty}\sup_{K}|\nabla_{g_{\infty}}^{k}(\varphi_{i}^{\ast}g_{i}-g_{\infty})|_{g_{\infty}}=0
\]
for any $k\in\mathbb{N}$ and any compact subset $K\subseteq M_{\infty}$.
Using Proposition \ref{rotdifgen} based at $\mathcal{O}(p_{i})$, we can identify $(M_{i})_{\operatorname{reg}}\cong(\alpha_{i},\beta_{i})\times\mathbb{S}^{n}$,
where $-\infty\leq\alpha_{i}<0<\beta_{i}\leq\infty$, and $\mathcal{O}(p_i)$ corresponds to $\{0\}\times \mathbb{S}^n$. With this identification, we have $g_{i}=dr^{2}+\psi_{i}^{2}(r)g_{\mathbb{S}^{n}},$ where
$\psi_{i}\in C^{\infty}((\alpha_{i},\beta_{i}),(0,\infty))$. Recall
that the sectional curvatures on $(M_{i})_{\operatorname{reg}}$ are given by
\[
\dfrac{1-(\psi'_{i})^{2}}{\psi_{i}^{2}},\quad -\dfrac{\psi_{i}''}{\psi_{i}},
\]
so there exists a sequence of positive numbers $\{\epsilon_i\}_{i\in \mathbb{N}}$ converging to $0$ so that $1-(\psi_{i}')^{2}\geq-\epsilon_{i}\psi_{i}^{2}$ and $\psi_{i}''\leq\epsilon_i\psi_{i}$
on $(\alpha_{i},\beta_{i})$. We can assume without loss of generality that $\epsilon_i\le \frac{1}{2}$ for all $i\in \mathbb{N}$. 

We claim that\textbf{ $\lim_{i\to\infty}\alpha_{i}=-\infty$ } and\textbf{
$\lim_{i\to\infty}\beta_{i}=\infty$} , and that for each large $K>0$, the quantity $\xi_i(r):=\log\left(\frac{\psi_i(r)}{\psi_i(0)}\right)$ is uniformly bounded for $r\in [-K,K]$. To see this, observe that we have the inequality
\[
\left|\dfrac{d}{dr}\log\left(1+\psi_{i}(r)\right)\right|\leq\dfrac{|\psi_{i}'(r)|}{1+\psi_{i}(r)}\leq\dfrac{\sqrt{1+\epsilon_{i}\psi_{i}^{2}(t)}}{1+\psi_{i}(r)}\leq 1
\]
for all $r\in(\alpha_{i},\beta_{i})$; this allows us to conclude 
\begin{align}\label{somenonsense}
\left|\log\left(\frac{1+\psi_i(r)}{1+\psi_i(0)}\right)\right|\le \left|r\right|
\end{align}
for each $r\in (\alpha_i,\beta_i)$. If $\alpha_i$ were bounded then $\lim_{i\to \infty}\psi_i(0)=\infty$, $\lim_{t\to \alpha_i^-}\psi_i(t)=0$, and (\ref{somenonsense}) combine to give a contradiction, hence $-\alpha_i$ and $\beta_i$ must become arbitrarily large. By the fact $\lim_{i\to \infty}\psi_i(0)=\infty$ again, and since (\ref{somenonsense}) implies $\xi_i(r)$ is uniformly bounded on compact subsets, we get $\psi_i\to\infty$ uniformly on compact sets.

Since
\[
-\epsilon_{i}\leq\dfrac{1-(\psi'_{i})^{2}}{\psi_{i}^{2}}\leq\frac{1}{\psi_{i}^{2}},
\]
the orbital sectional curvatures converge locally uniformly to 0.
Since $\frac{|\psi_{i}'|}{\psi_i}\leq \sqrt{\epsilon_{i}+\psi_i^{-2}}$, and since $\psi_i\to\infty$ uniformly on compact sets, we have $|\xi_{i}'|\to0$ locally
uniformly in $C_{loc}^{0}$. Moreover, $\sup_{r\in[-\frac{1}{2}\delta^{-1},\frac{1}{2}\delta^{-1}]}|\xi_{i}''(r)|$ can be bounded in terms of $\sup_{i}\sup_{B_{g_{i}}(p_{i},\delta^{-1})}|Rm_{g_i}|_{g_{i}}$, and is thus bounded uniformly.
Similarly, $\sup_{r\in[-\frac{1}{2}\delta^{-1},\frac{1}{2}\delta^{-1}]}|\xi_{i}'''(r)|$ also admits a uniform bound. 
Indeed, 
\begin{eqnarray*}
\xi_i''&=&\frac{\psi_i''}{\psi_i}-\frac{(\psi_i')^2}{\psi_i^2}\\
\xi_i'''&=&\frac{d}{dr}\left(\frac{\psi_i''}{\psi_i}\right)+\frac{d}{dr}\left(-\frac{(\psi_i')^2}{\psi_i^2}\right)\\
&=&\frac{d}{dr}\left(\frac{\psi_i''}{\psi_i}\right)+\frac{d}{dr}\left(\frac{1-(\psi_i')^2}{\psi_i^2}\right)-\frac{2\psi_i'}{\psi_i}\cdot\frac{1}{\psi_i^2},
\end{eqnarray*}
where the first two terms on the last line is bounded by $|\nabla Rm_{g_i}|_{g_i}$ and the last term approaches $0$ locally uniformly by our previous argument.
Thus, for any subsequence, we may pass to a further subsequence so that $\xi_{i}$ converges in $C_{loc}^{2}(\mathbb{R})$
to some $\xi_{\infty}\in C^{2}(\mathbb{R})$. However, $|\xi_{i}'|\to0$
in $C_{loc}^{0}(\mathbb{R})$ implies that $\xi_{\infty}$ is constant,
hence $\xi''_{\infty}\equiv0$. Since the subsequence was arbitrary,
we conclude that $\lim_{i\to\infty}\sup_{r\in[-\delta^{-1},\delta^{-1}]}|\xi_{i}''(r)|=0$
for any $\delta>0$. It follows from straightforward computation that $|Rm_{g_i}|_{g_{i}}\to0$ in $C_{loc}^{0}$.
Since the limit must be noncompact, we conclude that it is flat $\mathbb{R}^{n+1}$; to see this, recall that all rotationally invariant manifolds are $\kappa$-noncollasped with a universal $\kappa(n)$ (Proposition \ref{uniform_noncollapsing}), and hence the limit has positive asymptotic volume ratio. But a flat manifold with positive asymptotic volume ratio must have finite fundamental group, and hence cannot be a nontrivial quotient of $\mathbb{R}^{n+1}$.
\end{proof}

We now extend this compactness theorem to the setting of Ricci flows.

\begin{prop} \label{flowequivcompact} Suppose that $(M_i^{n+1},(g_t^i)_{t\in (\alpha ,\beta]},p_i)$ is a sequence of complete, pointed, rotationally invariant Ricci flows, with $\theta_i:O(n+1)\times M_i \to M_i$ the isometric action. Assume that $\alpha<0 <\beta$, $\sup_i \sup_{M_i \times (\alpha,\beta]}|Rm_{g^i}|_{g^i}\leq C<\infty$ and $\inf_i \operatorname{Vol}_{g_i}(B_{g_0^i}(p_i,1))\geq v_0$ for all $i\in\mathbb{N}$. If the intrinsic diameter of $\mathcal{O}(p_i)$ equipped with $g_0^i$ is uniformly bounded, then there is a complete, pointed, rotationally invariant Ricci flow $(M^{n+1}_{\infty},(g_t^{\infty})_{t\in (\alpha,\beta]},p_\infty)$, an exhaustion $U_i \subseteq M_{\infty}$ of $O(n+1)$-invariant precompact open sets, and embeddings $\varphi_i:U_i \to M_i$ such that $\varphi_i(p_{\infty})=p_i$, $\varphi_i \circ \theta_h^{\infty}=\theta_h^i\circ \varphi_i$ for all $h\in O(n+1)$, and 
$$\lim_{i\to \infty} \sup_{t\in [\alpha',\beta]} \sup_K |\nabla_{g_t^i}(\varphi_i^{\ast}g_t^i-g_t^{\infty})|_{g_t^{\infty}}=0$$
for each $k\in \mathbb{N}$, $\alpha'\in (\alpha,\beta)$ and compact set $K\subseteq M_{\infty}$. \end{prop}
\begin{proof} This follows from the classical fact (see \cite{Hamilton95}) that, given the assumed curvature bounds, convergence at one time implies convergence at all times.
\end{proof}

\begin{cor} \label{closenessimproved}
For any $\epsilon>0$ and $\kappa>0$, there exists $\delta=\delta(n,\epsilon,\kappa)>0$
such that the following holds. Suppose $(M^{n+1},(g_t)_{t\in[0,T)},(x_{0},t_{0}))$
is a pointed, complete rotationally invariant solution of Ricci flow
with bounded curvature, and that 
\[
\{(x,t)\in M\times [0,T);d_{g_t}(x,x_{0})<Q^{-\frac{1}{2}}\delta^{-\frac{1}{2}},t\in[t_{0}-Q^{-1}\delta^{-1},t_{0}]\}
\]
equipped with the rescaled metrics $Qg_t$, where $Q:=R(x_{0},t_{0})$,
is $\delta$-close to the corresponding subset of a $\kappa$-solution.
Then there exists a rotationally invariant $\kappa$-solution $(\widehat{M},\widehat{g})$
and an $O(n+1)$-equivariant diffeomorphism $\eta$ from
\[
\{(x,t)\in M;d_{g_t}(x,\mathcal{O}(x_{0}))<Q^{-\frac{1}{2}}\epsilon^{-\frac{1}{2}},t\in[t_{0}-Q^{-1}\delta^{-1},t_{0}]\}
\]
to the corresponding subset $U$ of $(\widehat{M},\widehat{g})$,
such that $||(\eta^{-1})^{\ast}(Qg_t)-\widehat{g}||_{C^{\left\lfloor \epsilon^{-1}\right\rfloor }(U,g^{\infty})}<\epsilon$.
\end{cor}

\begin{proof}
Suppose $(M_{i},(g_t^i)_{t\in[0,T)},(x_i,t_i))$ is a contradictory
sequence, so that if $Q_{i}:=R_{g^i}(x_{i},t_{i})$, then\\ $(M_{i},Q_{i}g_{Q_i^{-1}t}^i,(x_{i},t_{i}))$
converge in the pointed Cheeger-Gromov sense to a $\kappa$-solution
$(M_{\infty},g_t^{\infty},(x_{\infty},0))$ with $R_{g^{\infty}}(x_{\infty},0)=1$.
By Lemma \ref{flat}, the intrinsic diameters of $(\mathcal{O}(x_{i}),Q_i g_{t_i}^i)$ must be uniformly bounded, so Proposition \ref{flowequivcompact} gives equivariant convergence to
a rotationally invariant $\kappa$-solution, a contradiction.
\end{proof}

\subsection{Geometry of rotationally invariant $\kappa$-solutions}

In this section, we improve the description of high curvature regions of rotationally invariant Ricci flows, using the equivariant compactness theorem and various results from \cite{LiZhang18}.

\begin{defin}
A $\kappa$-solution of Ricci flow is an ancient solution $(M^{n},(g_t)_{t\in(-\infty,0]})$
of Ricci flow with uniformly bounded, nonnegative curvature operator,
that is $\kappa$-noncollapsed on all scales.
\end{defin}

In the rotationally-invariant setting, we know exactly what the relevant non-compact $\kappa$-solutions look like:
\begin{thm}[Li-Zhang \cite{LiZhang18}]\label{kappa-solution-classification-noncompact}
A non-negatively curved and non-compact rotationally invariant $\kappa$-solution is (isometric to) the Bryant soliton, or a quotient of the shrinking cylinders. 
\end{thm}
In the compact setting, a recent work shows the following:

\begin{thm}[Brendle-Daskalopolous-Naff-Sesum \cite{Brendle-etal21}]\label{kappa-solution-classification-compact}
Let $(\mathbb{S}^{n+1},(g_t^1)_{t\in(-\infty,0]})$ and $(\mathbb{S}^{n+1},(g_t^2)_{t\in (-\infty,0]})$ be two ancient $\kappa$-solutions on $\mathbb{S}^{n+1}$ which are rotationally invariant. Assume that neither $(\mathbb{S}^{n+1},(g_t^1)_{t\in (-\infty,0]})$ nor $(\mathbb{S}^{n+1},(g_t^2)_{t\in (-\infty,0]})$ is a family of shrinking round spheres. Then $(\mathbb{S}^{n+1},(g_t^1)_{t\in (-\infty,0]})$ and $(\mathbb{S}^{n+1},(g_t^2)_{t\in (-\infty,0]})$ coincide up
to a (time-independent) diffeomorphism, a translation in time, and a parabolic rescaling.
\end{thm}

In any case, we have the following useful observations about $\kappa$ solutions.
\begin{thm}\label{universal kappa}
(Universal $\kappa_{0}$) For any $n\in\mathbb{N}$, there exists
$\kappa_{0}=\kappa_{0}(n)$ such that any rotationally invariant $\kappa$-solution
is $\kappa_{0}$-noncollapsed on all scales.
\end{thm}

\begin{proof}
The proof is almost verbatim the 3-dimensional case. The necessary
ingredient is that all asymptotic shrinkers are either the round sphere
or the cylinder. Also, rotational symmetry rules out quotients of
the round sphere, so it is no longer necessary to exempt round space
forms.
\end{proof}

\begin{thm}
The set of pointed and rotationally-invariant $\kappa_0$-solutions $(M^{n+1},(g_t)_{t\in (-\infty,0]},p)$ with $R(p,0)=1$ is compact in the pointed equivariant Cheeger-Gromov-Hamilton topology. 
\end{thm}

\begin{defin}[Equivariant Neck]
Let $g_t^{cyl}:=dr^2 + (1-t)g_{\mathbb{S}^n}$, $t\in (-\infty,1)$ denote the standard shrinking cylinder, normalized so that the singular time is $t=1$. We say that a rotationally-invariant open subset $U$ of a rotationally-invariant Riemannian manifold $(M,g)$ is an equivariant $\epsilon$-neck if there is an $O(n+1)$-equivariant diffeomorphism $\Phi:(-\frac{1}{\epsilon},\frac{1}{\epsilon})\times \mathbb{S}^{n}\to M$ so that $|\lambda \Phi^{\ast}g-{g}_0^{cyl}|_{C^{\lfloor \epsilon^{-1}\rfloor}((-\frac{1}{\epsilon},\frac{1}{\epsilon})\times \mathbb{S}^{n})}$ for some $\lambda>0$, where $g_{cyl}$ is the standard cylinder metric with scalar curvature 1. A rotationally invariant subset $U\times [a,b]  \subseteq M\times I$ of a Ricci flow is called a strong equivariant $\epsilon$-neck if there is a diffeomorphism $\Phi:\mathbb{S}^n\times (-\frac{1}{\epsilon},\frac{1}{\epsilon})\to M$ such that $|\lambda\Phi^{\ast}g-g^{cyl}|_{C^{\lfloor \epsilon^{-1}\rfloor}((-\frac{1}{\epsilon},\frac{1}{\epsilon})\times \mathbb{S}^{n}\times [-1,0])} $
We call any point in $\Phi(\mathbb{S}^n\times \{0\})$ a center of the neck $U$. 
\end{defin}

\begin{defin}[Equivariant Cap]
Given a rotationally invariant Riemannian manifold $(M,g)$, we say a point $x\in M$ is the center of an equivariant $(E,\epsilon)$-cap $U$ if  $U\subseteq M$ is  a rotationally invariant open subset such that the following hold:\\
$(i)$ $U$  is equivariantly diffeomorphic to $B(0,1)\subseteq \mathbb{R}^n$,\\
$(ii)$ there is a compact set $K\subset U$ so that $U\setminus K$ is an equivariant $\epsilon$-neck and $x\in K$,\\
$(iii)$ $\operatorname{diam}_g(U) \leq E R^{-\frac{1}{2}}(x)$, \\
$(iv)$ $\operatorname{Vol}_g(U)\geq E^{-1}R^{-\frac{n+1}{2}}(x)$,\\
$(v)$ $0<E^{-1}R(y)\leq R(x)\leq E R(y)$ for all $x,y\in U$.
\end{defin}


Let $(M,(g_t)_{t\in(-\infty,0]})$ be a rotationally invariant $\kappa$-solution and let $M_{\epsilon}\subseteq M$ be the points in $(M,g_0)$ which
are not the centers of strong $\epsilon$-necks. A necessary ingredient
for the canonical neighborhood theorem is the following, which roughly
implies that, if 3 distant points in a $\kappa$-solution are almost
colinear, then the middle point is in an $\epsilon$-neck. 
\begin{lem} \label{kaptechlemma}
For any $\theta\in(0,\pi],$and any $\epsilon>0$, there exists $C_{\kappa}(\theta,\epsilon)<\infty$
such that the following holds. If $p,q,q'\in M$ satisfy $\widetilde{\angle}(qpq')\geq\theta$
and 
\[
d_{g(0)}^{2}(p,q)R(p,0),\quad d_{g(0)}^{2}(p,q')R(p,0)\geq C,
\]
then $p\not\in M_{\epsilon}.$ Here $\widetilde{\angle}$ stands for the comparison angle.
\end{lem}
\begin{proof}  We use Theorem 1.1 of \cite{LiZhang18} to justify that (after rescaling) any pointed Cheeger-Gromov limit of a fixed $\kappa$-solution with basepoints at a fixed time and approaching spatial infinity must be the shrinking cylinder. Given this, the claim is established by following the proof of Proposition 49.1 in \cite{KleinerLott18}. 
\end{proof}

\begin{prop} \label{kapsols} For any $\epsilon>0$, there exists $E=E(\epsilon,n)<\infty$ such that for any rotationally invariant $\kappa$-solution $(M,g_t)_{t\in (-\infty,0]}$ and any $x\in M$, one of the following is true:\\
$(i)$ $x$ is the center of a strong equivariant $\epsilon$-neck,\\
$(ii)$ $x$ is the center of an equivariant $(E,\epsilon)$-cap,\\
$(iii)$ $M \cong \mathbb{S}^n$, and $\mbox{diam}_{g_0}(M)\leq E R^{-\frac{1}{2}}(x,0)$.
\end{prop}
\begin{proof} After rescaling, we can assume $R(x,0)=1$. By Theorem 1.4 of \cite{LiZhang18} and Proposition \ref{flowequivcompact}, we have $\text{diam}_{g_0}(M_\epsilon)\leq C(n,\epsilon)$. If $x\notin M_{\epsilon}$, then $(i)$ holds, so assume $x\in M_{\epsilon}$ and $\text{diam}_{g_0}(M)\geq 2C(n,\epsilon)$. Then there is a point $y \in M\setminus M_{\epsilon}$ with $d_{g_0}(x,y)\leq 2C(n,\epsilon)+2\epsilon^{-1}$; let $N \subseteq M$ be the corresponding equivariant $\epsilon$-neck. Because $M\cong \mathbb{S}^{n+1}$ or $M\cong \mathbb{R}^n$, we know that $N$ separates $M$ into exactly two components; let $U$ be the union of $N$ with the component containing $x$. Then $U$ is an equivariant cap, and $x$ is a center of $U$.
\end{proof}

\subsection{The canonical neighborhood theorem}

We now justify the fact that Perelman's canonical neighborhood theorem (the version stated as Theorem 52.7 in \cite{KleinerLott08}, with a slightly different statement originally stated as Theorem 12.1 in \cite{Perelman02}) also holds in the higher-dimensional rotationally invariant setting.

\begin{thm}
Given $\epsilon,\kappa,\sigma>0$, there exists $r_{0}=r_{0}(\epsilon,\kappa,\sigma,n)>0$
such that the following holds. Let $(M^{n+1},(g_t)_{t\in[0,T]})$ be
a closed $O(n+1)$-invariant Ricci flow, that is $\kappa$-noncollapsed
below the scale $\sigma$, such that $T\geq1$ and $(M,g_0)$ satisfies $\mathcal{R}_{g_0}\geq -1$. Then, for any $(x_{0},t_{0})\in M\times[1,T]$ with $Q:=R(x_{0},t_{0})\geq r_{0}^{-2}$,
$P(x_{0},t_{0},(\epsilon Q)^{-\frac{1}{2}},-(\epsilon Q)^{-1})$ is, after
rescaling by $Q$, equivariantly $\epsilon$-close to the corresponding subset of
a $\kappa$-solution.
\end{thm}
\begin{proof}
The argument is again almost identical to the three-dimensional (non-rotationally invariant) case. For the reader's convenience, we list below the facts about three-dimensional Ricci flow (which fail in general for higher dimensions) which are used in the proof.\\
$(1)$ The Hamilton-Ivey pinching condition (Lemma B.6 of \cite{KleinerLott08}),\\
$(2)$ The Cheeger-Gromov compactness of three-dimensional $\kappa$-solutions (Theorem 46.1 and Lemma 59.4 of \cite{KleinerLott08}),\\
$(3)$ The canonical neighborhood property of three-dimensional $\kappa$-solutions (Lemma 59.4 in \cite{KleinerLott08}),\\
$(4)$ If two minimizing segments with far-away endpoints form an angle bounded away from zero, then the point of intersection is the center of an $\epsilon$-neck (Proposition 49.1 in \cite{KleinerLott08}).\\
We observe that each of these facts is true for rotationally-invariant Ricci flow in any dimension: $(1)$ is Lemma \ref{rotpinch} (with the constant $\frac{n(n-1)}{2}$ instead of 3), $(2)$ is Theorem 1.2 of \cite{LiZhang18}, $(3)$ is Proposition \ref{kapsols}, and $(4)$ is Lemma \ref{kaptechlemma}.
\end{proof}

\section{Rotationally-invariant Ricci flow through singularities}

\subsection{Ricci flow with surgery}

In this subsection, we review the notion of standard solutions, and the classification of ends of the incomplete Ricci flow time slice at the first singular time. Then we give a simpler construction of the surgery process, which uses the warping function induced from an equivariant $\delta$-neck structure rather than a conformal change, and preserves the $O(n+1)$-symmetry. Finally, we explain how the proof of long-time existence of Ricci flow with $(\delta,r)$-cutoff simplifies as a result of the automatic $\kappa$-noncollapsedness of rotationally invariant metrics. 

\begin{thm} [c.f. Section 60 of \cite{KleinerLott08}] Suppose that $g$ is a rotationally invariant complete Riemannian metric on $\mathbb{R}^{n+1}$ with nonnegative curvature, whose scalar curvature satisfies $R\geq a$ everywhere for some universal constant $a>0$, and which is equal to a round cylindrical metric with scalar curvature $1$ on the complement of some compact subset. Then there is a unique complete Ricci flow with bounded curvature $(\mathbb{R}^{n+1},(g_t)_{t\in [0,\frac{n}{2})})$ such that $g_0=g$. Moreover, this flow satisfies the following:\\
$(i)$ there exists $\kappa_0>0$ such that $(\mathbb{R}^{n+1},(g_t)_{t\in[0,\frac{n}{2})})$ is $\kappa$-noncollapsed on scales less than 1;\\
$(ii)$ $Rm(g_t)\geq 0$ for all $t\in [0,\frac{n}{2})$;\\
$(iii)$ $g_t$ is rotationally invariant for all $t\in [0,\frac{n}{2})$;\\
$(iv)$ for any $\tau_0\in (0,1)$ and $\epsilon>0$, there exists $r_0=r_0(\tau_0,\epsilon)>0$ such that for any $(x_0,t_0)\in \mathbb{R}^{n+1}\times[\tau_0,\frac{n}{2})$ with $Q:=R(x_0,t_0)\geq r_0^{-2}$, the solution in
$$\{(x,t)\in \mathbb{R}^{n+1}\times [t_0-(Q\epsilon)^{-1},t_0];d_{g_{t_0}}(x,x_0)<(\epsilon Q)^{-\frac{1}{2}} \}$$
is equivariantly-$\epsilon$-close to the corresponding subset of a $\kappa$-solution;\\
$(v)$ $\lim_{t\to \frac{n}{2}}\liminf_{x\in \mathbb{R}^{n+1}}R(x,t)=\infty$.\\
We call this Ricci flow solution the standard Ricci flow corresponding to $g$.
\end{thm}

We now explain the surgery procedure, mostly following \cite{BamlerNotes}. We indicate several simplifications in the rotationally invariant setting.\\

\begin{defin} If $(M,(g_t)_{t\in I})$ is a rotationally invariant Ricci flow, $r,\epsilon>0$, $E<\infty$, then a point $(x,t)\in M\times I$ satisfies assumption $CNA(r,\epsilon,E)$ if either $R(x,t)\leq r^{-2}$ or the following hold:\\
$(i)$ $|\nabla R^{-\frac{1}{2}}(x,t)|+|\partial _t R^{-1}(x,t)|\leq C$,\\
$(ii)$ $(x,t)$ is the center of a strong equivariant $\epsilon$-neck or an equivariant $(E,\epsilon)$-cap.
\end{defin}

Let $(M,g)=(M^{1},g_{0}^{1})$ be a closed, rotationally invariant
Riemannian manifold with $R<1$, which is $\kappa_0(n)$-noncollapsed on scales
less than 1, and $|\mathcal{R}|_{g_0} \leq 1$. A manifold satisfying these conditions is said to have \emph{normalized geometry}.
Any closed rotationally invariant Riemannian manifold can be rescaled
to have normalized geometry. The curvature bounds ensure that
the first singular time $T_{1}\in(0,\infty]$ satisfies $T_{1}>\frac{1}{8}$. 

Fix a small constant $\epsilon>0$. Suppose by induction that $\mathcal{M}=\left(M^j ,(g_t^j)_{t\in [T_{j-1},T_j)}\right)_{j=1}^k$ is a Ricci flow with surgery (for a precise definition, see Appendix A.9 of \cite{KleinerLott17}) satisfying the Hamilton-Ivey pinching condition and assumption $CNA(r,\epsilon,E)$ at all points (except possibly those removed by surgery), and suppose $\delta>0$. Let $$\Omega^{k}:=\left\{ x\in M^k ; \limsup_{t\nearrow T_k} R(x,t)<\infty \right\}  $$
be the regular set, an open subset with a limiting (incomplete) metric
$g_{T_k}$. Note that
$$\Omega_{0}^{k}:=\left\{ x\in\Omega^{k};R(x,T_k)<\frac{1}{(\delta r)^2}\right\} $$
is relatively compact in $\Omega^{k}$. Since $\partial_{t}(R^{-1})$ is uniformly bounded, $R(\cdot,t)\to\infty$ uniformly on $M^{k}\setminus\Omega^{k}$.

Note that $\Omega^k \setminus \Omega_0^k$ is covered by disjoint open sets $N_i$, each equivariantly diffeomorphic to $\mathbb{B}^{n+1}$, $\mathbb{S}^{n+1}$, $\mathbb{S}^1 \times \mathbb{S}^{n}$, or $\mathbb{R}\times \mathbb{S}^n$, and they are each themselves covered by strong equivariant $\epsilon$-necks, $(C_1,\epsilon)$-caps, and $(C,\epsilon)$-components. Let $\widetilde{\Omega}^k\subseteq \Omega^k$ denote the union of all connected components of $\Omega^k$ which are not completely covered by some of the $N_i$. Observe that every end $N_i$ of $\widetilde{\Omega}^k$ is diffeomorphic to $\mathbb{R}\times \mathbb{S}^n$, is covered by equivariant $\epsilon$-necks, and $R(\cdot,T_k)\to \infty$ along this end; such an end is called an equivariant $\epsilon$-horn. Also note that there are only finitely many such $N_i$, since any $\epsilon$-neck intersecting $\Omega_0^k$ must have some definite volume.

\begin{lem} \label{epshorn} There exists $\epsilon_0>0$ such that the following statement holds for any $\epsilon \in (0,\epsilon_0)$. For any $\delta>0$, there exists $h=h(\delta)\in (0,1)$ such that if $r>0$ and assumption $CNA(r,\epsilon,E)$ holds for all non-presurgery points, then for every $\epsilon$-horn $N_i$ and every point $x\in N_i$ with $R(x,T_k)\geq \frac{1}{(hr)^2}$, $x$ is the center of a strong, equivariant $\delta$-neck at time $T_k$.
\end{lem}

The proof is the same as the three-dimensional case.

Now fix $\delta>0$ to be determined below, and let $g_{\mathbb{S}^n}$ denote the standard round metric on the sphere with scalar curvature 1. By the lemma and previous discussion, if $N_i$ is any $\epsilon$-horn, and we choose any point $x_i \in N_i$ with $R(x_i,T)= \frac{1}{(hr)^2}$, then $x_i$ is the center of an equivariant $\delta$-neck $\mathcal{N}$ at time $T$. Let
$$\Phi: (-\delta^{-1},\delta^{-1})\times \mathbb{S}^n \to \mathcal{N}$$
be the corresponding $O(n+1)$-equivariant diffeomorphism, and write $\overline{g}:=\frac{1}{n(n-1)(rh)^2}\Phi^{\ast}g(t)$.
In particular,  
\[
\overline{g}-g_{cyl}=(\varphi^{2}(r)-1)dr^{2}+(\psi^{2}(r)-1)g_{\mathbb{S}^{n}},
\]
and for any $k\ge1$, 
\[
(\nabla^{g_{cyl}})^{k}(\overline{g}-g_{cyl})=\partial_{r}^{k}(\varphi^{2}(r))dr^{k+2}+\partial_{r}^{k}(\psi^{2}(r))dr^{k}\otimes g_{\mathbb{S}^{n}}.
\]
We can therefore conclude that 
\begin{equation}\label{neckmetriccomparison}
  ||\varphi-1||_{C^{k}((-\delta^{-1},\delta^{-1}))}+||\psi-1||_{C^{k}((-\delta^{-1},\delta^{-1}))}<\Psi(\delta,k),  
\end{equation}
where $\Psi(\delta |k)\to0$ as $\delta\to0$ for any fixed $k\in\mathbb{N}$. By possibly modifying $\delta$, We can replace $\Phi$ with the reparametrization
\[
\left(-\tfrac{1}{2}\delta^{-1},\tfrac{1}{2}\delta^{-1}\right)\times\mathbb{S}^{n}\to \mathcal{N},\quad (s,\zeta)\mapsto\Phi (\widetilde{s}^{-1}(s),\zeta),
\]
where $\widetilde{s}(r):=\int_0^r \varphi(\tau)d\tau$ is parametrization by arclength, to assume that actually $\varphi \equiv 1$. This is easily seen using the chain rule and the following consequence of (\ref{neckmetriccomparison}):
$$|\partial_r^k (s(r)-r)|\leq \Psi(\delta\,|\,k,\delta')\quad\text{ for all }\quad s\in (-\delta'^{-1},\delta'^{-1}),$$
where $\Psi(\delta\,|\,k,\delta')\to 0$ if $\delta\to 0$ while $k$ and $\delta'$ are kept fixed.

Now fix a partition of unity $(\phi _1,\phi _2)$ of $\mathbb{R}$ subordinate to the cover $(-\infty,\frac{1}{2})\cup (\frac{1}{4},\infty)$.

\begin{lem} \label{surger} There exist $D\in (\frac{5}{8},\infty)$, $\sigma, \delta_0 >0$ and $\widehat{\eta} \in C^{\infty} ((-\infty,D])$ such that the following hold:
\begin{enumerate}[(i)]
\item $\widehat{\eta}(r)=1$ for all $r\leq 0$;
\item for any $\delta\in (0,\delta_0)$ and $\psi\in C^{\infty}((-\delta^{-1},\delta^{-1}))$ satisfying 
$$||\psi - 1||_{C^k((-\delta^{-1},\delta))}<\delta,$$the rotationally invariant metric $\widetilde{g}=dr^2 + \widetilde{\psi}^2(r)g_{\mathbb{S}^n}$with warping function $$\widetilde{\psi}(r):= \big( \phi_1(r)\psi(r)+\phi_2(r) \big) \widehat{\eta}(r)$$satisfies $\widetilde{K}_{rad}(r)\geq K_{rad}(r)$, $\widetilde{K}_{orb}(r)>0$, $\widetilde{R}(r) \geq R(r)$ for all $r\in [0,\frac{1}{4}]$, and $\widetilde{K}_{rad}(r),\widetilde{K}_{orb}(r) \geq \sigma$ for all $r\in [\frac{1}{4},D]$; moreover, $\widetilde{g}$ extends to a smooth Riemannian metric on $((-\delta^{-1},D)\times\mathbb{S}^n)\cup \{0^{n+1}\} \cong \mathbb{B}^{n+1}$;
\item there exists $k>0$ such  that $\widehat{\eta}(r)=\frac{1}{\sqrt{k}}\sin \left( \sqrt{k}(D-r) \right)$ for all $r\in [D-\frac{1}{8},D]$;
\item the metric $g_{\operatorname{stan}}:=dr^2 + \widehat{\eta}^2(r)g_{\mathbb{S}^{n-1}}$ extends to a complete rotationally invariant metric on $\mathbb{R}^{n+1}$ with bounded nonnegative curvature, and $\inf_{\mathbb{R}^{n+1}} R>0$; we call $(\mathbb{R}^{n+1},g_{\operatorname{stan}})$ the standard metric;
\item for any $k\in \mathbb{N}$, we have
$$|\nabla^k (g_{\operatorname{stan}}-\widetilde{g})|_{g_{\operatorname{stan}}} \leq \Psi(\delta|k)$$ on the entire cap  $((-\delta,D)\times\mathbb{S}^n)\cup \{0^{n+1}\} \cong \mathbb{B}^{n+1}$.
\end{enumerate}
\end{lem}
\begin{proof} Define 
$$\eta(r):= \left\{ \begin{array}{cc}
    1, & r\in (-\infty,0], \\
    1-ae^{-b/r}, & r\in [0,1],
\end{array} \right.$$
where $0<a<\frac{1}{2}<b<\infty$ are constants to be determined, so that $\frac{1}{2}\leq \eta(r)\leq 1$ for all $r\in [0,1]$. For any $\psi$ satisfying the hypotheses of $(ii)$, we compute the curvature of $dr^2+(\eta(r)\psi(r))^2g_{\mathbb S^n}$:
\begin{align*} \widetilde{K}_{rad}(r) &= K_{rad}(r) - \frac{\eta''(r)}{\eta(r)} - 2\frac{\psi'(r)}{\psi(r)}\frac{\eta'(r)}{\eta(r)} \\ &\geq K_{rad}(r)+ \frac{ab}{2r^4}(b-2r-8\delta r^2)e^{-b/r},  
\end{align*}
\begin{align*}
    \widetilde{K}_{orb}(r) &= \frac{1}{\psi^2(r)}\left(\frac{1}{\eta^2(r)}-1 \right) +K_{orb}(r) +2\frac{\psi'(r)}{\psi(r)}\frac{\eta'(r)}{\eta(r)} - \left( \frac{\eta'(r)}{\eta(r)} \right)^2 \\ &\geq K_{orb}(r) - 4\delta \cdot \frac{ab}{r^2}e^{-b/r}-4 \left( \frac{ab}{r^2}e^{-b/r} \right)^2. 
\end{align*}
Combining the expressions above, we have
\begin{align*} \widetilde{R}(r)&= 2n\widetilde{K}_{rad}(r) +n(n-1)\widetilde{K}_{orb}(r)
\\&= R(r)+ \frac{nab}{2r^4}\left( (2-8na)b - 4r -16n\delta r^2 \right) e^{-b/r}.
\end{align*}
By choosing $\delta_0 =\delta_0 (b)>0$ sufficiently small, and by letting $b=100n$ and $a=\frac{1}{(100n)^2}$, we see that  $\widetilde{K}_{rad}(r)\geq K_{rad}(r)$, $\widetilde{K}_{orb}(r)>0$, $\widetilde{R}(r) \geq R(r)$ for all $r\in [0,\frac{1}{2}]$. 

Let $\chi\in C_{c}^{\infty}([0,\frac{1}{2}))$ be a cutoff function
satisfying $\chi'\leq0$ and $\chi|_{[0,\frac{1}{4}]}\equiv1$, and define $\widehat\zeta:(-\infty,\tfrac{5}{8}]\rightarrow \mathbb R$ with 
\begin{gather*}
  \widehat{\zeta}(r):=\left\{\begin{array}{cll}
    \eta'(r)&\text{if} &r\in(-\infty,\tfrac{1}{4}],
  \\
  \chi(r)\eta'(r)-(1-\chi(r))\cos\left(\frac{\pi}{2}\left(\frac{5}{8}-r\right)\right)&\text{if} & r\in [\frac{1}{4},\tfrac{1}{2}],
  \\
  -\cos\left(\frac{\pi}{2}\left(\frac{5}{8}-r\right)\right)&\text{if} &r\in[\tfrac{1}{2},\tfrac{5}{8}],
  \end{array}\right.
\end{gather*}
which clearly satisfies $-1<\widehat\zeta(r)<0$ for $r\in(0,\frac{5}{8})$. For any $r\in[\frac{1}{4},\frac{1}{2}]$, we have $\eta''(r)<0$ and $\eta'(r)>-\frac{1}{10}>-\cos\left(\frac{\pi}{2}\left(\frac{5}{8}-r\right)\right)$, so that
$$
\widehat{\zeta}'(r)<\chi'(r)\left(\eta'(r)+\cos\left(\frac{\pi}{2}\left(\frac{5}{8}-r\right)\right)\right)\leq0\quad\text{ for all }r\in[\tfrac{1}{4},\tfrac{1}{2}].
$$As a consequence, we have 
$$\widehat\zeta'(r)< 0 \quad \text{ for all } r\in(0,\tfrac{5}{8}).$$For each $\lambda\geq\frac{1}{4}$, define $$\sigma_{\lambda}(t):=\left(\chi_{(-\infty,\frac{1}{4}]\cup[\frac{1}{4}+\lambda,\infty)}+\frac{1}{4\lambda}\chi_{[\frac{1}{4},\frac{1}{4}+\lambda]}\right)\ast\phi,$$
where $\phi\in C_{c}^{\infty}(\mathbb{R})$ is a nonnegative mollifier with $\int_{\mathbb{R}}\phi dx=1$ and $\text{supp}(\phi)\subseteq(-\frac{1}{100},\frac{1}{100})$. Then $\int_{0}^{\frac{3}{8}+\lambda}\sigma_{\lambda}(s)ds=\frac{5}{8}$, so if we set $\zeta_{\lambda}(t):=\widehat{\zeta}(\int_{0}^{t}\sigma_{\lambda}(s)ds)$, then $\zeta_{\frac{1}{4}}\equiv\widehat{\zeta}$, $\zeta_{\lambda}|_{[0,\frac{1}{8}]}=\widehat\zeta|_{[0,\frac{1}{8}]}$, and $\zeta_{\lambda}|_{[\frac{26}{100}+\lambda,\frac{3}{8}+\lambda]}$
is the derivative of the warping function of a spherical cap. Moreover,
we have $\zeta_{\lambda}'(t)<0$ for all $t\in(0,\frac{3}{8}+\lambda)$, $\int_{0}^{\frac{5}{8}}\zeta_{\frac{1}{4}}(t)dt\geq-\frac{5}{8}>-1$,
and $\int_{0}^{\frac{3}{8}+\lambda}\zeta_{\lambda}(t)dt\leq \int_{\frac{1}{8}}^{\frac{3}{8}+\lambda}\zeta_{\lambda}(t)dt\leq \int_{\frac{1}{8}}^{\frac{3}{8}+\lambda}\zeta_{\lambda}(\frac{1}{8})dt=(\lambda+\frac{1}{4})\eta'(\frac{1}{8})\to-\infty$
as $\lambda\to\infty$. By the intermediate value theorem, there exists $\lambda\in(\frac{1}{4},\infty)$ such that $\int_{0}^{\frac{3}{8}+\lambda}\zeta_{\lambda}(t)dt=-1$.
Finally, setting $\widehat{\eta}(t):=1+\int_{0}^{t}\zeta_{\lambda}(s)ds$,
we get that $\widehat{\eta}|_{[\frac{26}{100}+\lambda,\frac{3}{8}+\lambda]}$
is the warping function of a spherical cap, $\widehat{\eta}|_{[0,\frac{1}{8}]}\equiv\eta|_{[0,\frac{1}{8}]}$,
$-1<\widehat\eta'(t)<0$ for all $t\in(0,\frac{3}{8}+\lambda)$, and $\widehat\eta''(t)<0$
for all $t\in(0,\frac{3}{8}+\lambda]$. Set $D:=\frac{3}{8}+\lambda$. Clearly, $\widehat{\eta}$ satisfies assertions $(i),(iii),(iv)$, and the part of assertion $(ii)$ concerning $r\in [0,\frac{1}{4}]$. We now compute
\begin{align*}
\widetilde{\psi}'(r)= &\ \left( \phi_1(r)\psi(r) +\phi_2(r) \right) \widehat{\eta}'(r)+\left( \phi_1(r)\psi'(r)+\phi_1'(r)(\psi(r)-1) \right) \widehat{\eta}(r),
\\
\widehat{\psi}''(r)=&\ \left( \phi_1(r)\psi(r)+\phi_2(r)\right) \widetilde{\eta}''(r)+2 \left( \phi_1(r)\psi'(r)+\phi_1'(r)(\psi(r)-1)\right)\widehat{\eta}'(r) 
\\
& \ +\left( 2\phi_1'(r)\psi'(r)+\phi_1'(r)\psi''(r)+\phi_1''(r)(\psi(r)-1) \right) \widehat{\eta}(r).
\end{align*}
Because $\widehat{\eta}$ is smooth, we know that there exists a universal constant $\widehat{\sigma}>0$ such that $-\widehat{\sigma}<\widehat{\eta}'(r)<1-\widehat{\sigma}$ and $\widehat{\eta}''(r)<-\widehat{\sigma}$ for all $r\in [\frac{1}{4},D-\frac{1}{8}]$. Thus
$$\widetilde{\psi}'(r)\leq (1+\delta)(1-\widehat{\sigma})+3\delta <1-\frac{1}{2}\widehat{\sigma},$$
$$\widetilde{\psi}''(r) \leq -(1-\delta)\widehat{\sigma} +20\delta <-\frac{1}{2}\widehat{\sigma}$$
for all $r\in [\frac{1}{4},D-\frac{1}{8}]$ assuming $\delta_0 <\delta_0(\widehat{\sigma})$. Assertion $(ii)$ follows. 
Finally, we note that
$$\widetilde{g}-g_{\operatorname{stan}}=\widehat{\eta}^2(r)\phi_1(r)(\psi(r)-1))\left( \phi_1(r)\psi(r)+\phi_2(r)+1 \right)dg_{\mathbb S^n},$$
which vanishes identically for $r\in [D-\frac{1}{8},D]$, while $\widehat{\eta}(r)\geq \sigma'$ for all $r\in (-\infty,D-\frac{1}{8}]$, for some universal constant $\sigma'>0$. Because $\phi_1,\psi,\phi_2,\widehat{\eta}$ are bounded in $C^k$ by some universal constants, we conclude that
$$|(\nabla^{g_{\operatorname{stan}}})^k\widetilde{g}-g_{\operatorname{stan}}|_{g_{\operatorname{stan}}}\leq C(k,\sigma')||\psi -1||_{C^k((-\delta^{-1},\delta^{-1}))}$$
on the entire surgery cap.
\end{proof}

By condition $(ii)$ ,  $\widetilde{g}$ satisfies the Hamilton-Ivey pinching condition. We use the diffeomorphism $\Phi$ to patch together the metric $\widetilde{g}$ with the Riemannian manifold $(\Omega_0,\overline{g})$.\\

Arguing exactly as in Lemma 63.1 of \cite{KleinerLott08}, we obtain a constant $E(\epsilon)$ large enough so that the standard solution $(\mathbb{R}^{n+1},g_{\operatorname{stan}})$ constructed above satisfies the assumption $CNA(1,\epsilon,E)$, except that the equivariant $\epsilon$-neck structure only extends to the initial time (rather than having duration 1 after the appropriate rescaling).

\begin{defin} We say that a Ricci flow with surgery $\mathcal{M}$ defined on $[0,T]$ satisfies the a priori assumptions with respect to a function $r:[0,T]\to (0,\infty)$ if it satisfies the Hamilton-Ivey pinching condition and assumption $CNA(r,\epsilon,E)$ for all points $(x,t)\in \mathcal{M}$ with $t\in [0,T]$ except possibly those which are removed by surgery.
\end{defin}

\begin{defin}\label{apassump} Suppose $\mathcal{M}$ is a Ricci flow with surgery defined on $[0,T]$ which satisfies the a priori assumptions. Given a nonincreasing function $\delta:[0,T]\to (0,\delta_0)$, where $\delta_0 >0$ is as in Lemma \ref{surger}, we say that $\mathcal{M}$ is a Ricci flow with $(r,\delta)$-cutoff if at each singular time $t$, the forward time slice $\mathcal{M}_t^+$ is obtained from $\mathcal{M}_t^-$ via the surgery procedure described above. More precisely, we discard all components of $\Omega \subseteq \mathcal{M}_t^-$ which do not intersect $\Omega_{\delta(t)r(t)}$, use Lemma \ref{epshorn} to find points $x_i$ in each horn $N_i \subseteq \Omega \subseteq \mathcal{M}_t^-$ with $R(x_i,t)=\max \{ h^{-2}(\delta(t)),\delta^{-2}(t) \}r^{-2}(t)$ which are centers of equivariant $\delta(t)$-necks, where we cut and apply Lemma \ref{surger}, gluing in caps to obtain $\mathcal{M}_t^+$.
\end{defin}

\begin{prop} There exist decreasing sequences $r_j, \delta_j>0$ such that for any nonincreasing functions $r,\delta: [0,\infty)\to(0,\infty)$ with $r<r_j$, $\delta <\delta _j$ on each time interval $[2^{j-1}\epsilon,2^j\epsilon]$, any rotationally invariant Riemannian manifold $(M^{n+1},g_0)$ with normalized initial data admits a 
rotationally invariant Ricci flow with $(r,\delta)$-cutoff defined for all time.
\end{prop}
\begin{proof} The proof is a modification of the three-dimensional case (Section 80 of \cite{KleinerLott08}), where the $\kappa$-noncollapsing through surgery theorem is replaced by Proposition \ref{uniform_noncollapsing}. We outline the proof for the convenience of the reader.

\end{proof}

\begin{thm} \label{extinction} For any rotationally invariant and closed Riemannian manifold $(M^{n+1},g_0)$, there exists $T=T(g_0)<\infty$ such that any Ricci flow with surgery with $(M,g_0)$ as initial condition becomes extinct before time $T$.
\end{thm}
\begin{proof}
Suppose the Ricci flow with surgery is given by the sequence of Ricci flows $(M_j,g_t^j)_{t\in [T_{j-1},T_j)}$, $j=1,...,N$, where $N\in \mathbb{N}\cup {\infty}$, and if $N=\infty$, then $\lim_{j\to \infty}T_j=\infty$. Let us consider a component $\widetilde M_j\subset M_j$ and fix an orbit $\mathcal{O}_j$ of $(\widetilde M_j,g_{T_{j-1}}^j)$, and let $\psi_j: \widetilde M_j \times [T_{j-1},T_j)$ be the corresponding warping function (with spatial parameter the signed distance from $\mathcal{O}_j$). Then 
$$\partial_t \psi_j = \partial_s^2 \psi_j -n\frac{1-(\partial_s\psi_j)^2}{\psi_j},$$
so that the maximum principle implies 
$$\max_{\widetilde M_j} \psi_j^2(\cdot,t)\leq \max_{\widetilde M_j} \psi_j^2(\cdot,T_{j-1})-n(t-T_{j-1})$$ 
for all $t\in [T_{j-1},T_j)$. Let $\Omega_j \subseteq M_j$ be the regular set of $(M_j, (g_t^j)_{t\in [T_{j-1},T_j)})$, so that $\psi_j| _{\Omega_j\cap\widetilde M_j}$ extends smoothly to a warping function $\psi_j(\cdot,T_j)$ for the rotationally-invariant Riemannian manifold $(\Omega_j\cap\widetilde M_j, \overline{g}_j)$, which thus satisfies $\max_{\Omega_j\cap\widetilde M_j}\psi_j(\cdot, T_j)\leq \max_{\widetilde M_j}\psi_j(\cdot,T_{j-1})-n(T_j-T_{j-1})$.
Moreover, $(M_{j+1},g_{T_j}^{j+1})$ is obtained from $(\Omega_j, \overline{g}_j)$ by the $\delta$-cutoff procedure, so that
$$\max_{M_{j+1}} \psi_{j+1}(\cdot,T_j)\leq \max_{\widetilde M_j\subset M_j}\max_{\Omega_j\cap\widetilde M_j}\psi_j(\cdot, T_j),$$
and the claim follows.
\end{proof}

\subsection{Singular Ricci flow spacetimes}

In this section, we show that, given any sequence of rotationally invariant Ricci flows with surgery whose time-0 time slices converge smoothly and whose surgery parameters converge to zero, the Ricci flow spacetimes also converge in the sense of \cite{KleinerLott17}. Most properties of this convergence are ensured by Theorem 4.1 of \cite{KleinerLott17}, whose proof carries over almost verbatim to our setting.  It remains to establish that the limit is also rotationally invariant in the appropriate sense, that the convergence is $O(n+1)$-equivariant, and that the limit satisfies the equivariant canonical neighborhood assumptions.

\begin{defin} A Ricci flow spacetime $\mathcal{M}$ is rotationally invariant if it admits a faithful $O(n+1)$ action $\theta: O(n+1)\times \mathcal{M}\to \mathcal{M}$ satisfying the following:\\
$(i)$ $\theta(A,\mathcal{M}_t)=\mathcal{M}_t$ for all $t\in [0,T)$ and $A\in O(n+1)$,\\
$(ii)$ for any $x\in \mathcal{M}$, there exists an $O(n+1)$-invariant product domain $U\subseteq \mathcal{M}$ containing $x$ and a time-preserving $O(n+1)$-equivariant diffeomorphism $\Phi:U\to M\times (-\epsilon,\epsilon)$ to a rotationally invariant Ricci flow.
\end{defin}

\begin{defin} A rotationally invariant Ricci flow spacetime $\mathcal{M}$ satisfies the equivariant $\epsilon$-canonical neighborhood assumption at scales $(r_1,r_2)$ if for every $x\in \mathcal{M}$ with $r_2^{-2} \leq R(x) \leq r_1^{-2}$,  $(\mathcal{M}_{\mathfrak{t}(x)},g_{\mathfrak{t}(x)},x)$ is equivariantly $\epsilon$-close to a rotationally invariant $\kappa$-solution.
\end{defin}

\begin{thm} (c.f. Theorem 4.1 in \cite{KleinerLott17}) \label{spacetimeexists} Let $\{\mathcal{M}^{j};j\in\N\}$ be a sequence of rotationally invariant $(n+1)$-dimensional Ricci flows with surgery, where $\mathcal{M}_{0}^{j}$ are compact normalized Riemannian manifolds such that $\{\mathcal{M}_{0}^{j};j\in\mathbb{N}\}$ is precompact in the $C^{\infty}$ Cheeger-Gromov topology, and let the surgery parameter $\delta_{j}:[0,\infty)\to(0,\infty)$ for $\mathcal{M}^{j}$ satisfy $\lim_{j\to\infty}\delta_{j}(0)=0$. Then, after passing to a subsequence, there is a Ricci flow spacetime $(\mathcal{M}^{\infty},\mathfrak{t}_{\infty},\partial_{\mathfrak{t}_{\infty}},g_{\infty})$ and a sequence of $O(n+1)$-equivariant diffeomorphisms $\Phi_{j}:U_{j}\to V_{j}$, where $U_{j}\subseteq\mathcal{M}^{j}$ and $V_{j}\subseteq\mathcal{M}^{\infty}$ are open, such that:

\begin{enumerate}[(i)]
\item For any $\overline{t}<\infty$ and $\overline{R}<\infty$, we have \begin{align*}\left\{x\in\mathcal{M}^{j};\mathfrak{t}_{j}(x)\leq\overline{t}\mbox{ and }R_{j}(x)\leq\overline{R}\right\} & \subseteq U_{j},
\\
\{x\in\mathcal{M}^{\infty};\mathfrak{t}_{\infty}(x)\leq\overline{t}\mbox{ and }R_{\infty}(x)\leq\overline{R}\} & \subseteq V_{j},\end{align*}
whenever $j\in\N$ is sufficiently large.
\item $\Phi_{j}$ is time-preserving, and $(\Phi_{j})_{\ast}\partial_{t_{j}}$ and $(\Phi_{j}^{-1})^{\ast}g_{j}$ converge smoothly on compact subsets of $\mathcal{M}^{\infty}$ to $\partial_{t_{\infty}}$ and $g_{\infty}$, respectively.
\item If $\mathcal{V}_{j},\mathcal{V}_{\infty}:[0,\infty)\to[0,\infty)$ is the volume at time t, then $\lim_{j\to\infty}\mathcal{V}_{j}=\mathcal{V}_{\infty}$ locally uniformly. 
\end{enumerate}
Furthermore, the following hold:
\begin{enumerate}[(a)]
\item $R_{\infty}:\mathcal{M}_{\leq T}^{\infty}\to\mathbb{R}$ is bounded from below and is proper for all $T\geq 0$.

\item $\mathcal{M}^{\infty}$ satisfies the Hamilton-Ivey pinching condition.

\item $\mathcal{M}^{\infty}$ is $\kappa_0$-noncollapsed on all scales, and for any $\epsilon>0$, $\mathcal{M}$ satisfies the equivariant $\epsilon$-canonical neighborhood assumption at scales $(0,r)$, where $r=r(\epsilon)>0$.
\end{enumerate}
\end{thm}
\begin{proof} By Theorem 4.1 of \cite{KleinerLott17}, there is a Ricci flow spacetime $\mathcal{M}^{\infty}$ satisfying the desired conclusions, except that $\mathcal{M}^{\infty}$ may not be a rotationally invariant spacetime, $\Phi_j$ may not be $O(n+1)$-equivariant, and the equivariant $\epsilon$-canonical neighborhood assumption may fail.

We observe that the intrinsic diameters of $O(n+1)$-orbits are uniformly bounded: smooth converge at time 0 gives $\sup_j \max_{x\in \mathcal{M}_0^j} \mbox{diam}_{g_0^j}(\mathcal{O}(x))<\infty$. Moreover, the quantities 
$$\mathcal{D}^j(t):=\max_{x\in \mathcal{M}_t^j} \mbox{diam}_{g_t^j}(\mathcal{O}(x))$$
are nonincreasing in $t$ for each fixed $j\in \mathbb{N}$ by the proof of Proposition \ref{extinction}.  

For each $x\in \mathcal{M}^{\infty}$, the canonical neighborhood theorem gives $r>0$ such that $P_{\mathcal{O}}(\Phi_j^{-1}(x),r))$ is unscathed, and $\rho \geq r$ on this set for all $j\in \mathbb{N}$ large enough. Thus the proof of Proposition \ref{equivariantcompactness} applies to give, after passing to a subsequence, a connected (possibly incomplete) Ricci flow $(M_x,g_t^x)_{t\in I_x}$ along with time-preserving and $\partial_{\mathfrak{t}}$-preserving diffeomorphisms $\xi_x^j:M_x \times (\mathfrak{t}(x)-\epsilon_x,\mathfrak{t}(x)+\epsilon_x) \to \mathcal{M}^j$ such that $(\xi_x^j)^{\ast}g^j \to g^x$ in $C_{loc}^{\infty}(M_x \times I_x)$. 

In particular, $\Phi_j \circ \xi_x^j:(M_x,(\xi_x^j)^{\ast}g^j)\to (\mathcal{M}^{\infty},(\Phi^{-1}_j)^{\ast}g^j)$ are open isometric embeddings of Ricci flow spacetimes, with $(\xi_x^j)^{\ast}g^j\to g^x$ and $(\Phi^{-1}_j)^{\ast}g^j\to g_{\infty}$ in $C_{loc}^{\infty}$ as $j\to \infty$. It follows that, after passing to a subsequence, $(\Phi_j \circ \xi_x^j)_{j\in \mathbb{N}}$ converges in $C_{loc}^{\infty}$ to an open isometric embedding of Ricci flow spacetimes $\xi_x^{\infty}:M_x \times I_x \to \mathcal{M}^{\infty}$. That is, $\xi_x^{\infty}$ is a diffeomorphism onto an open subset $W_x$ of $\mathcal{M}^{\infty}$, which is time-preserving, $\partial_{\mathfrak{t}}$-preserving, and is a Riemannian isometry when restricted to each time slice. We use $\xi_x^{\infty}$ to endow $W_x$ with the structure of a rotationally invariant product domain, by defining
$$A\cdot p := \xi_x^{\infty}(A\cdot (\xi_x^{\infty})^{-1}(p))$$
for $A\in O(n+1)$ and $p\in W_x$. Now define $\eta_x^j:=\xi_x^j \circ(\xi_x^{\infty})^{-1}: W_x \to \mathcal{M}^j$, which are $O(n+1)$-equivariant time-preserving embeddings of Ricci flow spacetimes. Then 
$$(\eta_x^j)^{\ast}g^j = ((\xi_x^{\infty})^{-1})^{\ast}(\xi_x^j)^{\ast}g^j\to((\xi_x^{\infty})^{-1})^{\ast}g^x=g_\infty$$
in $C_{loc}^{\infty}(W_x)$ as $j\to \infty$. Similarly, we have $((\eta_x^j)^{-1})_* \partial_{\mathfrak{t}}^i \to \partial_{\mathfrak{t}}$ in $C_{loc}^{\infty}(W_x)$.  
\\

\noindent \textbf{Claim 1:}
For any $O(n+1)$-equivariant diffeomorphism $\Phi: (a,b) \times \mathbb{S}^n \to (c,d) \times \mathbb{S}^n$, there is a unique diffeomorphism $\phi :(a,b)\to (c,d)$ such that $\Phi(r,\zeta)=(\phi(r),\pm \zeta)$ for all $(r,\zeta)\in (a,b)\times \mathbb{S}^n$.
\begin{proof}[Proof of the claim]
There are unique smooth maps $\rho: (a,b)\times \mathbb{S}^n \to (c,d)$ and $\theta: (a,b)\times \mathbb{S}^n \to \mathbb{S}^n$ such that $\Phi=(\rho,\theta)$, and $O(n+1)$-equivariance implies that $\rho(r,\cdot)$ is constant, and $\theta(r,\cdot):\mathbb{S}^n\to \mathbb{S}^n$ is $O(n+1)$-equivariant for each $r\in (a,b)$. Set $\phi(r):=\rho(r,\zeta)$ for any $\zeta \in \mathbb{S}^n$. Moreover, any $O(n+1)$-equivariant map $\mathbb{S}^n\to \mathbb{S}^n$ is equal to $\pm \mbox{id}_{\mathbb{S}^n}$, so the claim follows since $r\mapsto \theta(r,\cdot)$ is continuous with respect to the compact-open topology. Because $\Phi^{-1}$ is also of the form $\Phi^{-1}(r,\zeta)=(\psi(r),\pm \zeta)$, it follows that $\psi=\phi^{-1}$, hence $\phi$ is a diffeomorphism. \end{proof}

By applying Claim 1 to the $O(n+1)$-equivariant diffeomorphisms $(\eta_x^j)^{-1} \circ \eta_{x'}^j$ on $W_x \cap W_{x'}$, we see that the 
$O(n+1)$-actions on $W_x$ and $W_{x'}$ agree on their overlap, so we have a well-defined structure of a rotationally invariant Ricci flow spacetime on $\mathcal{M}^{\infty}$.  Moreover $(\eta_x^j)^{-1}\circ \eta_{x'}^j$ converge in $C_{loc}^{\infty}$ to an $O(n+1)$-equivariant isometry $\vartheta: W_x \cap W_{x'}\to W_x \cap W_{x'}$ of Ricci flow spacetimes.
\\

\noindent\textbf{Claim 2:} We can assume that $\vartheta = \text{id}_{W_x \cap W_{x'}}$.
\begin{proof}[Proof of the claim]
The only $O(n+1)$-equivariant isometries $(a,b)\times \mathbb{S}^n \to (a,b)\times \mathbb{S}^n$ are of the form $(s,\zeta)\mapsto ((-1)^k s, (-1)^l \zeta)$, where $k,l \in\{0,1\}$. The claim thus follows by replacing $\eta_{x'}^j$ with its composition by the isometry $\vartheta$. \end{proof}

We now let $K\subseteq \mathcal{M}^{\infty}$ be a fixed compact subset, and choose a finite subset $x_1,...,x_N\in K$ such that $K \subseteq \cup_{i=1}^N W_{x_i}$. We can moreover assume that $W_{x_i}\not \subseteq W_{x_j}$ whenever $i\neq j$, so that we can inductively (for any fixed $j\in \mathbb{N}$) alter $\eta_{x_i}^j$ as in Claim 2 so that $(\eta_{x_i}^j)^{-1}\circ \eta_{x_k}^j \to \text{id}_{W_{x_i}\cap W_{x_k}}$ in $C_{loc}^{\infty}$. The remainder of the proof proceeds along the lines of Theorem 9.31 in [Bam3]. We need only to verify that an averaging function can be chosen which preserves the $O(n+1)$-equivariance of spacetime embeddings, which is summarized in the ensuing claim.
\\

\noindent\textbf{Claim 3:} There exists a neighborhood $\Delta_2 \subseteq (\mathcal{M}^{\infty})^2$ of the diagonal $\{(x,x)\in\mathcal M^\infty; x\in\mathcal M^\infty \}$ and a smooth map 
$$\Sigma_2: [0,1]^2 \times \Delta_2 \to \mathcal{M}^{\infty}$$
which satisfies the following:
\begin{enumerate}[(1)]
\item $\Sigma_2(1,0,x_1,x_2)=x_1$; $\Sigma_2(0,1,x_1,x_2)=x_2$;
\item $\Sigma_2(s_1,s_2,x,x)=x$;
\item if $t=\mathfrak{t}(x_1)=\mathfrak{t}(x_2)$, then $\mathfrak{t}(\Sigma_2(s_1,s_2,x_1,x_2))=t$;
\item if $x_1,x_2 \in \mathcal{M}_t^{\infty}$ are connected by a minimizing geodesic segment $\gamma$ transverse to the regular orbits, then $\Sigma(s_1,s_2,x_1,x_2)\in \gamma$.
\end{enumerate}
\begin{proof}[Proof of the claim]
In general, we let $\Delta_2 \subseteq (\mathcal{M}^{\infty})^2$ be the open subset of points $(x_1,x_2)$ such that $x_1,x_2$ survive for all times $t\in [\min(t_1,t_2),\max(t_1,t_2)]$, where $t_i :=\mathfrak{t}(x_i)$,and such that $$\mbox{inj}_{g_t}(x_1(t)),\mbox{inj}_{g_t}(x_2(t))>2d_t(x_1(t),x_2(t))$$
for all $t\in [\min(t_1,t_2),\max(t_1,t_2)]$. Then, for any $x_1,x_2 \in \Sigma_2$ and $t\in [\min(t_1,t_2),\max(t_1,t_2)]$, there is a unique constant-speed minimizing geodesic $\gamma_{x_1,x_2,t}:[0,1]\to \mathcal{M}_t$ with respect to $g_t$ from $x_1(t)$ to $x_2(t)$, which depends smoothly on $(x_1,x_2)\in \Sigma_2$ and on $t\in [\min(t_1,t_2),\max(t_1,t_2)]$. We define
$$\Sigma_2(s_1,s_2,x_1,x_2):=\gamma_{x_1,x_2,\frac{s_1}{s_1+s_2}t_1+\frac{s_2}{s_1+s_2}t_2}\left( \frac{s_1}{s_1+s_2} \right).$$ 
It is easy to check that then (1)-(3) are satisfied. \end{proof}

It remains to show that the equivariant $\epsilon$-canonical neighborhood assumptions hold. Here, we argue as in Proposition 5.15 of \cite{KleinerLott17}. The essential idea is to repeat the arguments of Perelman's original proof (noting as before that the equivariant compactness theorem guarantees equivariant closeness to $\kappa$-solutions), but to replace the point-picking argument (which could a priori be troublesome due to the incompleteness of $\mathcal{M}$) with the $O(n+1)$-invariant  $\epsilon$-neck and $\epsilon$-cap neighborhoods (see Definition \ref{apassump}) obtained by passing the a priori assumptions to the limit. 
\end{proof}

Finally, we show that a rotationally invariant Ricci flow spacetime $\mathcal{M}$ can also be characterized by a warping function, now viewed as a function $\psi: \mathcal{M}\to (0,\infty)$. This will be used in section 11.
\begin{prop} \label{spacetimewarp} Given a Ricci flow spacetime $\mathcal{M}$, define a positive function $\psi\in C^\infty(\mathcal{M})$ by letting $\pi \psi(x)$ be the intrinsic $g_t$-diameter of $\mathcal{O}(x)$, given $x\in \mathcal{M}$. There is a vector field $\partial_s \in \mathfrak{X}(\mathcal{M})$, unique up to sign on each component of $\mathcal{M}$, which is $g_t$-orthogonal to each orbit $\mathcal{O}(x)\subseteq \mathcal{M}_t$, and which satisfies $g_t(\partial_s ,\partial_s)=1$. Moreover, we have
$$[\partial_{\mathfrak{t}}, \partial_s]=-n\frac{\partial_{s} (\partial_s \psi)}{\psi}\partial_s,$$
$$\partial_{\mathfrak{t}}\psi = \partial_s (\partial_s \psi )-(n-1)\frac{1-(\partial_s \psi)^2}{\psi}.$$
\end{prop}
\begin{proof} First consider $x_0 \in \mathcal{M}_{t_0}$, and let $U_{x_0}\subseteq \mathcal{M}$ be a product neighborhood of $x_0$ equipped with a spacetime isometry $\Phi: V_{x_0}\times I_{x_0} \to U_{x_0}$, where $V\subseteq \mathbb{S}^{n+1}$, $I_{x_0}\subseteq \mathbb{R}$ is an interval containing $t_0$, and $(V_{x_0},(g_t')_{t\in I_{x_0}})$ is a rotationally invariant Ricci flow. Let $\partial_s' \in \mathfrak{X}(V_{x_0}\times I_{x_0})$ be the vector field determined by choosing a signed distance function from the orbit $\mathcal{O}(x_0)$, and define $\partial_s^{x_0} := \Phi_{\ast}\partial_s'$. Now choose an open cover of such neighborhoods $(U_{x_j})_{j\in \mathbb{N}}$, and note that on overlaps $U_{x_j} \cap U_{x_k}$, we have $\partial_s^{x_j} =\pm \partial_s^{x_k}$. By changing the sign of the $\partial_s^{x_j}$ appropriately, we may therefore glue these vector fields together to obtain $\partial_s \in \mathfrak{X}(\mathcal{M})$. The claimed identities follow by pushing forward the corresponding identities on the $U_{x_j}\times I_{x_j}$
\end{proof}

\section{Comparison domain and preliminary lemmas}

\subsection{Comparison domain, comparison, and a priori assumptions}

Recall that Bamler-Kleiner \cite{BamlerKleiner17} proved their stability and uniqueness theorem for the singular Ricci flow by constructing a comparison between two Ricci flow space-times, and estimating the Ricci-DeTurck perturbation of the comparison map. In this subsection, we shall introduce  modifications to these definitions adapted to the rotationally-invariant setting. Since the original definitions are long and are not substantially different from ours, we shall only indicate their differences, and the reader may refer to \cite{BamlerKleiner17} for more details.

\begin{defin}[Definition 7.1 in \cite{BamlerKleiner17}, comparison domain.]
An $O(n+1)$-invariant comparison domain over the time interval $[0,t_J]$ on a $O(n+1)$-invariant Ricci flow spacetime is a triple $\left(\mathcal{N},\{\mathcal N^j\}_{1\leq j\leq J},\{t_j\}_{j=0}^J\right)$ satisfying \cite[Definition 7.1]{BamlerKleiner17} along with the following:
\begin{enumerate}[(1)]
    \item  Each slice of $\mathcal N$ is a union of $O(n+1)$-orbits.
    \item Each extension cap is also a union of orbits, with exactly $1$ singular orbit, and is $O(n+1)$-equivariantly diffeomorphic to the $(n+1)$-disk.
\end{enumerate}
\end{defin}

\begin{defin}[Definition 7.2 in \cite{BamlerKleiner17}, comparison]
Consider two $O(n+1)$-invariant Ricci flow spacetimes $\mathcal M$ and $\mathcal M'$ and an $O(n+1)$-invariant comparison domain $\left(\mathcal{N},\{\mathcal N^j\}_{1\leq j\leq J},\{t_j\}_{j=0}^J\right)$ defined on $\mathcal M$ over the time interval $[0,t_J]$. An $O(n+1)$-equivariant comparison from $\mathcal M$ to $\mathcal M'$ defined on $\left(\mathcal{N},\{\mathcal N^j\}_{1\leq j\leq J},\{t_j\}_{j=0}^J\right)$ is a triple $\left(\operatorname{Cut},\phi,\{\phi^j\}_{j=1}^{J^*}\right)$ (over the time interval $[0,t_{J^*}]$) satisfying \cite[Definition 7.2]{BamlerKleiner17} along with the following:
\begin{enumerate}[(1)]
    \item The disks and parabolic disks are the orbit version $B_{\mathcal{O}}$ and $P_{\mathcal{O}}$ as defined in Section 2.1.
    \item $\phi$ is $O(n+1)$-equivariant.
    \item Each cut is a union of $O(n+1)$-orbits, containing exactly one singular orbit, and is $O(n+1)$-equivariantly diffeomorphic to the $(n+1)$-disk.
\end{enumerate}
\end{defin}

We now introduce modifications to the a priori assumptions which determine the properties of the inductive construction of the comparison and its domain. We remind the reader that our definition of the Ricci-DeTurck perturbation is not in any way different from \cite[Definition 7.3]{BamlerKleiner17}, only that, because of our assumptions of the comparison domain and the comparison, the Ricci-DeTurck perturbation in our case is always $O(n+1)$-invariant.

\begin{defin}[Definition 7.4 in \cite{BamlerKleiner17}, equivariant a priori assumptions (APA1)---(APA6)]
Consider two $O(n+1)$-invariant Ricci flow spacetimes $\mathcal M$ and $\mathcal M'$. Let $\left(\mathcal{N},\{\mathcal N^j\}_{1\leq j\leq J},\{t_j\}_{j=0}^J\right)$ be a comparison domain defined on $\mathcal M$ over the time interval $[0,t_J]$ and let $\left(\operatorname{Cut},\phi,\{\phi^j\}_{j=1}^{J^*}\right)$ be a comparison from $\mathcal{M}$ to $\mathcal M'$ defined on $\left(\mathcal{N},\{\mathcal N^j\}_{1\leq j\leq J},\{t_j\}_{j=0}^J\right)$ over the time interval $[0,t_{J^*}]$. We say that $\left(\mathcal{N},\{\mathcal N^j\}_{1\leq j\leq J},\{t_j\}_{j=0}^J\right)$ and $\left(\operatorname{Cut},\phi,\{\phi^j\}_{j=1}^{J^*}\right)$ satisfy the equivariant a priori assumptions (APA1)---(APA6) with respect to the tuple of parameters $(\eta_{\operatorname{lin}},\delta_{\operatorname{n}},\lambda,D_{\operatorname{cap}},\Lambda,\delta_{\operatorname{b}},\epsilon_{\operatorname{can}},r_{\operatorname{comp}})$ if they satisfy \cite[Definition 7.4]{BamlerKleiner17} along with the following:
\begin{enumerate}[(1)]
    \item (APA3(a)) The boundary components of $\partial\mathcal N_{t_j-}$ are central $n$-spheres ($O(n+1)$-orbits) of $O(n+1)$-equivariant $\delta_{\operatorname{n}}$-necks at scale $r_{\operatorname{comp}}$.
    \item (APA4) The component $\mathcal C$ is $O(n+1)$-equivariantly diffeomorphic to an $(n+1)$-disk.
    \item (APA5(c)) The closeness between the extension cap $\mathcal C$ and the scaled Bryant soliton is measured by an $O(n+1)$-equivariant diffeomorphism.
\end{enumerate}
\end{defin}

\begin{defin}[Definition 7.5 in \cite{BamlerKleiner17}, equivariant a priori assumptions (APA7)---(APA13)]
Let $\left(\mathcal{N},\{\mathcal N^j\}_{1\leq j\leq J},\{t_j\}_{j=0}^J\right)$ and $\left(\operatorname{Cut},\phi,\{\phi^j\}_{j=1}^{J}\right)$ be a comparison domain and comparison as the previous definition, over the same time interval $[0,t_J]$. They are said to satisfy the equivariant a priori assumptions (APA7)---(APA13) with respect to the tuple of parameters $(T,E,H,\eta_{\operatorname{lin}},\nu,\lambda,\eta_{\operatorname{cut}},D_{\operatorname{cut}},W,A,r_{\operatorname{comp}})$ if they satisfy \cite[Definition 7.4]{BamlerKleiner17}.
\end{defin}

\subsection{Geometry of non-necklike canonical neighbourhood}
\newcommand{\dn}{\delta_{\operatorname{n}}}
\newcommand{\db}{\delta_{\operatorname{b}}}
\newcommand{\bardn}{\overline{\delta}_{\operatorname{n}}}
\newcommand{\bardb}{\overline{\delta}_{\operatorname{b}}}
\newcommand{\ecan}{\epsilon_{\operatorname{can}}}
\newcommand{\barecan}{\overline{\epsilon}_{\operatorname{can}}}
\newcommand{\Bry}[1]{(M_{\operatorname{Bry}},{#1}g_{\operatorname{Bry}},x_{\operatorname{Bry}})}
\newcommand{\rcomp}{r_{\operatorname{comp}}}
\newcommand{\barrcomp}{\overline{r}_{\operatorname{comp}}}
\newcommand{\etalin}{\eta_{\operatorname{lin}}}
\newcommand{\baretalin}{\overline{\eta}_{\operatorname{lin}}}
\newcommand{\Dcap}{D_{\operatorname{cap}}}
\newcommand{\uDcap}{\underline{D}_{\operatorname{cap}}}

In this subsection, we shall include some auxiliary lemmas which are essential for the simplication of \cite[Section 11]{BamlerKleiner17}. These lemmas, especially Lemma \ref{nonneckCN=Bryant_Application} below, basically say that, if a canonical neighborhood is not close to a neck, and geometrically it does not look like a closed component, then this  canonical neighborhood is very close (up to scaling) to a piece of the Bryant soliton not too far from the tip. 

In this section, we will always let $s$ be a signed distance function whose gradient is orthogonal to the orbits.  We first note a modification of Lemma 8.2 of \cite{BamlerKleiner17}.

\begin{lem} \label{eq:cannbd} If $$\delta \leq \overline{\delta}, \hspace{6 mm} \ecan \leq \barecan(\delta),$$
then there exists $C_0 = C_0(\delta)<\infty$ such that the following holds. Let $(M^n,g)$ be a rotationally invariant Riemannian manifold which satisfies the equivariant $\ecan$-canonical neighborhood assumption at some point $x\in M$ which is not the center of an equivariant $\delta$-neck at scale $\rho(x)$. Then there is a compact, connected, $O(n+1)$-invariant domain $V\subseteq \mathcal{M}_t$, where $t=\mathfrak{t}(x)$, such that the following hold:
\begin{enumerate}[(1)]
    \item $B_{\mathcal{O}}(x,\delta^{-1}\rho (x))\subseteq V$,
    \item $0 <\rho (y_1) <C_0 \rho(y_2)$ for all $y_1,y_2\in V$,
    \item $\mbox{diam}_t(V)<C_0 \rho(x)$,
    \item If $\partial V\neq \emptyset$, then $V$ is a rotationally invariant (n+1)-disk, and $\partial V\neq \emptyset$ is a central sphere of an equivariant $\delta$-neck.
\end{enumerate}
\end{lem}

\begin{lem}\label{nonneckCN}
If
\begin{eqnarray*}
\delta\leq\bar{\delta},\quad \ecan\leq\barecan(\delta),
\end{eqnarray*}
then there exists $C_0=C_0(\delta)$ such that the following holds. Let $\mathcal{M}$ be an $(n+2)$-dimensional rotationally invariant Ricci flow spacetime. Let $x\in\mathcal {M}_t$ for some $t\in[0,\infty)$. Assume that $\mathcal{M}$ satisfies the equivariant $\ecan$-canonical neighborhood assumption at $x$ and that $\operatorname{diam}_{g_t}(N)>C_0(\delta)\rho(x)$, where $N$ is the connected component of $\mathcal{M}_t$ that contains $x$. 
\begin{enumerate}[(1)]
\item With respect to these parameters Lemma \ref{eq:cannbd} holds at $x$.
\item If $x$ is not the center of an equivariant $\delta$-neck, then exactly one singular orbit of $\mathcal{M}_{t}$ is contained in $\displaystyle B_{\mathcal{O}}\big(x,C_0(\delta\big)\rho(x))$.
\item If (2) happens, then all $O(n+1)$-orbits lying between the orbit of $x$ and this singular orbit are entirely contained in $\displaystyle B_{\mathcal{O}}\big(x,C_0(\delta\big)\rho(x))$.
\end{enumerate}
\end{lem}
\begin{proof} 
Let $x$ satisfy the assumptions of (2). By Lemma \ref{eq:cannbd}, since $x$ is not the center of a equivariant $\delta$-neck, there exists a compact, connected domain $V\subseteq \mathcal{M}_{\mathfrak{t}(x)}$ satisfying (1)---(4) in the statement of that lemma. If $\partial V=\emptyset$, then $V$ equals the entire component of $\mathcal{M}_{\mathfrak{t}(x)}$, but $\mbox{diam}(V)<C_0 \rho (x)$, contradicting our hypothesis. Thus $\partial V\neq \emptyset$, and $V$ is equivariantly diffeomorphic to $\mathbb{B}^{n+1}$. Since $B(x,\delta^{-1}\rho(x))\subset V$, we have that $V$ must contain the orbit $\{s=s(x)\}$ at $x$. By Lemma \ref{eq:cannbd}, we also have that $\partial V$ is an $O(n+1)$ orbit. Therefore, either $\partial V\subset \{s>s(x)\}$ or $\partial V\subset \{s<s(x)\}$. If, say, the latter happened, then $\{s>s(x)\}\subset V$. Since $V$ is a set with bounded diameter, and consisting of smooth points, we have $\{s>s(x)\}$ must contain a singular orbit. This proves (2) and (3).
\end{proof}

\vspace{6 mm}

\begin{lem}\label{nonneckCN=Bryant}
If
\begin{eqnarray*}
\dn\leq\bardn,\quad \delta\leq\overline{\delta},
\end{eqnarray*}
then there exists $D_0=D_0(\delta_{\operatorname{n}},\delta)<\infty$ such that the following hold. Let $(M,g)$ be a time slice of a rotationally invariant $\kappa$-solution, and suppose $x\in M$ is not the center of an equivariant $\dn$-neck. Then  one of the following is true.
\begin{enumerate}[(1)]
\item $(M,g)$ is compact and $\operatorname{diam}(M)\leq D_0\rho(x)$.
\item $B_{\mathcal{O}}(x,\delta^{-1}\rho(x))$-neighborhood of $x$ is equivariantly $\delta$-close (in the $C^{\lfloor\delta^{-1}\rfloor}$ topology) to a piece of a scaled Bryant soliton (though $x$ does not necessarily corresponding to $x_{\operatorname{Bry}}$). 
\end{enumerate}

\noindent In case (1), we also have $C_0(D_0)^{-1}\rho(x)\leq\rho(y)\leq C_0(D_0)\rho(x)$ for all $y\in M$, where $C_0$ is a constant depending only on $D_0$.

\noindent In case (2), if we also assume
\begin{eqnarray*}
\delta\leq\bar{\delta}(\dn),
\end{eqnarray*}
then there exist $c(\dn)\leq 1\leq C(\dn)\leq \frac{1}{10}\delta^{-1}$, such that the aforementioned Bryant soliton has scale within $[c(\dn)\rho(x),2\rho(x)]$, and the distance from the tip to $x$ is no greater than $C(\dn)\rho(x)$.
\end{lem}

\begin{proof}
To prove the first statement, we argue by contradiction. Suppose the lemma fails for some $\dn, \delta>0$. Let $(M_i,g_i,x_i)$ be a sequence of counterexamples, rescaled so that $\rho(x_i)=1$. Hence we have, that each $x_i$ is not the center of an equivariant $\dn$-neck at scale 1, that each $B(x_i,\delta^{-1})$ is not $\delta$-close to any Bryant soliton in the $C^{\lfloor\delta^{-1}\rfloor}$ sense, and that $\operatorname{diam}_{g_i}(M_i)\nearrow\infty$.

By the equivariant $\kappa$-compactness theorem, $(M_i,g_i,x_i)$ converge in the smooth Cheeger-Gromov sense to another final slice of a rotationally invariant $\kappa$-solution $(M_\infty,g_\infty, x_\infty)$, such that $x_\infty$ is not the center of an equivariant $\dn$-neck and $B_{\mathcal{O}}(x_\infty,\delta^{-1})$ is not $\delta$-close to any Bryant soliton in the $C^{\lfloor\delta^{-1}\rfloor}$ sense. But this is not possible, for $M_\infty$ is noncompact, hence it can only be either the standard cylinder or the Bryant soliton.

If (1) holds, then Lemma 9.65 of \cite{MorganTian} and $\operatorname{diam}(M)\leq D_0\rho(x)$ give
$$C_0(D_0)^{-1}\rho(x) \leq \rho (y) \leq C_0(D_0)\rho(x)$$
for all $y\in M$.

Suppose instead that (2) holds. Given $\dn \leq \bardn$, we can find $c(\dn)$ such that any $y\in \Bry{}$ (Bryant soliton of scale $1$) satisfying $\rho(y)\geq \big(2c(\dn)\big)^{-1}$ is the center of an equivariant $\frac{1}{10}\dn$-neck. This follows from the observation that for any sequence of basepoints $y_i \in M_{\operatorname{Bry}}$ with $d(x_{\operatorname{Bry}},y_i)\to \infty$, the sequence $$(M_{\operatorname{Bry}},\rho(y_i)^{-2}g_{\operatorname{Bry}},y_i)$$ 
converges in the pointed Cheeger-Gromov sense to a round cylinder. 

Now fix an open embedding
$$\psi: B_{\mathcal{O}}(x,\delta^{-1}\rho(x))\rightarrow W\subset\Bry{}$$
along with some $a>0$ such that
$$||g_{\operatorname{Bry}}-a^{-2}(\psi^{-1})^*g ||_{C^{\lfloor\delta^{-1}\rfloor}(W,g_{\operatorname{Bry}})}<\delta.$$
Because $\rho$ attains its minimum on $(M_{\operatorname{Bry}},g_{\operatorname{Bry}})$ at $x_{\operatorname{Bry}}$, we have $2a^{-1}\rho(x)\geq\rho(\psi(x))\geq\rho(x_{\operatorname{Bry}})=1$. On the other hand, we have $\rho(\psi(x))\geq\frac{1}{2}a^{-1}\rho(x)$ when $\delta \leq \overline{\delta}$. If $a<c(\dn)\rho(x)$, then $\psi(x)$ would be the center of a $\frac{1}{10}\dn$-neck, and so $x$ would be the center of a $\dn$-neck, contradicting our assumption on $x$. Hence $a\geq c(\dn)\rho(x)$. Finally, the scalar curvature asymptotics for $\Bry{}$ give $c(n)>0$ such that 
$$\rho(y) \geq c(n)d_{g_{Bry}}^{\frac{1}{2}}(x_{\operatorname{\text{Bry}}},y)$$
for all $y\in M_{\operatorname{Bry}}$, so
$$d_g(\psi^{-1}(x_{\operatorname{Bry}}),x) \leq  C(n,\dn)\cdot a\cdot (\rho(\psi(x)))^2\leq  C(n,\dn)\cdot a\cdot (2a^{-1}\rho(x))^2\leq C(n,\dn)\rho(x).$$
\end{proof}
\bigskip

Applying the above theorem to the a Ricci flow spacetime which satisfies the canonical neighborhood condition, we obtain the following result.
\\

\begin{lem}\label{nonneckCN=Bryant_Application}
If
\begin{eqnarray*}
\dn\leq\bardn,\quad \delta\leq\bar{\delta}(\dn),\quad D\geq\underline{D}(\dn, \delta), \quad \ecan\leq\barecan(\dn,\delta, D), 
\end{eqnarray*}
then there exist $c(\dn)\leq 1\leq C(\dn)\leq \frac{1}{100}\delta^{-1}$, such that the following holds. Let $x\in\mathcal{M}_t$ be a point in an $(n+1)$-dimensional rotationally invariant Ricci flow space-time $\mathcal{M}$, such that at this point the  equivariant $\ecan$-canonical neighborhood  condition holds. Assume
\begin{enumerate}[(1)]
\item $x$ is not the center of an equivariant $\dn$-neck and
\item there exists $y\in\mathcal{M}_t$ in the same component of $\mathcal{M}_t$ as $x$, such that $\rho(y)\geq D\rho(x)$ or $\rho(y)\leq D^{-1} \rho(x)$. 
\end{enumerate}
Then there exists an rotationally equivariant diffeomorphism
\begin{eqnarray*}
\psi: M_{\operatorname{Bry}}(\delta^{-1})\rightarrow W\subset\mathcal{M}_t,
\end{eqnarray*}
such that
\begin{eqnarray*}
\big| \big| (a\rho(x))^{-2}\psi^*g-g_{\operatorname{Bry}} \big| \big|_{C^{\lfloor\delta^{-1}\rfloor}(B(x_{\operatorname{Bry}},\delta^{-1}))}<\delta,
\\
d_{g_{\operatorname{Bry}}}(x_{\operatorname{Bry}},\psi^{-1}(x))\leq C(\dn),
\end{eqnarray*}
where $a\in[c(\dn),2]$.
\end{lem}

\begin{proof}
Let $\phi : U \to V$ be an $O(n+1)$-equivariant diffeomorphism, where $U \subseteq \mathcal{M}_t$ is an $O(n+1)$-invariant open set satisfying
$$\left\{ y\in \mathcal{M}_t ; d_t(\mathcal{O}(x),y)< \ecan^{-\frac{1}{2}}\cdot \rho(x) \right\} \subseteq U,$$ and $V \subseteq M$ is an open subset of the final slice $(M,g)$ of a rotationally invariant $\kappa$-solution, such that $\rho(\phi (x))=1$ and
$$||\rho(x)^{-2}(\phi^{-1})^{\ast}g_t -g||_{C^{\lfloor \ecan^{-1} \rfloor}(V,g)} <\ecan.$$ Assuming $\ecan < \barecan(\dn)$, we have that $\phi(x)$ is not the center of an equivariant $\frac{1}{2}\dn$-neck, and $\phi(x)$ must fall into one of the cases in Lemma \ref{nonneckCN=Bryant}.

Let $\delta\leq\frac{1}{2}\bar\delta(\frac{1}{2}\delta_{\operatorname{n}})$, $D_0=D_0(\frac{1}{2}\dn, \frac{1}{2}\delta)$, and $D\geq\underline{D}:=2C_0(2D_0)$, where $\bar{\delta}(\cdot)$, $D_0(\cdot,\cdot)$, and $C_0(\cdot)$ are the parameters as defined in Lemma \ref{nonneckCN=Bryant}. We shall show that if $\ecan\leq\bar\epsilon_{\operatorname{can}}(\delta_{\operatorname{n}},\delta,D)$, then case (1) of Lemma \ref{nonneckCN=Bryant} (based at $x$) cannot occur. Assume by contradiction that $(M,g)$ is closed and $\operatorname{diam}(M)\leq D_0\rho(\phi(x))$, then we must have that $U$ contains the whole connected component of $\mathcal{M}_t$, and this component is smooth and closed. It then follows from Lemma \ref{nonneckCN=Bryant}(1) that $(\frac{2}{3}D)^{-1}\rho(x)\leq\rho(y)\leq \frac{2}{3}D\rho(x)$ for all $y\in \mathcal{M}_t$ in the same component as $x$. This is a contradiction to assumption (2) of the lemma.

With the above arrangement of parameters, we have that the geometry near $\phi(x)$ must be as described in case (2) of Lemma \ref{nonneckCN=Bryant}, the lemma then follows immediately.
\end{proof}

\section{Semi-local maximum principle}

\noindent In this section, we list the appropriate modifications that must be made in Section 9 of \cite{BamlerKleiner17}. The Anderson-Chow pinching estimate is necessary. For the rotationally invariant case, both \cite{Carson18} and \cite{ChenWu16} have already proved this estimate. We shall include a simpler version below.


\begin{lem}[Lemma 9.4 in \cite{BamlerKleiner17}]
Let $(M^{n+1},g_t)_{t\in(-T,0]}$ be a rotationally invariant Ricci flow on a connected $(n+1)$-manifold $M$. Assume that $(M,g_t)$ has nonnegative scalar curvature everywhere. Let $(h_t)_{t\in(-T,0]}$ be a rotationally invariant linearized Ricci-DeTurck flow:
\begin{align*}
    \frac{\partial h}{\partial t}=\Delta_{L,g_t}h.
\end{align*}
Assume that $|h|\leq CR$ on $M\times(-T,0]$ for some $C>0$ and that $|h|(x_0,0)=CR(x_0,0)$ for some $x_0\in M$. Then we have
\begin{eqnarray*}
|h|=CR\quad\text{ on }\quad M\times(-T,0].
\end{eqnarray*}
\end{lem}
\begin{proof}
First of all, we observe that we may assume $R>0$ everywhere on $M\times(-T,0]$ instead. Suppose this is not the case, then let $t_0$ be the last time when $R(\cdot,t_0)$ attains $0$. By the strong maximum principle and the assumption of the lemma, we have
$$0=|h|\leq CR=0 \quad\text{ everywhere on }\quad M\times(-T,t_0]$$
and the conclusion automatically holds on $M\times(-T,t_0]$. If $t_0=0$, then we are done. If $t_0<0$, since $R>0$ everywhere on $M\times(t_0,0]$, we may then replace $-T$ by $t_0$. 

By the same reasoning, we may also assume $|h|>0$ everywhere on $M\times(-T,0]$. The inequality below follows from a straightforward computation and Kato's inequality.
\begin{eqnarray*}
\frac{\partial}{\partial t}|h|
\leq \Delta |h|+\frac{2Rm(h,h)}{|h|^2}\cdot |h|.
\end{eqnarray*}

On the other hand, the curvature evolution equation implies that
\begin{eqnarray*}
\frac{\partial}{\partial t}(CR)=\Delta(CR)+\frac{2|Ric|^2}{R}\cdot CR.
\end{eqnarray*}
Since, by the next lemma, we have $$\frac{Rm(h,h)}{|h|^2}\leq\frac{|Ric|^2}{R},$$
the conclusion then follows from applying the parabolic strong maximum principle to $|h|-CR$.
\end{proof}

\begin{lem}
Suppose the Riemannian metric $g$ is rotationally invariant and has positive scalar curvature. Then for any symmetric $(0,2)$ tensor field $h$ that is also rotationally-symmetric, 
we have $$\frac{Rm(h,h)}{\left|h\right|^2}\le \frac{\left|Rc\right|^2}{R},$$ where 
$Rm(h,h)=R_{ijkl}h_{il}h_{jk}$. 
\end{lem}
\begin{proof}
At a given point $p=(r,x)\in I\times \mathbb{S}^n$, let $e_0=\partial_r$, and let $\{e_i\}_{i=1}^{n}$ be a basis for $T_x \mathbb{S}^n$ which is orthonormal with respect to the restricted metric, and such that $h_{ij}=\delta_{ij} h_{ii}$. Because $h_{ii}$ are the eigenvalues of $h|T_x \mathbb{S}^n$, it is clear by $O(n+1)$-invariance that $h_{ii}=h_{jj}$ for $1\leq i,j\leq n$. The only independent components of the Riemann curvature are $R_{0ii0}=\lambda=-\frac{\psi''}{\psi}$ and 
$R_{ijji}=\mu=\frac{1-(\psi')^2}{\psi^2}$ (note here that $\lambda$ and $\mu$ are now half what they were in Section \ref{geometryhighcurvature}), where both $i,j$ are non-zero and distinct, and here the index $0$ corresponds to an input of $\partial_r$, and all other indices correspond to vectors on $\mathbb{S}^n$. 
Then 
\begin{align*}
    R_{ijkl}h_{il}h_{jk}&=2\sum_{i<j}R_{ijji}h_{ii}h_{jj}
    =2\lambda h_{00}\sum_{j=1}^{n}h_{jj}+2\sum_{1\le i<j\le n}\mu h_{ii}h_{jj}\\
    &=2n\lambda h_{00}h_{11}+n(n-1)\mu h_{11}^2,\\
    R_{00}&=n\lambda,\\
    R_{ii}&=\lambda+(n-1)\mu,\\
    \left|Rc\right|^2&=n^2\lambda^2+n(\lambda+(n-1)\mu)^2, \ \text{and} \\
    R&=2n\lambda+n(n-1)\mu .
\end{align*}
Consider the matrix 
\begin{align*}
   A=\begin{pmatrix}
    0&\sqrt{n}\lambda\\
    \sqrt{n}\lambda &(n-1)\mu
    \end{pmatrix};
\end{align*}
its eigenvalues are given by
\begin{align*}
    t_1=\frac{1}{2}\left((n-1)\mu+\sqrt{(n-1)^2\mu^2+4n\lambda^2}\right),\quad t_2=\frac{1}{2}\left((n-1)\mu-\sqrt{(n-1)^2\mu^2+4n\lambda^2}\right).
\end{align*}
Now, at any point $p$, the expression $\frac{Rm(h,h)}{\left|h\right|^2}$ coincides with $\frac{v^T Av}{\left|v\right|^2}$, where $v=(h_{00},\sqrt{n}h_{11})$. Therefore, the result will follow from the observation that  $$\max\{|t_1|,|t_2|\}\leq\frac{\left|Rc\right|^2}{R}.$$ This is clearly true if $\mu=0$. In the case when $\mu> 0$, we have that $\max\{|t_1|,|t_2|\}=t_1$. By the scale invariance of this inequality, we can assume $(n-1)\mu=1$, we have $2\lambda+1>0$ because $R>0$, and the following computation holds 
\begin{align*}
    &\quad\ \left(1+\sqrt{1+4n\lambda^2}\right)\le \frac{n\lambda^2+(\lambda+1)^2}{\lambda+\frac{1}{2}}\\
    &\Leftarrow(\lambda+\frac{1}{2})\sqrt{(1+4n\lambda^2)}\le n\lambda^2+(\lambda+1)^2-(\lambda+\frac{1}{2})\\
    &\Leftarrow(\lambda^2+\lambda+\frac{1}{4})(1+4n\lambda^2)\le \left((n+1)\lambda^2+\lambda+\frac{1}{2}\right)^2\\
   &\Leftarrow 0 \leq \big((n-1)\lambda-1\big)^2.
\end{align*}
In the case when $\mu<0$, we have that $\max\{|t_1|,|t_2|\}=-t_2$. Then, scaling so that $(n-1)\mu=-1$, we have $2\lambda-1>0$ because $R>0$, so we compute:
\begin{align*}
    &\quad\ \left(\sqrt{1+4n\lambda^2}-1\right)\le \frac{n\lambda^2+(\lambda-1)^2}{\lambda-\frac{1}{2}}\\
    &\Leftarrow(\lambda-\frac{1}{2})\sqrt{(1+4n\lambda^2)}\le n\lambda^2+(\lambda-1)^2+(\lambda-\frac{1}{2})\\
    &\Leftarrow(\lambda^2-\lambda+\frac{1}{4})(1+4n\lambda^2)\le \left((n+1)\lambda^2-\lambda+\frac{1}{2}\right)^2\\
   &\Leftarrow 0 \leq \big((n-1)\lambda+1\big)^2,
\end{align*}
as required.
\end{proof}

\begin{prop}[Theorem 9.8 in \cite{BamlerKleiner17}, The Vanishing Theorem]
Let $(M^{n+1},g_t)_{t\in(-\infty,0]}$ be a rotationally invariant $\kappa$-solution, and let $(h_t)_{t\in(-\infty,0]}$ be a rotationally invariant linearized Ricci-DeTurck flow. Assume that there are numbers $\chi>0$ and $C<\infty$ such that
\begin{eqnarray}
|h|\leq CR^{1+\chi}\quad\text{ on }\quad M\times(-\infty,0].
\end{eqnarray}
Then $h\equiv 0$ everywhere.
\end{prop}
\begin{proof}
The proof is essentially identical to the original proof of \cite[Theorem 9.8]{BamlerKleiner17}. The only major difference occurs when it comes time to apply the fact that there are no $\kappa$-solutions with scalar curvature bounded uniformly from below. This is known to be true for three-dimensional flows. On the other hand, Theorems \ref{kappa-solution-classification-noncompact} and \ref{kappa-solution-classification-compact} together imply that the complete list of  $n+1$-dimensional rotationally-invariant Ricci flow $\kappa$-solutions is as follows:
\begin{itemize}
    \item the Bryant soliton;
    \item the shrinking cylinder;
    \item the shrinking sphere; or 
    \item Perelman's `sausage' ancient solution (a compact flow which asymptotically behaves like two Bryant solitons glued together). 
\end{itemize}
None of these has scalar curvature uniformly bounded from below.
\end{proof}
We now state two propositions which become useful later in the paper. 


\vspace{6 mm}

\begin{prop}[Proposition 9.1 in \cite{BamlerKleiner17}, the semi-local maximum principle]
If $$E>2,\quad H\geq\underline{H}(E),\quad \eta_{\operatorname{lin}}\leq \bar \eta_{\operatorname{lin}}(E),\quad \epsilon_{\operatorname{can}}\leq \bar \epsilon_{\operatorname{can}}(E),$$ then there are constants $L=L(E)$ and $C=C(E)<\infty$, such that the following holds. Let $\mathcal{M}$ be an $(n+1)$-dimensional rotationally invariant Ricci flow spacetime and let $x\in\mathcal{M}_t$ be a fixed point. Assume that $\mathcal{M}$ is $(\epsilon_{\operatorname{can}}\rho_1(x),t)$-complete and satisfies the $\epsilon_{\operatorname{can}}$-canonical neighborhood assumption at scales $(\epsilon_{\operatorname{can}}\rho_1(x),1)$. Then the parabolic neighborhood $P:=P_{\mathcal{O}}(x,L\rho_1(x))$ is unscathed and the following is true. Let $h$ be a rotationally invariant Ricci-DeTurck perturbation on $P$. Assume that $|h|\leq\eta_{\operatorname{lin}}$ everywhere on $P$ and define the scalar function
\begin{eqnarray}\label{definition of Q}
Q:=e^{H(T-\mathfrak{t})}\rho_1^E|h|
\end{eqnarray}
on $P$, where $T\geq t$ is some arbitrary number. Then
\begin{enumerate}[(1)]
    \item If $t>(L\rho_1(x))^2$ (i.e., if $P$ does not intersect the time-$0$ slice), then we have
    $$Q(x)\leq\frac{1}{100}\sup_PQ.$$
    \item If $t\leq (L\rho_1(x))^2$ (i.e., if $P$ intersects the time-$0$ slice), then we have
    $$Q(x)\leq\frac{1}{100}\sup_PQ+C\sup_{P\cap\mathcal{M}_0}Q.$$
\end{enumerate}
\end{prop}
\begin{proof}
The proof of this proposition is no different from the original version, except for a few technicalities. First of all, there is a universal $\kappa_0=\kappa_0(n)>0$, such that any $(n+1)$-dimensional $\kappa$-solution is also a $\kappa_0$-solution (Theorem \ref{universal kappa}). Secondly, the auxiliary lemma applied in the proof \cite[Lemma 8.10]{BamlerKleiner17} is a consequence of the canonical neighborhood assumption and Perelman's $\kappa$-compactness theorem (c.f. \cite[Theorem 1.2]{LiZhang18}). Therefore, one may follow the arguments in the original proof to reach the required conclusion.
\end{proof}

\begin{prop}[Proposition 9.3 in \cite{BamlerKleiner17}, The Interior Decay Theorem]
If
\begin{gather*}
    E>2,\quad H\geq\underline{H}(E),\quad \eta_{\operatorname{lin}}\leq\bar \eta_{\operatorname{lin}}(E),\quad\alpha>0,
    \quad
    A\geq\underline{A}(E,\alpha),\quad\epsilon_{\operatorname{can}}\leq\bar\epsilon_{\operatorname{can}}(E,\alpha),
\end{gather*}
then there is a constant $C=C(E)<\infty$ such that the following holds. Let $\mathcal{M}$ be an $(n+1)$-dimensional rotationally invariant Ricci flow spacetime and let $x\in\mathcal{M}_t$ be a fixed point. Assume that $\mathcal{M}$ is $(\epsilon_{\operatorname{can}}\rho_1(x),t)$-complete and satisfies the $\epsilon_{\operatorname{can}}$-canonical neighborhood assumption at scales $(\epsilon_{\operatorname{can}}\rho_1(x),1)$. Consider the parabolic neighborhood $P:=P_{\mathcal{O}}(x,A\rho_1(x))$ and let $h$ be a rotationally invariant Ricci-DeTurck perturbation on $P$ such that $|h|\leq \eta_{\operatorname{lin}}$ everywhere. Define $Q$ as (\ref{definition of Q}).  Then
\begin{enumerate}[(1)]
    \item If $t>(A\rho_1(x))^2$ (i.e., if $P$ does not intersect the time-$0$ slice), then we have $$Q(x)\leq\alpha\sup_PQ.$$
    \item If $t\leq (A\rho_1(x))^2$ (i.e., if $P$ intersects the time-$0$ slice), then we have $$Q(x)\leq\alpha\sup_PQ+C\sup_{P\cap\mathcal{M}_0}Q.$$
\end{enumerate}
\end{prop}
\begin{proof}
As the case of the above proposition, the proof of this proposition is the same as the original proof of \cite[Proposition 9.3]{BamlerKleiner17}. Note that the auxiliary lemma \cite[Lemma 8.13]{BamlerKleiner17} applied therein also follows from the canonical neighborhood assumption and Perelman's $\kappa$-compactness theorem (c.f. \cite[Theorem 1.2]{LiZhang18}).
\end{proof}


\section{Bryant extensions}

In this section, we replicate the results in Section 10 of \cite{BamlerKleiner17} in our setting.  Because our Ricci flow spacetimes are rotationally invariant and the comparison maps between our Ricci flow spacetimes are all equivariant, some of the proofs are simpler. The main results of this section concern determining when almost isometries between parts of Bryant solitons can be extended to more comprehensive isometries.
\subsection{The Bryant soliton}

The Bryant soliton is a certain complete steady gradient Ricci soliton on $\mathbb{R}^{n+1}$. If we let $r:\mathbb{R}^{n+1}\to [0,\infty)$ be the Euclidean distance to the origin, then the Bryant soliton metric at a point $x\in \mathbb{R}^{n+1}$ is given by 
\begin{align*}
   g_{Bry}(r)=dr^2+w(r(x))^2g_{\mathbb{S}^{n}}, 
\end{align*}
where $g_{\mathbb{S}^{n}}$ is the standard round metric of Ricci curvature $n-1$ on the level set $r^{-1}(r(x))$, and $w:\mathbb{R}\to \mathbb{R}$ is smooth, odd about $r=0$, has $w'(0)=1$, and $w(r)>0$ for all $r>0$. 
We let $(M_{Bry},g_{Bry})$ denote the resulting Bryant soliton. 
The function $w$ is uniquely determined by the condition that there is a smooth and even function $u:\mathbb{R}\to \mathbb{R}$ so that the $(w,u)$ pair satisfies the gradient shrinking Ricci soliton equations
\begin{align*}
    -n\frac{w''}{w}+u''=0, \qquad -\frac{w''}{w}+(n-1)\frac{1-(w')^2}{w^2}+\frac{u'w'}{w}=0\\
\end{align*}
with $u''(0)=\frac{-1}{n+1}$. 
The scalar curvature at a point $x$ is 
\begin{align*}
    -2n\frac{w''(r)}{w(r)}-n(n-1)\left(\frac{1-w'(r)^2}{w(r)^2}\right);
\end{align*}
this becomes $1$ at the origin. 
The sectional curvatures of a generic rotationally-symmetric metric at a non-origin point $x$ are 
\begin{align*}
    K_{orb}=\frac{1-w'(r)^2}{w(r)^2}, \qquad K_{rad}=-\frac{w''(r)}{w(r)}.
\end{align*}

  In the case of the Bryant soliton, we find the following curvature quantities all converge to positive numbers as $r$ tends to $\infty$ (c.f. Section 4.4 of \cite{ChowI}):
\begin{align}\label{Bryantasym}
\frac{w(r)}{\sqrt{r}}, \qquad K_{orb}r, \qquad K_{rad}r^2.
\end{align}
Also note that for the Bryant soliton, the scalar curvature is maximised at the origin. 
\subsection{Comparisons between `almost' Bryant solitons}
Since the Bryant soliton is rotationally invariant, we find it convenient to define the following subsets of $\mathbb{R}^{n+1}$:
\begin{align*}M_{Bry}(D):=\{x\in M_{Bry};r(x)<D\}, \qquad M_{Bry}(D_1,D_2):=\{x\in M_{Bry};D_1 <r(x)<D_2\}.
\end{align*}
The main result of this section is the following proposition, which concerns the extensions of `almost' isometries of geometries defined on these annular regions of $\mathbb{R}^{n+1}$. 

\begin{prop}\label{BryantExtension}
There is a large $\underline{E}$, as well as constants $\underline{D}(E,C,\beta)$ and $\overline{\delta}(E,C,\beta,D,b)$ so that for each $E\ge \underline{E}$, $C>1$, $\beta>0$, $D\ge \underline{D}$, $0<b\le C$, $0<\delta<\overline{\delta}$ and $D'>0$, the following holds. 
Suppose that $g$ and $g'$ are $O(n+1)$-invariant Riemannian metrics on $M_{Bry}(D)$ and $M_{Bry}(D')$, respectively, such that, for some $\lambda\in [C^{-1},C]$, it holds that $||g-g_{Bry}||_{C^{[\delta^{-1}]}(M_{Bry}(D))}$,  $||\lambda^{-2}g'-g_{Bry}||_{C^{[\delta^{-1}]}(M_{Bry}(D'))}<\delta$. Suppose also that $\Phi:M_{Bry}(\frac{1}{2}D,D)\to M_{Bry}(D')$ is an $O(n+1)$-equivariant open embedding such that $h:=\Phi^{\ast}g'-g$ satisfies $\rho_g ^E |\nabla_g ^m h|_g \leq b$ on $M_{Bry}(\frac{1}{2}D,D)$. Then there is an $O(n+1)$-equivariant open embedding $\widetilde{\Phi}:M_{Bry}(D)\to M_{Bry}(D')$ such that the following holds: $\widetilde{\Phi}=\Phi$ on $M_{Bry}(D-1,D)$ and $\widetilde{h}:=\widetilde{\Phi}^{\ast}g'-g$ satisfies $\rho_g^3|\widetilde{h}|_g \leq \beta b$ on $M_{Bry}(D)$. 
\end{prop}

As in Section 10 of \cite{BamlerKleiner17}, to prove Proposition \ref{BryantExtension} it is convenient to first prove a simpler result: 
\begin{prop}\label{Bryantextensionsimple}
There is a universal constant $C_1>0$ and a function $\underline{D}(\alpha)$ so that if $0<\alpha<1$, $E\ge \underline{E}$, and $D\ge \underline{D}(\alpha)$, then the following holds. 
Let $g_i=\lambda_i^2g_{Bry}$ with $\lambda_1=1$ and $\lambda_2\in [\alpha^{-1},\alpha]$. Let $\Phi:M_{Bry}[\frac{D}{2},D]\to M_{Bry}$ be an $O(n+1)$-equivariant diffeomorphism onto its image, such that  $\left|\nabla_{g_1}^m (\Phi^* g_2-g_1)\right|_{g_1}\le b D^{-E}$ for some $b\le \alpha^{-1}$ and $m=0,1,2,3,4$. Then there is an $O(n+1)$-equivariant diffeomorphism onto its image $\tilde{\Phi}:M_{Bry}(D)\to M_{Bry}$ so that $\tilde{\Phi}=\Phi$ on $M_{Bry}(D-1,D)$, and $\left|\tilde{\Phi}^* g_2-g_1\right|\le b\cdot C_1\alpha^{-C_1}D^{-E+C_1}$.  
\end{prop}
\begin{proof}
It is helpful to note that 
any $O(n+1)$-equivariant diffeomorphism $\Phi$ between $(n+1)$-balls must have $0$ as a fixed point, and on $\mathbb{R}^{n+1}\setminus\{0\}\simeq (0,\infty)\times \mathbb{S}^n$, we have
\begin{align*}
    \Phi(r,x)=(\phi(r), Ax),
\end{align*}
where $\phi:(0,\infty)\to (0,\infty)$ is smooth, and smoothly extendable to an odd function on $\mathbb{R}$, and $A=\pm I$. As in Lemma 10.5 of \cite{BamlerKleiner17}, the proof of this result proceeds by combining the hypothesis of the proposition with the Bryant soliton asymptotics \eqref{Bryantasym} to conclude that the Ricci curvatures of $g_1$ and $\phi^* g_2$ satisfy 
\begin{align*}
    \left|Ric(g_1)-Ric(\phi^* g_2)\right|_{C^1}\le bC_1D^{-E}\quad\text{on}\quad M_{Bry}(\tfrac{D}{2},D)
\end{align*} for some universal constant $C_1$, provided $D$ is large enough. One then combines this estimate with the identities
\begin{align*}
    R_{g_i}+\left|\nabla_{g_{i}}f\right|^2=\lambda_i^{-2}, \qquad dR_{g_{i}}(X)=2Ric_{g_{i}}(\nabla_{g_{i}}f,X), \ i=1,2
\end{align*} to show that 
$\left|\lambda_2-1\right|\le bC_1\alpha^{-C_1}D^{-E+4}$, i.e., the scale of $g_2$ is close to $g_1$.

Now, it is well known that the scalar curvature function $R_{g_{Bry}}(r)$ is monotone decreasing as a function of $r$, and that $rR_{g_{Bry}}(r)$ is convergent to a positive number. In fact, the identity for $dR_{g_{Bry}}$ then also implies that $r^2R_{g_{Bry}}'(r)$ is also convergent to a negative number. Lemma B.1 of \cite{BamlerKleiner17} further implies that the second and third derivatives of $R_{g_{Bry}}(r)$ are uniformly bounded. Therefore, $R_{g_{Bry}}(r)$ admits an inverse $H:(0,R_{g_{Bry}}(0))\to (0,\infty)$ so that 
\begin{align*}
    H'(r)r^2, \qquad H''(r)r^{6}, \qquad H'''(r)r^{10}
\end{align*}
are all uniformly bounded. 

Now the scalar curvature of $g_2$ at a point distance $r$ from the tip is simply $\lambda_2^{-2}R_{g_{Bry}}$. Because $$|\lambda_2^{-2}R_{g_{Bry}}(\phi(r)))-R_{g_{Bry}}(r)|_{C^2}\leq bD^{-E},$$ the estimates for $H,H',H''$ imply $$|H(\lambda_2^{-2}R_{g_{Bry}}(\phi(r)))-r|_{C^2}\leq bD^{-E+C'}.$$ 
Combining this with the estimate for $|\lambda_2-1|$ and also using the bound for $H'''$, we obtain
$\left|\phi(r)-r\right|_{C^2}\le bC_1\alpha^{-C_1}D^{-E+C_1}$ on $[\frac{D}{2},D]$.

We now extend $\phi:(\frac{D}{2},D)\to (0,\infty)$ to $\tilde{\phi}:(0,D)\to (0,\infty)$ by defining $\tilde{\phi}(r)=\zeta_1(r)r+\zeta_2(r)\phi(r)$, where  $\{\zeta_1,\zeta_2\}$ is a partition of unity for $(0,D)$ so that the support of $\zeta_1,\zeta_2$ is contained in $(0,\frac{7D}{8})$ and $(\frac{5D}{8},D)$ respectively. It is clear that, by possibly enlarging $C_1$, that 
\begin{align*}
    \left|\tilde{\phi}(r)-r\right|_{C^2}\le bC_1\alpha^{-C_1}D^{-E+C_1}
\end{align*}
on $(0,D)$. If we define $\tilde{\Phi}(r,x)=(\tilde{\phi}(r),Ax)$, then the result follows from the formula 
\begin{align*}
    \Phi^*(dr^2+w(r(x))^2g_{\mathbb{S}^n})=(\tilde{\phi}'(r))^2dr^2+w(\tilde{\phi}(r))^2g_{\mathbb{S}^n}
\end{align*}
coupled with the estimate $\left|w'(r)\right|\le 1$ for all $r\in (0,\infty)$.
\end{proof}

\begin{proof}[Proof of Proposition  \ref{BryantExtension}]
The main idea behind this proof is to show that the closeness of $g$ and $g'$ to $g_1$ and $g_2$ respectively, coupled with the almost isometry $\Phi$, implies that the hypothesis of Proposition \ref{Bryantextensionsimple} is applicable so that we obtain an extension diffeomorphism so that $\tilde{\Phi}^*g_2-g_1$ is small in $g_1$. A straightforward computation reveals that the required estimates hold for $\tilde{\Phi}^*(\lambda^2 g_2)-g_1$. 
\end{proof}

\section{Construction of the comparison domain}
The proofs of several results in \cite[Section 11]{BamlerKleiner17} can be simplified considerably in the rotationally invariant setting. In particular, we give a new proof of \cite[Lemma 11.17]{BamlerKleiner17} under our assumption, which takes advantage of the classification of noncompact rotationally invariant $\kappa$-solutions in \cite{LiZhang18}. For the results in \cite[Section 11]{BamlerKleiner17} with no substantial simplification, we shall simply state the results under our assumption.

Before we proceed, we remind the reader that the main idea in the simplification is Lemma \ref{nonneckCN=Bryant_Application} above. We note that this key lemma required the construction of a new parameter $D(\dn,\delta)$, but in this section, it will become an auxiliary parameter, and it will not appear in the statement of the theorem below. The main goal of the section is to prove:
\begin{prop}[Proposition 11.1 in \cite{BamlerKleiner17}, extending the comparison domain]\label{comparisondomainextension}
Suppose that
\begin{gather}
    \etalin\leq\overline\eta_{\operatorname{lin}},\quad \dn\leq\overline\delta_{\operatorname{n}},\quad \lambda\leq\overline\lambda(\dn),\quad D_{\operatorname{cap}}\geq\underline{D}_{\operatorname{cap}}(\lambda),\quad \Lambda\geq\underline{\Lambda}(\dn,\lambda),\quad \db\leq\overline\delta_{\operatorname{b}}(\lambda,\Lambda),
    \\\nonumber
    \ecan\leq\barecan(\dn,\lambda,\Lambda,\db),\quad\rcomp\leq\barrcomp(\lambda,\Lambda),
\end{gather}
and assume that
\begin{enumerate}[(i)]
    \item $\mathcal{M}$ and $\mathcal{M}'$ are two $(n+1)$-dimensional, $O(n+1)$-invariant, and $(\ecan\rcomp,T)$-complete Ricci flow spacetimes that each satisfies the $O(n+1)$-equivariant $\ecan$-canonical neighborhood assumption at scales $(\ecan\rcomp,1)$.
    \item $(\mathcal{N},\{\mathcal{N}^j\}_{j=1}^J,\{t_j\}_{j=0}^J)$ is a $O(n+1)$-invariant comparison domain in $\mathcal{M}$ that is defined on the time-interval $[0,t_J]$. We allow the case $J=0$, in which the comparison domain is empty.
    \item $(\operatorname{Cut},\phi,\{\phi^j\}_{j=1}^J)$ is a $O(n+1)$-equivariant comparison from $\mathcal{M}$ to $\mathcal{M}'$ defined on $(\mathcal{N},\{\mathcal{N}^j\}_{j=1}^J,\{t_j\}_{j=0}^J)$ over the (same) time-interval $[0,t_J]$. In the case $J=0$, this comparison is the trivial comparison.
    \item $(\mathcal{N},\{\mathcal{N}^j\}_{j=1}^J,\{t_j\}_{j=0}^J)$ and $(\operatorname{Cut},\phi,\{\phi^j\}_{j=1}^J)$ satisfy the equivariant a priori assumptions (APA1)---(APA6) for parameters $(\etalin,\dn,\lambda,\Dcap,\Lambda,\db,\ecan,\rcomp)$.
    \item $t_{J+1}:=t_J+\rcomp^2\leq T$.
\end{enumerate}

Then there is a $O(n+1)$-invariant subset $\mathcal{N}^{J+1}\subset \mathcal{M}_{[t_J,t_{J+1}]}$ such that $(\mathcal{N}\cup\mathcal{N}^{J+1},\{\mathcal{N}^J\}_{j=1}^{J+1},\{t_j\}_{j=0}^{J+1})$ is a $O(n+1)$-invariant comparison domain defined on the time interval $[0,t_{J+1}]$ and such that $(\mathcal{N}\cup\mathcal{N}^{J+1},\{\mathcal{N}^J\}_{j=1}^{J+1},\{t_j\}_{j=0}^{J+1})$ and $(\operatorname{Cut},\phi,\{\phi^j\}_{j=1}^J)$ satisfy the equivariant a priori assumptions (APA1)---(APA6) for the parameters $(\etalin,\dn,\lambda,\Dcap,\Lambda,\db,\ecan,\rcomp)$.
\end{prop}

In this section, we will always let $s$ be a signed distance function for a component of $\mathcal{M}_{t_{J+1}}$ whose gradient is $g_{t_J}$-orthogonal to the orbits. As in \cite[Section 11]{BamlerKleiner17}, let us consider the collection $\mathcal{S}$ of all embedded $(n+1)$-spheres $\Sigma\subset \mathcal{M}_{t_{J+1}}$ that occur as central orbit of $O(n+1)$-equivariant $\dn$-necks at scale $\rcomp$ in $\mathcal{M}_{t_{J+1}}$.

\begin{lem}[Lemma 11.3 in \cite{BamlerKleiner17}]
We can find a subcollection $\mathcal{S}'\subset\mathcal{S}$ such that
\begin{enumerate}[(a)]
    \item $d_{t_{J+1}}(\Sigma_1,\Sigma_2)>10\rcomp$ for all distinct $\Sigma_1,\Sigma_2\in\mathcal{S}'$.
    \item For every $\Sigma\in\mathcal{S}$ there is an $\Sigma'\in\mathcal{S}'$ such that $d_{t_{J+1}}(\Sigma,\Sigma')<100\rcomp$.
\end{enumerate}
\end{lem}

\begin{lem}[Lemma 11.4 in \cite{BamlerKleiner17}]
If
\begin{eqnarray*}
\dn\leq\bardn,\quad \lambda\leq\bar{\lambda}(\dn),\quad \Lambda\geq\underline{\Lambda}(\dn),\quad \ecan\leq\barecan(\dn),\quad \rcomp<1,
\end{eqnarray*}
then the collection $\mathcal{S}'$ separates the $100\lambda\rcomp$-thin points of $\mathcal{M}_{t_{J+1}}$ from the $\Lambda\rcomp$-thick points.
\end{lem}

\begin{proof}
Let $x$ and $y$ be points in the same components of $\mathcal{M}_{t_{J+1}}$ such that $\rho(x)\leq 100\lambda\rcomp$ and $\rho(y)\geq\Lambda\rcomp$. Without loss of generality, we assume $s(x)<s(y)$. We will show that there exists $t\in(s(x),s(y))$, such that $\{s=t\}$ is a member of $\mathcal{S}'$.

Since $\rho$ can be represented as a continuous function of $s$, we can then find a $t'\in(s(x),s(y))$, such that $\rho(t')=\rcomp$. If $\{s=t'\}$ is not the central sphere of a $\dn$-neck, then, by Lemma \ref{nonneckCN}, we have that all points in, say, $\{s\geq t'\}$ satisfy $C(\dn)^{-1}\rcomp\leq\rho\leq C(\dn)\rcomp$ (note that on each component, a singular orbit is either the maximum or the minimum point of $s$). This set will necessarily contain either $x$ or $y$, and this cannot happen if we take $\lambda\leq\bar{\lambda}(\dn)$ and $\Lambda\geq\underline{\Lambda}(\dn)$.

Hence, it must be that $\{s=t'\}$ is the central sphere of a $\dn$-neck, which is at most $100\rcomp$ away from some $\{s=t\}\in\mathcal{S}'$. Taking $\lambda\leq\bar{\lambda}(\dn)$ and $\Lambda\geq\underline{\Lambda}(\dn)$ again, we have that $x$ and $y$ cannot be in the $\dn$-neck centered at $\{s=t'\}$, that is, $|s(x)-t'|, |s(y)-t'|\geq\dn^{-1}\rcomp$. In consequence, we have $t\in(s(x),s(y))$. This finishes the proof.
\end{proof}

We let $\Omega\subset \mathcal{M}_{t_{J+1}}$ be the union of the closure of all components of $$\mathcal{M}_{t_{J+1}}\setminus\cup_{\Sigma\in\mathcal{S}'}\Sigma$$ that contain $\Lambda\rcomp$-thick points. Then we have that $\Omega$ is $O(n+1)$-equivariant and is weakly $100\lambda\rcomp$-thick.

\begin{lem}[Lemma 11.7 in \cite{BamlerKleiner17}]\label{backwardsurvival}
If $$\ecan\leq\barecan(\lambda),$$ then all points in $\Omega$ survives until time $t_J$.
\end{lem}

\begin{lem}[Lemma 11.8 in \cite{BamlerKleiner17}]\label{B-K Lemma 11.8}
If $J\geq 1$ and $$\dn\leq\bardn,\quad \Lambda\geq\underline\Lambda,\quad \ecan\leq\barecan,\quad \rcomp\leq\barrcomp(\Lambda),$$then for every $\Lambda\rcomp$-thick point $x\in\mathcal{M}_{t_{J+1}}$ we have $x(t_J)\in\operatorname{Int}\mathcal{N}_{t_J-}$.
\end{lem}


\begin{lem}[Lemma 11.9 in \cite{BamlerKleiner17}]
If $J\geq 1$ and
\begin{eqnarray*}
\dn\leq\bardn,\quad \lambda\leq\bar{\lambda}(\dn), \quad \Lambda\geq\underline{\Lambda}(\dn), \quad \ecan\leq\barecan(\dn), \quad \rcomp\leq\barrcomp(\Lambda),
\end{eqnarray*}
then $\Omega(t_J)\subset\operatorname{Int}\mathcal{N}_{t_J-}$.
\end{lem}

\begin{proof}
We argue by contradiction. Suppose this lemma is not true. Let $\Omega_0$ be a component of $\Omega$, such that the conclusion fails for $\Omega_0$. Then we can find a $\Lambda\rcomp$-thick point $x\in\Omega_0$, and $y\in\Omega_0$ satisfying $y(t_J)\in\partial \mathcal{N}_{t_J-}$. Note that such $y$ always exist, since by Lemma \ref{backwardsurvival}, $\Omega_0(t_J)$ is connected, and one can connect $x(t_J)$ with a point in $\Omega_0(t_J)\setminus\operatorname{Int}\mathcal{N}_{t_J-}$ by a curve lying in $\Omega_0$. This curve must intersect $\partial \mathcal{N}_{t_J-}$, and the intersection can be taken to be $y(t_J)$. Since $y(t_J)$ is the center of a $\dn$-neck at scale $\rcomp$, we can argue as in the proof of \cite[Lemma 11.9]{BamlerKleiner17} and apply Lemma \ref{B-K Lemma 8.32} and Lemma \ref{B-K Lemma 8.8}, obtaining $\rho(y)<0.7\rcomp$ if we take
\begin{eqnarray*}
\dn\leq\bardn,\quad \lambda\leq\bar{\lambda}(\dn),\quad \rcomp\leq\barrcomp.
\end{eqnarray*}

Without loss of generality, we assume $s(x)<s(y)$. Let us choose $t'\in(s(x),s(y))$, such that $\rho(t')=\rcomp$ and $\rho<\rcomp$ on $(t',s(y)]$. We will next show that $\{s=t'\}$ cannot be the central sphere of a $\dn$-neck. 

Arguing by contradiction, assume that $\{s=t'\}$ is the central sphere of a $\dn$-neck, then there exists $\{s=t_0\}\in\mathcal{S}'$ at most $100\rcomp$ away from $\{s=t'\}$; we also would like to assume that $\{s=t_0\}$ is the member of $\mathcal{S}'$ closest to $\{s=t'\}$. Hence $\{s=t_0\}\subset \Omega_0$. We then study two different cases, based on the behavior of the function $\rho$ on the set $\{s\geq t_0\}$.
\begin{enumerate}[(1)]
\item $\rho$ drops below $100\lambda\rcomp$ at some point. Define 
$$t_1=\sup\{t>t_0; \rho(\overline{t})> 100\lambda\rcomp \text{ for all } \overline{t}\in [t_0,t)\} >s(y).$$ 
There are two subcases.
\begin{enumerate}[($i$).]
\item If there is some $t_2\in(s(y),t_1)$, such that $\rho(t_2)\geq\Lambda\rcomp$. (Note that by our choice of $t_0$, it cannot happen that $t_2\in(t_0,s(y))$.) By Lemma \ref{B-K Lemma 11.8}, we have that both $\{s=s(x)\}(t_J)$ and $\{s=t_2\}(t_J)$ are in $\operatorname{Int}\mathcal{N}_{t_J-}$, and all orbits between them must be $50\lambda\rcomp$-thick according to Lemma \ref{B-K Lemma 8.8}. This, according to (APA3)(d), then implies there can be no component of $\mathcal{M}_{t_J}\setminus\mathcal{N}_{t_J-}$ between these two orbits, and $y(t_J)\in\operatorname{Int}\mathcal{N}_{t_J-}$. This is a contradiction.
\item If $\rho(s)<\Lambda\rcomp$ for all $s\in(s(y),t_1)$, then $\{s=t_0\}$ should have been a boundary component of $\Omega$. The reason is as follows. Let $t_2>t_0$ be the smallest number such that $\{s=t_2\}\in\mathcal{S}'$. We further have the following two cases.
\begin{enumerate}[($a$).] 
\item If no such $t_2$ exists, then we have $\rho<\Lambda\rcomp$ on $\{s>t_0\}$. This is because there exists a $100\lambda\rcomp$-thin point on $\{s>t_0\}$ while in this same set there is no member of $\mathcal{S}'$ to separate it from $\Lambda\rcomp$-thick points. In this case the whole $\{s>t_0\}$ should be discarded. 
\item On the other hand, if such $t_2$ does indeed exist, then there should be no $\Lambda\rcomp$-thick points appearing in $\{t_0<s<t_2\}$. For if there exists $t\in(t_0,t_2)$, such that $\{s=t\}$ is $\Lambda\rcomp$-thick, then $t>t_1$, and $\{s=t_1\}$ (being $100\lambda\rcomp$-thin) and $\{s=t\}$ (being $\Lambda\rcomp$-thick) must be separated by some elements in $\mathcal{S}'$, and hence $\{s=t_2\}$ cannot be the first member of $\mathcal{S}'$ which we see in $\{s>t_0\}$. In this case, $\{t_0<s<t_2\}$ is a discarded component.  
\end{enumerate}
This contradicts our choice of $t_0$.
\end{enumerate}
\item $\rho(s)\geq 100\lambda\rcomp$ for all $\{s\geq t_0\}$. If this ever happens, then $s$ will attain its maximum only at a singular orbit, and $\{s\geq s(y)\}\subset\{s\geq t_0\}$ is topologically a disk, completely contained in $\Omega_0$. Then we have $\{s\geq s(y)\}(t_J)$ is a $50\lambda\rcomp$-thick disk. This cannot sustain any component of $\mathcal{M}_{t_J}\setminus\mathcal{N}_{t_J-}$ by (APA3)(d).
\end{enumerate}

Since $\{s=t'\}$ is not the center of a $\dn$-neck, we may then assume
\begin{eqnarray*}
\dn\leq\bardn,\quad \delta_\#\leq\bar{\delta}_\#(\dn),\quad \Lambda\geq\underline{\Lambda}(\dn,\delta_\#),\quad \ecan\leq\barecan(\dn,\delta_\#,\Lambda)
\end{eqnarray*}
such that there is an $O(n+1)$-equivariant diffeomorphism onto its image $\psi: \Bry{} \supset B(x_{\operatorname{Bry}},\delta_\#^{-1})\rightarrow \mathcal{M}_t$ and a positive number $a\in[c(\dn),2]$ satisfying the following:
\begin{gather*}
\big|g_{\operatorname{Bry}}-(a\rcomp)^{-2}\psi^*g\big|_{C^{\lfloor\delta_\#^{-1}\rfloor}}\leq\delta_\#,
\\
\operatorname{dist}_g(\psi(x_{\operatorname{Bry}}),x)\leq C(\dn)\rcomp,
\\
s(\psi(x_{\operatorname{Bry}}))\geq t',
\end{gather*}
where the last inequality is because the scale of $x$ is too large, and hence cannot be on the side of the singular orbit. It then follows that $\{s\geq s(y)\}\subset\{s\geq t'\}$ and on this set $\rho\geq c(\dn)\rcomp\geq 100\lambda\rcomp$ if we take $\lambda\leq\bar{\lambda}(\dn)$. By Lemma 8.8 again, $\{s\geq s(y)\}(t_J)$ is a $50\lambda\rcomp$-thick topological disk and hence cannot contain any component of $\mathcal{M}_{t_J}\setminus\mathcal{N}_{t_J-}$. This finishes the proof.
\end{proof}

 Recall that Bamler-Kleiner \cite[Section 11.3]{BamlerKleiner17} classified the components $Z$ of $\mathcal M_{t_{J+1}}\setminus\operatorname{Int}\Omega$ into four types ((I)---(IV)). We make the same classification here:
 \begin{enumerate}[(I)]
     \item $Z$ has non-empty boundary and all points on Z are weakly $10\lambda\rcomp$-thick (in particular, $Z$ is not a closed component of $M_{t_{J+1}}$).
     \item \begin{enumerate}
         \item $Z$ is diffeomorphic to an $(n+1)$-disk by an $O(n+1)$-equivariant diffeomorphism.
         \item $Z(t)$ is well-defined and $\lambda\rcomp$-thick for all $t\in[t_J,t_{J+1}]$.
         \item $Z(t_J)\subset\mathcal N_{t_J-}$ if $J\geq 1$.
     \end{enumerate}
     \item \begin{enumerate}
         \item $Z$ is diffeomorphic to an $(n+1)$-disk by an $O(n+1)$-equivariant diffeomorphism.
         \item $Z(t)$ is well-defined and $\lambda\rcomp$-thick for all $t\in[t_J,t_{J+1}]$.
         \item $\mathcal C:=Z(t_J)\setminus\operatorname{Int}\mathcal N_{t_J-}$ is a component of $\mathcal M_{t_J}\setminus\operatorname{Int}\mathcal N_{t_J-}$, and there is a component $\mathcal C'\subset\mathcal M'_{t_J}\setminus\phi(\operatorname{Int}\mathcal N_{t_J-})$ such that a priori assumptions (APA5)(a)--(e) hold, that is:
         \begin{itemize}
             \item $\mathcal C$ and $\mathcal C'$ are $3$-disks.
             \item $\partial\mathcal C'=\phi_{t_J-}(\partial\mathcal C)$.
             \item There is a point $x\in\mathcal C$ such that $(\mathcal M_{t_J},x)$ is $\db$-close to the pointed Bryant soliton $(M_{\operatorname{Bry}},g_{\operatorname{Bry}},x_{\operatorname{Bry}})$ at scale $10\lambda\rcomp$ via an $O(n+1)$-equivariant map.
             \item There is a point $x'\in\mathcal M'_{t_J}$, at distance $\le\Dcap\rcomp$ from $\mathcal C'$, such that $(\mathcal M'_{t_J},x')$ is $\db$-close to the pointed Bryant soliton $(M_{\operatorname{Bry}},g_{\operatorname{Bry}},x_{\operatorname{Bry}})$ at some scale in ther interval $[\Dcap^{-1}\rcomp,\Dcap\rcomp]$ via an $O(n+1)$-equivariant map.
             \item $\mathcal C$ and $\mathcal C'$ have diameter $\le\Dcap\rcomp$.
         \end{itemize}
     \end{enumerate}
      \item None of the above.
 \end{enumerate}
Note that, in addition to the original definitions of the above four types of connected components, we also require that all the $(n+1)$-disks are $O(n+1)$-invariant, and that all maps involved are $O(n+1)$-equivariant. The most substantial simplification is the following proof of \cite[Lemma 11.7]{BamlerKleiner17}.  Once this is established, the proof of Proposition \ref{comparisondomainextension} proceeds as in \cite{BamlerKleiner17}, where we choose $\mathcal{N}_{J+1}$ to be the union of $\Omega$ with components of type (I)-(III).

\begin{lem}[Lemma 11.7 in \cite{BamlerKleiner17}]
If
\begin{gather*}
\etalin\leq\baretalin,\quad \dn\leq\bardn,\quad \lambda\leq\bar{\lambda},\quad \Dcap\geq\uDcap(\lambda),\quad \Lambda\geq\underline{\Lambda}(\lambda),
\\
\ecan\leq\barecan(\lambda,\Lambda,\db),\quad \rcomp\leq\barrcomp(\lambda),\quad \db\leq\bardb(\lambda),
\end{gather*}
then the following holds:

If $Z\subset\mathcal{M}_{t_{J+1}}\setminus\operatorname{Int}\Omega$ is a component of type (I), then either $Z(t_J)\subset\mathcal{N}_{t_J-}$, or $Z$ is also of type (III).
\end{lem}

\begin{proof}
Assume $Z\subset\mathcal{M}_{t_{J+1}}\setminus\operatorname{Int}\Omega$  is a component such that $Z(t_J)\not\subset\mathcal{N}_{t_J-}$. Since $\partial Z(t_J)\subset\operatorname{Int}\mathcal{N}_{t_J-}$, we have that $Z(t_J)$ must contain at least one component of $\mathcal{M}_{t_J}\setminus\mathcal{N}_{t_J-}$, and hence also contains a $10\lambda\rcomp$-thin point. Hence, there exists a $10\lambda\rcomp$-thick point $x\in Z$, while $x(t_J)$ becomes $10\lambda\rcomp$-thin.

In the course of proof, we shall introduce some auxiliary parameters
\begin{eqnarray*}
D_\#, \quad \delta_\#, 
\end{eqnarray*}
whose relation with other parameters will be determined later. Note that $\delta_\#$ will be the last parameter determined before $\ecan$.  Taking $\ecan\leq\barecan(\lambda, \delta_\#)$, we may apply Lemma \ref{B-K Lemma 8.40} to $x$ with $\alpha=10\lambda$, $\delta=\delta_\#$, $\beta=0$, $J=1$, $r=\rcomp$, and $a=10\lambda\rcomp$. Thus, there is a $O(n+1)$-equivariant time-preserving and $\partial_{\mathfrak{t}}$-preserving diffeomorphism
\begin{eqnarray*}
\psi: \overline{M_{\operatorname{Bry}}(\delta_\#^{-1})}\times[-(10\lambda)^{-2},0]\rightarrow W\subset \mathcal{M},
\end{eqnarray*}
such that $\psi(x_{\operatorname{Bry}},0)=x$, and
\begin{eqnarray*}
\big\|(10\lambda\rcomp)^{-2}\psi^*g-g_{\operatorname{Bry}}\big\|_{C^{\lfloor\delta_\#^{-1}\rfloor}}<\delta_\#.
\end{eqnarray*}
By taking $\delta_\#\leq\bar{\delta}_\#(\lambda,\Lambda)$, we may ensure that $W_{t_{J+1}}\cap\Omega\neq\emptyset$ and $W_{t_J}\cap\operatorname{Int}\mathcal{N}_{t_J-}\neq\emptyset$, and the Bryant slice lemma (Lemma \ref{B-K Lemma 8.41}) and the Bryant slab lemma (Lemma \ref{B-K Lemma 8.42}) give us clear geometric characterization of $\mathcal{C}_1=W_{t_{J+1}}\setminus\Omega$ and $\mathcal{C}_0=W_{t_J}\setminus \operatorname{Int}\mathcal{N}_{t_J-}$. Moreover, if $\Dcap \geq \uDcap(\lambda)$, then by the Lemma \ref{B-K Lemma 8.41}, $\mathcal{C}_0$ is an $O(n+1)$-invariant $(n+1)$-disk with diameter no greater than $ \Dcap \rcomp$. (APA5)(c) is also satisfied for $\mathcal C_0$ if we take $\delta_\#\leq\overline{\delta}_\#(\db)$. 

Let us then proceed to find $\mathcal{C}'$ as in (APA5). To this end, we observe that if we take $\delta_\#\leq\bar{\delta}_\#(\lambda)$, then there are points $y, z\in W_{t_J}\cap\operatorname{Int}\mathcal{N}_{t_J-}$, satisfying 
\begin{eqnarray*}
\operatorname{dist}(y,\mathcal{C}), \quad \operatorname{dist}(z,\mathcal{C})\leq C(\lambda)\rcomp,
\\
\rho(y)\geq 40C_{\operatorname{SD}}^2\rcomp,\quad \rho(z)\leq 2\rcomp,
\end{eqnarray*}
where $C_{\operatorname{SD}}$ is the constant from the bi-Lipschitz scale distortion lemma (Lemma \ref{B-K Lemma 8.22}). By Lemma \ref{B-K Lemma 8.22}, we have 
\begin{eqnarray*}
\operatorname{dist}\big(\phi(y),\phi(\partial\mathcal{C})\big), \quad \operatorname{dist}\big(\phi(z),\phi(\partial\mathcal{C})\big)\leq C(\lambda)\rcomp,
\\
\rho(\phi(y))\geq 20C_{SD}\rcomp,\quad \rho(\phi(z))\leq 4C_{SD}\rcomp.
\end{eqnarray*}
It then follows that there exists $\delta_o(\lambda)>0$ and some  $x'\in\phi(\mathcal{N}_{t_J-})$ within the $C(\lambda)\rcomp$-neighborhood of $\phi(\partial\mathcal{C})$, such that $x'$ is not the center of a $\delta_o(\lambda)$-neck, and such that $\rho(x')\in[C_{SD}\rcomp, 40C_{SD}\rcomp]$. We then take
\begin{eqnarray*}
\db\leq\bardb(\delta_o),\quad D_\#\geq\underline{D}(\db,\delta_o), \quad \ecan\leq\barecan(\db,D_\#,\delta_o),
\end{eqnarray*}
such that Lemma \ref{nonneckCN=Bryant_Application} holds for $\delta_{\operatorname{n}}=\delta_o$, $D=D_\#$, and $\delta=\db$. Furthermore, we take
\begin{eqnarray*}
\delta_\#\leq\bar{\delta}_\#(\db, D_\#,\delta_o),
\end{eqnarray*}
such that $W_{t_J}$ contains a $100C_{SD}^2D_\#\rcomp$-thick point $x''$. Moreover, taking $D_\# \geq D_\#(\Lambda)$, we can assume $100C_{SD}^2D_\# \rcomp \geq \Lambda \rcomp$, so that $x''\in \mathcal{N}_{t_J -}$. Then $\phi(x'')$ is $80C_{SD}D_\#\rcomp\geq 2D_\#\rho(x')$-thick and is on the same component of $\mathcal{M}'_{t_J}$ as $x'$.

Lemma \ref{nonneckCN=Bryant_Application} then implies the existence of an $O(n+1)$-equivariant diffeomorphism
$$\psi':\overline{M_{\operatorname{Bry}}(\db^{-1})}\rightarrow W'\subset \mathcal{M}'_{t_J}$$ 
and a scale $a\in [c(\delta_0)\rho(x'),2\rho(x')]$ satisfying the following:
\begin{eqnarray*}
\big\|a^{-2}\psi'^*g-g_{\operatorname{Bry}}\big\|_{C^{\lfloor\db^{-1}\rfloor}}<\db,
\\
d\big(\psi'(x_{\operatorname{Bry}}),x'\big)\leq C(\delta_o)\rho(x').
\end{eqnarray*}
Taking $\db\leq\bardb(\lambda)$, we may make sure that $W'$ contains $\phi(\partial\mathcal{C}_0)$, which is an orbit with scale no greater than $2C_{SD}\rcomp$. Let $\mathcal{C}'$ be the component of $W'\setminus \phi(\partial \mathcal{C}_0)$ containing the singular orbit. Since $\phi(W\setminus\mathcal{C}_0)$ contains large scale points, it follows that this image and the Bryant-like tip must lie on the opposite sides of $\phi(\partial\mathcal{C}_0)$, hence the diameter of $\mathcal{C}'$ is no greater than $ C(\lambda) \rcomp$. Taking into consideration $d\big(\psi'(x_{\operatorname{Bry}}),x'\big)\leq C(\delta_o)\rho(x')$, we have that (APA5)(d)(e) hold so long as we take $\Dcap\geq\uDcap(\lambda)$.
\end{proof}

\section{Construction of the comparison map}

After the previous section, the comparison domain $(\mathcal{N},\{\mathcal{N}^j\}_{j=1}^{J+1},\{t_j\}_{j=0}^{J+1})$ has already been extended forward by one step. The next inductive step is to extend the comparison $(\operatorname{Cut},\phi,\{\phi^j\}_{j=1}^J)$ one step forward under suitable a priori assumptions. We shall collect the statements from \cite{BamlerKleiner17} concerning this construction.

\begin{prop}[Proposition 12.1 in\cite{BamlerKleiner17}, extending the comparison map by one step]\label{BK-prop 12.1}
Suppose that
\begin{gather}
    T>0,\quad E>\underline{E},\quad H\geq \underline{H}(E),\quad\eta_{\operatorname{lin}}\leq \overline \eta_{\operatorname{lin}}(E),\quad \nu\leq\overline\nu(T,E,H,\eta_{\operatorname{lin}}),\quad\delta_{\operatorname{n}}\leq\overline \delta_{\operatorname{n}}(T,E,H,\eta_{\operatorname{lin}}),
    \\\nonumber
    \lambda\leq\overline\lambda,\quad D_{\operatorname{cap}}>0,\quad \eta_{\operatorname{cut}}\leq \overline\eta_{\operatorname{cut}},\quad D_{\operatorname{cut}}\geq \underline{D}_{\operatorname{cut}}(T,E,H,\eta_{\operatorname{lin}},\lambda, D_{\operatorname{cap}},\eta_{\operatorname{cut}}), 
    \\\nonumber
    W\geq \underline{W}(E,\lambda, D_{\operatorname{cut}}),\quad A\geq \underline{A}(E,\lambda,W),\quad \Lambda\geq\underline{\Lambda}(\lambda, A),\quad \delta_{\operatorname{b}}\leq\overline \delta_{\operatorname{b}}(T,E,H,\eta_{\operatorname{lin}},\lambda,D_{\operatorname{cap}},\eta_{\operatorname{cut}},D_{\operatorname{cut}},A,\Lambda),
    \\\nonumber
    \epsilon_{\operatorname{can}}\leq \overline \epsilon_{\operatorname{can}}(T,E,H,\eta_{\operatorname{lin}},\lambda,D_{\operatorname{cap}},\eta_{\operatorname{cut}},D_{\operatorname{cut}},W,A,\Lambda),\quad r_{\operatorname{comp}}\leq \overline r_{\operatorname{comp}}(T,H,\lambda,D_{\operatorname{cut}}).
\end{gather}
and assume that
\begin{enumerate}[(i)]
    \item $\mathcal{M}$ and $\mathcal{M}'$ are two $(n+1)$-dimensional $O(n+1)$-invariant and $(\epsilon_{\operatorname{can}}r_{\operatorname{comp}},T)$-complete  Ricci flow spacetimes that each satisfies the $O(n+1)$-equivariant $\epsilon_{\operatorname{can}}$-canonical neighborhood assumption at scales $(\epsilon_{\operatorname{can}}r_{\operatorname{comp}},1)$.
    \item $(\mathcal{N},\{\mathcal{N}^j\}_{j=1}^{J+1},\{t_j\}_{j=0}^{J+1})$ is a $O(n+1)$-invariant comparison domain in $\mathcal{M}$, which is defined over the time-interval $[0,t_{J+1}]$. We allow the case $J=0$.
    \item $(\operatorname{Cut},\phi,\{\phi^j\}_{j=1}^J)$ is a $O(n+1)$-equivariant comparison from $\mathcal{M}$ to $\mathcal{M}'$ over the time interval $[0,t_{J}]$. If $J=0$, then this comparison is trivial.
    \item $(\mathcal{N},\{\mathcal{N}^j\}_{j=1}^{J+1},\{t_j\}_{j=0}^{J+1})$ and $(\operatorname{Cut},\phi,\{\phi^j\}_{j=1}^J)$ satisfy the equivariant a priori assumptions (APA1)---(APA6) for the parameters $(\eta_{\operatorname{lin}},\delta_{\operatorname{n}},\lambda,D_{\operatorname{cap}},\Lambda,\delta_{\operatorname{b}},\epsilon_{\operatorname{can}},r_{\operatorname{comp}})$ and equivariant a priori assumptions (APA7)---(APA13) for the parameters $(T,E,H,\eta_{\operatorname{lin}},\nu,\lambda,\eta_{\operatorname{cut}},D_{\operatorname{cut}},W,A,r_{\operatorname{comp}})$.
    \item $t_{J+1}\leq T$.
    \item If $J=0$, then we assume in addition the existence of an $O(n+1)$-equivariant map $\zeta: X\rightarrow \mathcal{M}'_0$ with the following properties. First, $X\subset \mathcal{M}_0$ is an $O(n+1)$-invariant open set that contains the $\delta_{\operatorname{n}}^{-1}r_{\operatorname{comp}}$-tubular neighborhood around $\mathcal{N}_0$. Second, $\zeta: X\rightarrow \mathcal{M}'_0$ is an $O(n+1)$-equivariant diffeomorphism onto its image that satisfies the following bounds on $X$:
    \begin{align*}
        \left|\zeta^*g_0'-g_0\right|\leq &\eta_{\operatorname{lin}},
        \\
        e^{HT}\rho_1^E\left|\zeta^*g_0'-g_0\right|\leq &\nu \overline{Q}=\nu\cdot 10^{-E-1}\eta_{\operatorname{lin}}r_{\operatorname{comp}}^E,
        \\
        e^{HT}\rho_1^3\left|\zeta^*g_0'-g_0\right|\leq & \overline{Q}^*=\nu\cdot 10^{-1}\eta_{\operatorname{lin}}(\lambda r_{\operatorname{comp}})^3.
    \end{align*}
    Assume moreover that the $O(n+1)$-equivariant $\ecan$-canonical neighborhood assumption holds at scales $(0,1)$ on the image $\zeta(X)$.
\end{enumerate}
Then, under the above assumptions, there are a set $\operatorname{Cut}^J$ of pairwise disjoint $O(n+1)$-equivariant disks in $\mathcal{M}_{t_J}$, a time-preserving $O(n+1)$-equivariant diffeomorphism onto its image $\phi^{J+1}:\mathcal{N}^{J+1}\rightarrow \mathcal{M}'$ and a continuous map $$\overline{\phi}:\mathcal{N}\setminus \cup_{\mathcal{D}\in \operatorname{Cut}\cup\operatorname{Cut}^J}\mathcal{D}\rightarrow \mathcal{M}'$$ such that the following holds.

The tuple $(\operatorname{Cut}\cup\operatorname{Cut}^J,\overline{\phi},\{\phi^j\}_{j=1}^{J+1})$ is an $O(n+1)$-equivariant comparison from $\mathcal{M}$ to $\mathcal{M}'$ defined on $(\mathcal{N},\{\mathcal{N}\}_{j=1}^{J+1},\{t_j\}_{j=1}^{J+1})$ over the time-interval $[0,t_{J+1}]$. This comparison and the corresponding domain still satisfy equivariant a priori assumptions (APA1)---(APA6) for the parameters $(\eta_{\operatorname{lin}},\delta_{\operatorname{n}},\lambda,D_{\operatorname{cap}},\Lambda,\delta_{\operatorname{b}},\epsilon_{\operatorname{can}},r_{\operatorname{comp}})$ and equivariant a priori assumptions (APA7)---(APA13) for the parameters $(T,E,H,\eta_{\operatorname{lin}},\nu,\lambda,\eta_{\operatorname{cut}},D_{\operatorname{cut}},W,A,r_{\operatorname{comp}})$.

Lastly, in the case $J=0$ we have $\phi_0^1=\zeta
\vert_{\mathcal{N}_0^1}$.
\end{prop}

The proof of the above proposition consists of three steps. First of all, by virtue of the Bryant extension theorem, $\phi^J_{t_J}$ can be extended over the extension caps, and this provides the initial data for solving $\overline{\phi}$. Secondly, With this initial data, one can solve the harmonic map heat flow on the Ricci flow spacetime background. By grafting a shrinking cylinder on each component of the space boundary of $\mathcal{N}^{J+1}$, the boundary issue can be avoided. This harmonic map heat flow exists as long as the Ricci-DeTurck perturbation is small. Thirdly, by adjusting the parameters, the Ricci-DeTurck perturbation can be ensured to be small as far as it goes, and this implies that the harmonic map heat flow can be solved all the way until $t=t_{J+1}$. This solution is $\overline{\phi}$. We shall include below the statements of the propositions related to these steps. 

\begin{prop}[Proposition 12.3 in \cite{BamlerKleiner17}, extending the comparison over the extension caps]\label{BK-prop 12.3}
Suppose that
\begin{gather}
    E\geq\underline{E},\quad \eta_{\operatorname{lin}}\leq\overline\eta_{\operatorname{lin}},\quad \delta_{\operatorname{n}}\leq\overline\delta_{\operatorname{n}},\quad\lambda\leq\overline\lambda,\quad D_{\operatorname{cut}}\geq \underline{D}_{\operatorname{cut}}(T,E,H,\eta_{\operatorname{lin}},\lambda,D_{\operatorname{cap}},\eta_{\operatorname{cut}}),\quad \Lambda\geq\underline{\Lambda}
    \\\nonumber
    \delta_{\operatorname{b}}\leq\overline\delta_{\operatorname{b}}(T,E,H,\eta_{\operatorname{lin}},\lambda,D_{\operatorname{cap}},\eta_{\operatorname{cut}},D_{\operatorname{cut}},A,\Lambda),\quad\epsilon_{\operatorname{can}}\leq\overline\epsilon_{\operatorname{can}}(T,E,H,\eta_{\operatorname{lin}},\lambda,D_{\operatorname{cap}},\eta_{\operatorname{cut}},D_{\operatorname{cut}},A,\Lambda),
    \\\nonumber
    r_{\operatorname{comp}}\leq\overline r_{\operatorname{comp}}(T,H,\lambda, D_{\operatorname{cut}})
\end{gather}
and assume that assumptions (i)---(v) in Proposition \ref{BK-prop 12.1} hold and that $J\geq 1$.

Then there is a set of cuts $\operatorname{Cut}^J$ at time $t_J$, i.e., a family of $O(n+1)$-invariant $(n+1)$-disks in $\operatorname{Int} \mathcal{N}_{t_J+}$, and a rotationally equivariant diffeomorphism onto its image $\widehat{\phi}:\mathcal{N}_{t_J-}\cup\mathcal{N}_{t_J+}\to \mathcal{M}'_{t_J}$, such that the following hold:
\begin{enumerate}[(a)]
    \item Each $\mathcal{D}\in\operatorname{Cut}^J$ contains exactly one extension cap of the comparison domain $(\mathcal{N},\{\mathcal{N}^j\}_{j=1}^{J+1},\{t_j\}_{j=0}^{J+1})$ and each extension cap of this comparison domain that is in $\mathcal{M}_{t_J}$ is contained in one $\mathcal{D}\in\operatorname{Cut}^J$.
    \item $\widehat\phi=\phi_{t_J-}$ on $\mathcal{N}_{t_J-}\setminus \cup_{\mathcal{D}\in\operatorname{Cut}^J}\mathcal{D}$.
    \item Every cut $\mathcal{D}\in\operatorname{Cut}^J$ has diameter $<D_{\operatorname{cut}}r_{\operatorname{comp}}$ and contains a $\tfrac{1}{10}D_{\operatorname{cut}}r_{\operatorname{comp}}$-neighborhood of the corresponding extension cap in $\mathcal{D}$.
    \item The associated perturbation $\widehat h:=\widehat\phi^* g'_{t_J}-g_{t_J}$ satisfies $|\widehat h|\leq\eta_{\operatorname{lin}}$ on $\mathcal{N}_{t_J-}\cup\mathcal{N}_{t_J+}$ and $$e^{H(T-t_J)}\rho_1^3|\widehat h|\leq\eta_{\operatorname{cut}}\overline{Q}^*$$ on each $\mathcal{D}\in \operatorname{Cut}^J$.
    \item The $O(n+1)$-equivariant $\epsilon_{\operatorname{can}}$-canonical neighborhood assumption holds at scales $(0,1)$ on the image $\widehat\phi(\mathcal{N}_{t_J-}\cup\mathcal{N}_{t_J+}).$
\end{enumerate}
\end{prop}

\begin{prop}[Proposition 12.22 in \cite{BamlerKleiner17}, extending the comparison map until we lose control]\label{BK-prop 12.22}
If
\begin{gather}
    E>2,\quad F>0,\quad H\geq\underline{H}(E),\quad\eta_{\operatorname{lin}}\leq\overline\eta_{\operatorname{lin}}(E),\quad \nu\leq\overline\nu(T,E,F,H,\eta_{\operatorname{lin}}),\\\nonumber
    \delta_{\operatorname{n}}\leq\overline\delta_{\operatorname{n}}(T,E,F,H,\eta_{\operatorname{lin}}),\quad 
    \lambda\leq\overline\lambda,\quad \delta_{\operatorname{b}}\leq\overline\delta_{\operatorname{b}}(T,E,F,H,\eta_{\operatorname{lin}},\lambda,D_{\operatorname{cut}},A,\Lambda),
    \\\nonumber
    \epsilon_{\operatorname{can}}\leq\overline\epsilon_{\operatorname{can}}(T,E,F,H,\eta_{\operatorname{lin}},\lambda,D_{\operatorname{cut}},A,\Lambda),\quad r_{\operatorname{comp}}\leq\overline r_{\operatorname{comp}},
\end{gather}
then the following holds.

Assume that assumptions (i)---(vi) in Proposition \ref{BK-prop 12.1} hold.

Recall that in the case $J=0$, assumption (vi) imposes the existence of an $O(n+1)$-invariant domain $X\subset\mathcal{M}_{t_J}$ and an $O(n+1)$-equivariant map $\zeta:X\to \mathcal{M}'_{t_J}$ with certain properties. In this case we set $\widehat \phi:=\zeta$.

In the case $J\geq 1$, we set $X:=\mathcal{N}_{t_J-}\cup\mathcal{N}_{t_J+}$ and consider the set $\operatorname{Cut}^J$ and the map $\widehat\phi: X\to \mathcal{M}'_{t_J}$ satisfying all assertions of Proposition \ref{BK-prop 12.3}.

Then there is some $t^*\in(t_J,t_{J+1}]$ and a smooth, time-preserving, $O(n+1)$-equivariant diffeomorphism onto its image $\phi^{J+1}:\mathcal{N}^{J+1}_{[t_J,t^*]}\to\mathcal{M}'$ with $\phi^{J+1}_{t_J}=\widehat\phi\vert_{\mathcal N_{t_J+}}$ whose inverse $(\phi^{J+1})^{-1}:\phi^{J+1}(\mathcal{N}^{J+1}_{[t_J,t^*]})\to \mathcal{N}^{J+1}_{[t_J,t^*]}$ evolves by harmonic map heat flow and such that the following hold for the associated perturbation $h^{J+1}:=(\phi^{J+1})^*g'-g$ (which is Ricci-DeTurck flow):
\begin{enumerate}[(a)]
    \item $|h^{J+1}|\leq 10\eta_{\operatorname{lin}}$ on $\mathcal{N}^{J+1}_{[t_J,t^*]}$.
    \item For any $t\in[t_J,t^*]$ and $x\in\mathcal{N}^{J+1}_t$ whose time-$t$ distance to $\partial \mathcal{N}^{J+1}_t$ is smaller than $Fr_{\operatorname{comp}}$ we have $$Q_+(x)=e^{H(T-\mathfrak{t}(x))}\rho_1^E(x)|h^{J+1}(x)|<\overline{Q}=10^{-E-1}\eta_{\operatorname{lin}}r_{\operatorname{comp}}^E.$$
    \item If even $|h^{J+1}|\leq \etalin$ on $\mathcal{N}^{J+1}_{t^*}$, then $t^*=t_{J+1}$.
    \item $\phi^{J+1}(\mathcal{N}^{J+1}_{[t_J,t^*]})$ is $\epsilon_{\operatorname{can}}r_{\operatorname{comp}}$-thick.
\end{enumerate}
\end{prop}

To explain the last of the three steps, we shall first of all consider a relaxed version of (APA7)---(APA9). The precise forms of these relaxed assumptions are presented in (\ref{rapa7})---(\ref{rapa9}), and they are weaker than the original a priori assumptions since the constants involved are larger. Then, whenever the comparison is known to satisfy the original version of these a priori assumptions up to some time $t'\in[t_J,t^*]$, it can be extended up to $\min\{t'+\tau,t^*\}$ for some uniform $\tau>0$ while satisfying the relaxed version of the a priori assumptions. Finally, by chiefly applying the Interior Decay Theorem \cite{BamlerKleiner17}, we show that the comparison satisfies the original version of the a priori assumptions so long as it satisfies the relaxed version. It then follows that the comparison exists up to $t_{J+1}$ and satisfies the original version of the a priori assumptions. 

We consider the comparison $(\operatorname{Cut}\cup\operatorname{Cut}^J,\overline{\phi},\{\phi^j\}_{j=0}^{J+1})$ constructed in the above proposition and let $(h,\{h^j\}_{j=1}^{J+1})$ be the associated Ricci-DeTurck perturbation. Let $Q$, $Q^*$, $\overbar{Q}$, and $\overbar{Q}^*$ be as defined before. Choose the maximum time $t^{**}\in[t_J,t^*]$ such that the following conditions hold for all $x\in\mathcal{N}^{J+1}_{[t_J,t^{**}]}\setminus\cup_{\mathcal{D}\in\operatorname{Cut}^J}\mathcal{D}$:
\begin{align}\label{rapa7}
    & Q(x)\leq 10\overbar{Q} \quad \text{whenever } P_{\mathcal{O}}(x,10A\rho_1(x))\cap\mathcal{D}=\emptyset\text{ for all } \mathcal{D}\in\operatorname{Cut}\cup\operatorname{Cut}^J,
    \\
    & Q(x)\leq 10W\overbar{Q},\label{rapa8}
    \\
    & Q^*(x)\leq 10\overbar{Q}^* \quad \text{ whenever } B_{\mathcal{O}}(x,10A\rho_1(x))\subset \mathcal{N}_{\mathfrak{t}(x)-}.\label{rapa9}
\end{align}
The following lemma then implies that, if the comparison satisfies the original version of the a priori assumptions, then it can be extended by a short time while keeping the relaxed version of a priori assumptions satisfied.

\begin{lem}[Lemma 12.47 in \cite{BamlerKleiner17}]\label{conitunation}
If
\begin{gather*}
    E\geq\underline{E},\quad F\geq\underline{F},\quad \etalin\leq\overline \eta_{\operatorname{lin}},\quad \delta_{\operatorname{n}}\leq \overline\delta_{\operatorname{n}},\quad\lambda\leq\overline\lambda,\quad\eta_{\operatorname{cut}}\leq\overline \eta_{\operatorname{cut}},\quad D_{\operatorname{cut}}\geq\underline{D}_{\operatorname{cut}}(\lambda),\quad W\geq\underline{W}(E,\lambda,D_{\operatorname{cut}}),
    \\
    A\geq \underline{A},\quad \Lambda\geq \underline{\Lambda},\quad\delta_{\operatorname{b}}\leq\overline\delta_{\operatorname{b}}(\lambda,D_{\operatorname{cut}},A,\Lambda),\quad \epsilon_{\operatorname{can}}\leq\overline\epsilon_{\operatorname{can}}(\lambda,D_{\operatorname{cut}},A,\Lambda),\quad r_{\operatorname{comp}}\leq\overline r_{\operatorname{comp}},
\end{gather*}
then we may choose $t^{**}>t_J$.

Furthermore, there is a constant $\tau=\tau(T,E,H,\eta_{\operatorname{lin}},\lambda,A,r_{\operatorname{comp}})>0$ with the following property. If a priori assumptions (APA7)---(APA9) hold up to $t'\in[t_J,t^*]$, meaning that for all $x\in \mathcal{N}^{J+1}_{[t_J,t']}\setminus\cup_{\mathcal{D}\in\operatorname{Cut}^J}\mathcal{D}$
\begin{align}
    & Q(x)\leq \overbar{Q} \quad \text{whenever } P_{\mathcal{O}}(x,A\rho_1(x))\cap\mathcal{D}=\emptyset\text{ for all } \mathcal{D}\in\operatorname{Cut}\cup\operatorname{Cut}^J,
    \\
    & Q(x)\leq W\overbar{Q},
    \\
    & Q^*(x)\leq \overbar{Q}^* \quad \text{ whenever } B_{\mathcal{O}}(x,A\rho_1(x))\subset \mathcal{N}_{\mathfrak{t}(x)-},
\end{align}
then (\ref{rapa7})---(\ref{rapa9}) hold for all $x\in \mathcal{N}^{J+1}_{[t_J,\min\{t'+\tau,t^*\}]}\setminus\cup_{\mathcal{D}\in\operatorname{Cut}^J}\mathcal{D}$.
\end{lem}

The next several results verify that by arranging the parameters properly, then (APA6)---(APA9) hold on $\mathcal{N}_{[0,t^{**}]}$. It then follows from Lemma \ref{conitunation} that $t^{**}=t^*$. Furthermore, (APA6) means $|h|\leq \eta_{\operatorname{lin}}$ on $\mathcal{N}^{J+1}_{t^*}$, we then have, by Proposition \ref{BK-prop 12.22}(c), $t^*=t_J$, and this finishes the proof of Proposition \ref{BK-prop 12.1}.

\begin{lem}[Verification of a priori assumptions]
\
\begin{enumerate}
    \item (Lemma 12.57 in \cite{BamlerKleiner17}) If 
    \begin{gather*}
        \delta_{\operatorname{n}}\leq\overline\delta_{\operatorname{n}},\quad\lambda\leq\overline\lambda,\quad\Lambda\geq\underline\Lambda(\lambda,A),\quad\delta_{\operatorname{b}}\leq\overline\delta_{\operatorname{b}}(\lambda,D_{\operatorname{cut}},A,\Lambda),\quad\epsilon_{\operatorname{can}}\leq\overline\epsilon_{\operatorname{can}}(\lambda,D_{\operatorname{cut}},A,\Lambda),\quad r_{\operatorname{comp}}\leq \overline r_{\operatorname{comp}},
    \end{gather*}
    then (APA6) holds on $\mathcal{N}_{[0,t^{**}]}$.
    \item (Lemma 12.59 in \cite{BamlerKleiner17}) If
    \begin{gather*}
        E\geq\underline{E},\quad F\geq\underline{F}(E),\quad H\geq \underline{H}(E),\quad \etalin\leq \overline \eta_{\operatorname{lin}}(E),\quad \nu\leq\overline\nu(E),\quad\delta_{\operatorname{n}}\leq\overline\delta_{\operatorname{n}},\quad\lambda\leq\overline\lambda,\quad A\geq\underline{A}(E,W),
        \\
        \Lambda\geq\underline{\Lambda},\quad \delta_{\operatorname{b}}\leq\overline\delta_{\operatorname{b}}(E,\lambda,D_{\operatorname{cut}},A,\Lambda),\quad \epsilon_{\operatorname{can}}\leq\overline\epsilon_{\operatorname{can}}(E,\lambda,D_{\operatorname{cut}},W,A,\Lambda),\quad r_{\operatorname{comp}}\leq\overline r_{\operatorname{comp}},
    \end{gather*}
    then (APA7) holds on $\mathcal{N}_{[0,t^{**}]}$.
    \item (Lemma 12.62 in \cite{BamlerKleiner17}) If
    \begin{gather*}
        E\geq \underline{E},\quad H\geq\underline{H}(E),\quad \etalin\leq\overline \eta_{\operatorname{lin}}(E),\quad \nu\leq\overline\nu(E),\quad \delta_{\operatorname{n}}\leq\overline\delta_{\operatorname{n}},\quad \lambda\leq \overline\lambda,\quad W\geq\underline{W}(E,\lambda,D_{\operatorname{cut}})
        \\
        A\geq \underline{A}(E),\quad \Lambda\geq\underline{\Lambda},\quad \delta_{\operatorname{b}}\leq\overline\delta_{\operatorname{b}}(\lambda,D_{\operatorname{cut}},A,\Lambda),\quad\epsilon_{\operatorname{can}}\leq\overline\epsilon_{\operatorname{can}}(E,\lambda,D_{\operatorname{cut}},A,\Lambda),\quad r_{\operatorname{comp}}\leq\overline r_{\operatorname{comp}},
    \end{gather*}
    then (APA8) holds on $\mathcal{N}_{[0,t^{**}]}$.
    \item (Lemma 12.64 in \cite{BamlerKleiner17}) If
    \begin{gather*}
        E\geq\underline{E},\quad H\geq\underline{H},\quad\etalin\leq\overline\eta_{\operatorname{lin}},\quad \nu\leq\overline\nu, \quad \delta_{\operatorname{n}}\leq\overline\delta_{\operatorname{n}},\quad\lambda\leq\overline\lambda,\quad \eta_{\operatorname{cut}}\leq\overline\eta_{\operatorname{cut}},\quad D_{\operatorname{cut}}\geq\underline{D}_{\operatorname{cut}}(\lambda),
        \\
        A\geq\underline{A}(E,\lambda),\quad \Lambda\geq \underline{\Lambda},\quad \delta_{\operatorname{b}}\leq \overline\delta_{\operatorname{b}}(\lambda,D_{\operatorname{cut}},A,\Lambda),\quad \epsilon_{\operatorname{can}}\leq\overline\epsilon_{\operatorname{can}}(E,\lambda,D_{\operatorname{cut}},A,\Lambda),\quad r_{\operatorname{comp}}\leq \overline r_{\operatorname{comp}}(\lambda),
    \end{gather*}
    then (APA9) holds on $\mathcal{N}_{[0,t^{**}]}$.
\end{enumerate}
\end{lem}

\section{Uniqueness of singular Ricci Flows}

Summarizing the results of the previous two sections, we may conclude the following result, namely, if two Ricci flow spacetimes have sufficiently close initial data, then there exists a comparison domain and a comparison between them.

\begin{thm}[Theorem 13.1 in \cite{BamlerKleiner17}, existence of comparison domain and comparison] \label{biguniqueness}
If
\begin{gather*}
    T>0,\quad E\geq\underline{E},\quad H\geq \underline{H}(E),\quad \etalin\leq\baretalin(E),\quad \nu\leq\overline \nu (T,H,\etalin),\quad \dn\leq\bardn(T,H,\etalin),\quad \lambda\leq\overline\lambda(\dn),
    \\ 
    \Dcap\geq\uDcap(\lambda),\quad \eta_{\operatorname{cut}}\leq\overline\eta_{\operatorname{cut}},\quad D_{\operatorname{cut}}\geq \underline{D}_{\operatorname{cut}}(\Dcap,\eta_{\operatorname{cut}}),\quad W\geq\underline{W}(D_{\operatorname{cut}}),\quad A\geq \underline{A}(W),\quad \Lambda\geq \underline{\Lambda}(A),
    \\\db\leq\bardb(\Lambda),\quad \ecan\leq\barecan(\db),\quad \rcomp\leq\barrcomp(\Lambda),
\end{gather*}
then the following holds. 

Consider two $(n+1)$-dimensional $O(n+1)$-invariant and  $(\ecan\rcomp, T)$-complete Ricci flow spacetimes $\mathcal{M}$ and $\mathcal{M}'$ that each satisfy the $O(n+1)$-equivariant $\ecan$-canonical neighborhood assumption at scales $(\ecan\rcomp,1)$. Let $\zeta:\{x\in\mathcal{M}_0\, |\, \rho(x)>\lambda\rcomp\}\to\mathcal{M}'_0$ be an $O(n+1)$-equivariant diffiemorphism onto its image that satisfies the following bounds:
\begin{eqnarray*}
\left|\zeta^* g_0'-g_0\right|&\leq&\etalin,
\\
e^{HT}\rho_1^E\left|\zeta^* g_0'-g_0\right|&\leq&\nu\overline{Q}=\nu\cdot 10^{-E-1}\etalin\rcomp^E,
\\
e^{HT}\rho_1^3\left|\zeta^* g_0'-g_0\right|&\leq&\nu\overline{Q}^*=\nu\cdot 10^{-1}\etalin(\lambda\rcomp)^3.
\end{eqnarray*}
Assume moreover that the $O(n+1)$-equivariant $\ecan$-neighborhood assumption holds at scales $(0,1)$ on the image of $\zeta$.

Then for any $J\geq 1$ with $J\rcomp^2\leq T$ there is a comparison domain $(\mathcal{N},\{\mathcal{N}^j\}_{j=1}^J,\{t_j\}_{j=0}^J)$ and a rotationally equivariant comparison $(\operatorname{Cut},\phi,\{\phi^j\}_{j=1}^J)$ from $\mathcal{M}$ to $\mathcal{M}'$ defined on this domain such that equivariant a priori assumptions (APA1)---(APA6) hold for the tuple of parameters $(\etalin,\dn,\lambda,\Dcap,\Lambda,\db,\ecan,\rcomp)$ and equivariant a priori assumptions (APA7)---(APA13) hold for the tuple of parameters $(T,E,H,\etalin,\nu,\lambda,\eta_{\operatorname{cut}},D_{hi\operatorname{cut}},W,A,\rcomp)$. Moreover, $\phi_{0+}=\phi_0^1=\zeta\vert_{\mathcal{N}_0}$.
\end{thm}

Indeed, if one observes the parameter diagram on page 44 of \cite{BamlerKleiner17}, then it is clear that the basic constants are but a few, and they are either universal, or determined by the initial data and the time interval of the flow. All other parameters are dependent on them. Hence we may reduce the amount of parameters to the basic ones, and this gives rise to the weak and strong stability of the Ricci flow spacetime.

\begin{thm}[Theorem 1.5 in \cite{BamlerKleiner17}, weak stability of Ricci flow spacetimes]
For every $\delta>0$ and $T<\infty$, there is an $\epsilon=\epsilon(\delta,T)>0$ such that the following holds. 

Consider two  $(n+1)$-dimensional $O(n+1)$-invariant and $(\epsilon,T)$-complete Ricci flow spacetimes $\mathcal{M}$ and $\mathcal{M}'$ that each satisfies the $O(n+1)$-equivariant $\epsilon$-canonical neighborhood assumption at scales $(\epsilon,1)$. Let $\phi:U\to U'$ be an $O(n+1)$-equivariant diffeomorphism between two open subsets $U\in\mathcal{M}_0$ and $U'\in\mathcal{M}'_0$. Assume that $|Rm|\geq \epsilon^{-2}$ on $\mathcal{M}_0\setminus U$ and $$\left|\phi^* g_0'-g_0\right|\leq\epsilon.$$ Assume moreover that the $O(n+1)$-equivariant $\epsilon$-canonical neighborhood assumption holds on $U'$ at scales $(0,1)$.

Then there is a time-preserving $O(n+1)$-equivariant diffeomorphism $\widehat\phi:\widehat U\to\widehat U'$ between two open subsets $\widehat U\subset \mathcal{M}_{[0,T]}$ and $\widehat U'\subset \mathcal{M}'_{[0,T]}$ that evolves by the harmonic map heat flow and that satisfies $\widehat\phi=\phi$ on $U\cap\widehat U$ and $$\left|\widehat\phi^* g'-g\right|\leq\delta.$$ Moreover, $|Rm|\geq\delta^{-2}$ on $\mathcal{M}_{[0,T]}\setminus \widehat U$.
\end{thm}

\begin{thm}[Theorem 1.7 in \cite{BamlerKleiner17}, strong stability of Ricci flow spacetimes]\label{thestrongstability}
There is a constant $\underline{E}<\infty$ such that for every $\delta>0$, $T<\infty$, and $\underline{E}\leq E<\infty$ there are constants $\ecan=\ecan(E)$ and $\epsilon=\epsilon(\delta, T, E)>0$, such that for all $0<r<1$ the following holds.

Consider two $(n+1)$-dimensional $O(n+1)$-invariant and $(\epsilon r,T)$-complete Ricci flow spacetimes $\mathcal{M}$ and $\mathcal{M}'$ that each satisfies the $O(n+1)$-equivariant $\ecan$-canonical neighborhood assumption at scales $(\epsilon r,1)$. Let $\phi: U\to U'$ be an $O(n+1)$-equivariant diffeomorphism between two open subsets $U\subset \mathcal{M}_0$ and $U'\subset\mathcal{M}_0'$. Assume that $|\Rm|\geq(\epsilon r)^{-2}$ on $\mathcal{M}_0\setminus U$ and $$|\phi^*g_0'-g_0|\leq\epsilon \cdot r^{2E}(|\Rm|+1)^E\quad \text{on}\quad U.$$ Assume moreover that the $O(n+1)$-equivariant $\ecan$-canonical neighborhood assumption holds on $U'$ at scales $(0,1)$. Then there is a time-preserving $O(n+1)$-equivariant diffeomorphism $\widehat \phi:\widehat U\to\widehat U'$ between two open subsets $\widehat U\subset \mathcal{M}_{[0,T]}$ and $\widehat U'\subset\mathcal{M}'_{[0,T]}$ that evolves by the harmonic map heat flow, satisfies $\widehat\phi=\phi$ on $U\cap\widehat U$ and that satisfies $$|\widehat\phi^*g'-g|\leq\delta\cdot r^{2E}(|\Rm|+1)^E\quad\text{ on }\quad \widehat U.$$ Moreover, we have $|\Rm|\geq r^{-2}$ on $\mathcal{M}_{[0,T]}\setminus\widehat U$. If additionally $|\Rm|\geq(\epsilon r)^{-2}$ on $\mathcal{M}_0'\setminus U'$, then we also have $|\Rm|\geq r^{-2}$ on $\mathcal{M}'_{[0,T]}\setminus\widehat U'$.
\end{thm}

Finally, the uniqueness of the singular Ricci flow is a consequence of the stability theorems stated above. Indeed, if we consider two $O(n+1)$-invariant and $(0,T)$-complete Ricci flow spacetimes with the same initial data, then, by taking the parameters $\varepsilon$ and $\delta$ to $0$, we obtain a limit for the map $\widehat\phi$ given by Theorem \ref{thestrongstability}. One can then show that this limit map is an $O(n+1)$-equivariant isometry. The technical details can be found in the proof of \cite[Theorem 1.3]{BamlerKleiner17}.

\begin{thm}[Theorem 1.3 in \cite{BamlerKleiner17}, uniqueness of Ricci flow spacetime]
There is a dimension-dependent universal constant $\ecan>0$ such that the following holds. Let $(\mathcal{M},\mathfrak{t},\partial_{\mathfrak t},g)$ and $(\mathcal{M}',\mathfrak{t}',\partial_{\mathfrak t'},g')$ be two $(n+1)$-dimensional $O(n+1)$-invariant Ricci flow spacetimes that are both $(0,T)$-complete for some $T\in(0,\infty]$ and satisfy the $O(n+1)$-invariant $\ecan$-canonical neighborhood assumption at scales $(0,r)$ for some $r>0$. If the initial time-slices $(\mathcal{M}_0,g_0)$ and $(\mathcal{M}'_0,g_0')$ are isometric, then the flows $(\mathcal{M},\mathfrak{t},\partial_{\mathfrak t},g)$ and $(\mathcal{M}',\mathfrak{t}',\partial_{\mathfrak t'},g')$ are isometric as well. More precisely, assume that there is an $O(n+1)$-equivariant isometry $\phi:(\mathcal{M}_0,g_0)\to (\mathcal{M}'_0,g_0')$. Then there is a unique $O(n+1)$-equivariant smooth diffeomorphism $\widehat\phi:\mathcal{M}'_{[0,T]}\to\mathcal{M}_{[0,T]}$ such that
$$\widehat\phi^*g'=g,\quad \widehat\phi\vert_{\mathcal{M}_0}=\phi,\quad \widehat\phi_*\partial_{\mathfrak t}=\partial_{\mathfrak t'},\quad \mathfrak t'\circ\phi=\mathfrak t.$$
\end{thm}

\section{Number of bumps in a Ricci flow spacetime}

Let $(M^{n+1},g)$ be a connected rotationally invariant Riemannian
manifold, which is not necessarily complete. Fix any $O(n+1)$-orbit
$\mathcal{O}\subseteq M$, and fix a parametrization of the nonsingular
orbits, so that we can write
\[
g=\varphi^{2}(r)dr^{2}+\psi^{2}(r)g_{\mathbb{S}^{n}}
\]
for some smooth warping function $\psi\in C^{\infty}((a,b))$. Define $\text{Bump}(M)$ to be the number of
critical points $(\psi')^{-1}(0)$ of $\psi$ if $(\psi')^{-1}(0)$
is finite, and $\text{Bump}(M):=\infty$ otherwise. Thus, $\text{Bump}(M)$ records the number of bumps of the warping function of $(M,g)$. Note that this function does not depend on the choice of parametrization, since a change of parametrization corresponds to composing $\psi$ with a diffeomorphism between open intervals in $\mathbb{R}$. If $M$ is not connected, we let $\text{Bump}(M)$ be the sum of $\text{Bump}(M')$ over all connected components $M'$ of $M$. 
\begin{lem} Let $\mathcal{M}$ be a rotationally invariant Ricci flow spacetime as obtained in \ref{spacetimeexists}. Then $(0,\infty)\to\mathbb{N}\cup\{\infty\},t\mapsto\operatorname{Bump}(\mathcal{M}_{t})$ is nonincreasing. 
\end{lem}

\begin{proof}
First observe that each $\mathcal{M}_{t}$ has the structure of a real
analytic Riemannian manifold (by taking an atlas of geodesic coordinates)
by Shi's local derivative estimates. Thus, for any non-surgery time
$t>0$, we have $\text{Bump}(\mathcal{M}_{t})<\infty$. 

Consider the case where $0<t_{1}<t_{2}$ and $\mathcal{M}=M\times[t_{1},t_{2}]$
has no surgery times in the time interval $[t_{1},t_{2}]$. It suffices
to consider the case where $M$ is connected. Fix normal exponential coordinates with respect to fixed principal orbit at some fixed time $t_0\in [t_1,t_2]$. By the proof of Theorem 2.1 in \cite{deTurckKazdan}, the corresponding warping function $\psi:(a,b)\to (0,\infty)$ is real analytic. By possibly considering a subinterval of $[t_1 ,t_2]$ containing $t_0$, and using a finite cover of $[t_1,t_2]$ by such intervals, we can apply Theorem 4 of \cite{Kotschwar15} to obtain that the warping function is real analytic in spacetime:
$\psi\in C^{\omega}((a,b)\times[t_{1},t_{2}])$. Let $x\in(a,b)$
satisfy $\partial_x\psi(x,t_{2})=0$. If $\partial_{x}^{2}\psi(x,t_{2})=0$,
then Theorem 1.2 of \cite{Galaktionov04} gives $\epsilon>0$ and a continuous
curve $\gamma:(t_{2}-\epsilon,t_{2}]\to M$ (there may be multiple
such curves in the zero set due to branching behavior backwards in
time, but we just choose one) such that $\gamma(t_{2})=x$, and $\partial_{x}\psi(\gamma(t),t)=0$,
$\partial_{x}^{2}\psi(\gamma(t),t)\neq0$ for all $t\in(t_{2}-\epsilon,t_{2})$.
If instead $\partial_{x}^{2}\psi(x,t_{2})\neq0$, then the implicit
function theorem gives $\epsilon>0$ and a smooth function $\gamma(t)$
with $\gamma(t_{2})=x$ and $\partial_{x}\psi(\gamma(t),t)=0$ for
all $t\in(t_{2}-\epsilon,t_{2}]$. In particular, there exist $\epsilon>0$
and at least $N:=\text{Bump}(\mathcal{M}_{t_{2}})$ distinct smooth curves
$\gamma_{j}:(t_{2}-\epsilon,t_{2}]\to M$ such that $\partial_{x}^{2}\psi(\gamma_{j}(t),t)\neq0$
and $\partial_{x}\psi(\gamma_{j}(t),t)=0$ for all $t\in(t_{2}-\epsilon,t_{2}]$,
which we can assume satisfy $\gamma_{1}<\cdots<\gamma_{N}$. By the
implicit function theorem, the maximal time interval $(-\epsilon,t_{2}]\subseteq I\subseteq[t_{1},t_{2}]$
such that all $\gamma_{j}(t)$ are defined, continuous and satisfy $\partial_x \psi (\gamma_j(t),t)=0$,  $a<\gamma_{1}(t)<\cdots<\gamma_{N}(t)<b$
for all $t\in I$ is open in $[t_{1},t_{2}]$, so we
can write $I=(t^{\ast},t_{2}]$ if $t_{1}<t^{\ast}$. 
\\

\noindent \textbf{Claim 1: }$\liminf_{t\searrow t^{\ast}}\gamma_{1}(t)>a$
and $\limsup_{t\searrow t^{\ast}}\gamma_{N}(t)<b$.
\begin{proof}[Proof of the claim]
Because $g(t^{\ast})$ extends to a smooth metric over the singular
orbit corresponding to $x=a$, we know that $L:=\lim_{x\searrow a}\partial_{x}\psi(x,t^{\ast})>0$,
so there exists $\delta>0$ such that $\partial_{x}\psi(x,t^{\ast})>\frac{1}{2}L$
for all $x\in(a,a+\delta)$. Moreover, $|\partial_{t}\psi|$ is bounded
in terms of the curvatures of $g(t)$ and their covariant derivatives,
which are uniformly bounded on $M\times[t^{\ast},t_{2}]$, so by shrinking
$\delta>0$, we can assume $\partial_{x}\psi(x,t)>\frac{1}{4}L$ for
all $(x,t)\in(a,a+\delta)\times[t^{\ast},t^{\ast}+\delta^{2})$. If
$\liminf_{t\searrow t^{\ast}}\gamma_{1}(t)=a$, then there exist $t_{k}\searrow t^{\ast}$
such that $\partial_{x}\psi(\gamma_{1}(t_{k}),t_{k})=0$. However,
$(\gamma_{1}(t_{k}),t_{k})\in(a,a+\delta)\times[t^{\ast},t^{\ast}+\delta^{2})$
for $k\in\mathbb{N}$ sufficiently large, a contradiction. The remaining
claim is completely analogous. \end{proof}

\noindent \textbf{Claim 2: }$\lim_{t\searrow t^{\ast}}\gamma_{j}(t)\in(a,b)$
exist for $j=1,...,N$. 
\begin{proof}[Proof of the claim]
Let $J\subseteq[a,b]$ be the set of limit points of $\gamma_{j}(t)$
as $t\searrow t^{\ast}$. By the previous claim, $J\subseteq(a,b)$.
Because $\psi$ is analytic and $\partial_{x}\psi(\gamma_{j}(t),t)=0$
for $t\in(t^{\ast},t_{2}]$, we have $\partial_{x}\psi(y,t^{\ast})=0$ for any $y\in J$.
Moreover, for any $y_{1},y_{3}\in J$ and $y_{2}\in(y_{1},y_{3})$,
and any $\epsilon>0$, there is a sequence $s_{k}\searrow t^{\ast}$
with $|\gamma_{j}(s_{2k})-y_{1}|<\epsilon$ and $|\gamma_{j}(s_{2k+1})-y_{3}|<\epsilon$
for all $j\in\mathbb{N}$. By the intermediate value theorem, there
are $s_{k}^{\ast}\in(s_{2k},s_{2k+1})$ with $|\gamma_{j}(s_{k}^{\ast})-y_{2}|<2\epsilon$,
hence $y_{2}\in J$. That is, $J$ is connected, so it is a singleton
or an interval. However, $J$ cannot be an interval because then $\partial_{x}\psi(\cdot,t^{\ast})$
vanishes on an interval, a contradiction. \end{proof}

The previous claim shows that $\gamma_{j}$ extend to continuous curves
$\gamma_{j}:[t^{\ast},t_{2}]\to(a,b)$ which satisfy $\gamma_{j}(t^{\ast})\leq\gamma_{j+1}(t^{\ast})$
for $j=1,...,N-1$ and $\psi(\gamma_{j}(t^{\ast}),t^{\ast})=0$. We
also check that $\gamma_{j}(t^{\ast})<\gamma_{j+1}(t^{\ast})$ for
$j=1,...,N-1$. If this were not the case, then $\partial_{x}\psi$
would vanish on the parabolic boundary of the set bounded by $\gamma_{j}([t^{\ast},t_{2}])$,
$\gamma_{j+1}([t^{\ast},t_{2}])$, and $(a,b)\times\{t_{2}\}$, so
by the maximum principle, $\partial_{x}\psi=0$ there. This implies
that $\partial_{x}\psi(\cdot,t_{2})$ vanishes on the interval $(x_{j},x_{j+1})$,
a contradiction. Repeating previous reasoning, we see that $\gamma_{j}$
can be extended to curves $\gamma_{j}:[t^{\ast}-\epsilon,t_{2}]\to(a,b)$
with $\partial_{x}\psi(\gamma_{j}(t),t)=0$ for $j=1,...,N$ and $\gamma_{j}(t)<\gamma_{j+1}(t)$
for all $t\in[t^{\ast}-\epsilon,t_{2}]$, $j=1,...,N-1$. This contradicts
the minimality of $t^{\ast}$, so we must actually have $t^{\ast}=t_{1}$.
Thus $\gamma_{1}(t_{1})<\cdots<\gamma_{N}(t_{1})$, and $\partial_{x}\psi(\gamma_{j}(t_{1}),t_{1})=0$
for $j=1,...,N$, so $\text{Bump}(\mathcal M_{t_{1}})\geq N$.

Observe that the curves $\gamma_{j}$ are smooth in a neighborhood
of any $t\in[t_{1},t_{2}]$ where $\partial_{x}^{2}\psi(\gamma_{j}(t),t)\neq0$. 
\\

\noindent \textbf{Claim 3: }There are only finitely many $t\in[t_{1},t_{2}]$
where $\partial_{x}^{2}\psi(\gamma_{j}(t),t)=0$.
\begin{proof}[Proof of the claim]
In fact, $X_1 :=\{(x,t)\in (a,b)\times (t_1-\epsilon,t_2+\epsilon) ; \partial_x \psi (x,t)=0\}$ and $X_2 :=\{(x,t)\in (a,b)\times (t_1-\epsilon,t_2+\epsilon) ; \partial_x^2 \psi (x,t)=0\}$ are proper real analytic subsets of $(a+\epsilon,b-\epsilon)\times (t_1-\epsilon,t_2+\epsilon)$ for some $\epsilon>0$. Thus, either $X_1 \cap X_2 \cap \left((a,b) \times [t_1,t_2]\right)$ is finite, or this set contains an entire real-analytic curve $\eta: (-\delta,\delta)\to (a,b)\times (t_1 ,t_2)$. However, the proof of Theorem 1.2 in \cite{Galaktionov04} indicates that any curve in $X_1$ oriented backwards in time emanating from some point $(x_0,t_0)\in X_2$ only paramatrizes simple zeros for some interval $t\in (t_0-\delta,t_0)$, contradicting $\partial_x^2 \psi(\eta(t),t)\equiv 0$. \end{proof}

For any $t_{0}\in[t_{1},t_{2}]$ where $\partial_{x}^{2}\psi(\gamma_{j}(t_{0}),t_{0})=0$,
we have (again by \cite{Galaktionov04}) $\gamma_{j}(t)=\gamma_{j}(t_{0})+O(|t-t_{0}|)$ as $t\searrow t_{0}$,
while $\gamma_{j}(t)=\gamma_{j}(t_{0})+a_{j}\sqrt{t_{0}-t}+O(|t-t_{0}|)$
as $t\nearrow t_{0}$, for some $a_{j}\in\mathbb{R}$. In either case,
$\gamma_{j}$ is absolutely continuous. Because $\psi$ is smooth,
we therefore know that $t\mapsto\psi(\gamma_{j}(t),t)$ is absolutely
continuous.
\\

\noindent \textbf{Claim 4:} The right upper Dini-derivative of $t\mapsto \psi(\gamma_j(t),t)$ is bounded.
\begin{proof}[Proof of the claim]
We only need to consider the case where $\gamma_j(t)$ hits a degenerate critical point, for otherwise $\gamma_j(t)$ is smooth. Let $\gamma_j(t_0)$ be a degenerate critical point of $\psi(\cdot,t_0)$. Then according to Theorem 1.2 in \cite{Galaktionov04}, either $\gamma_j(t)$ vanishes once $t>t_0$, which contradicts our setting, or, $\gamma_j(t)$ is a simple critical point for $t>t_0$ close enough to $t_0$, satisfying
\begin{eqnarray*}
|\gamma_j(t)-\gamma_j(t_0)|\leq O(t-t_0).
\end{eqnarray*}
It follows immediately that
\begin{eqnarray*}
\frac{D^+}{dt}\gamma_j(t_0):=\limsup_{\tau\searrow t_0}\frac{\gamma_j(t_0+\tau)-\gamma_j(t_0)}{\tau}<\infty , \\
\frac{D_+}{dt}\gamma_j(t_0):=\liminf_{\tau \searrow t_0}\frac{\gamma_j(t_0+\tau)-\gamma_j(t_0)}{\tau}>-\infty 
\end{eqnarray*}

Knowing that the right Dini derivative of $\gamma_j(t)$ is always finite, then one may compute the right Dini derivative of $\psi(\gamma_j(t),t)$ as follows
\begin{eqnarray*}
\frac{D^+}{dt}\psi(\gamma_j(t),t)&=&\limsup_{\tau \searrow 0}\frac{\psi(\gamma_j(t+\tau),t+\tau)-\psi(\gamma_j(t),t)}{\tau}
\\ 
&=&\limsup_{\tau \searrow 0} \psi_s(\gamma_j(t),t)\left( \frac{\gamma_j(t+\tau)-\gamma_j(t)}{\tau} \right)+\psi_t(\gamma_j(t),t)
\\ 
&=& \psi_t(\gamma_j(t),t)
\\
&=&2Rc_{\text{tang}}(\gamma_{j}(t))\psi(\gamma_{j}(t)).
\end{eqnarray*}
Here we have used the fact that $\gamma_j(t)$ is a critical point of $\psi(\cdot,t)$. \end{proof}

We now consider the case where singularities occur in the time interval
$[t_{1},t_{2}]$, and suppose $p_{1},...,p_{N}\in\mathcal{M}_{t_{2}}$
satisfy $\partial_{s}\psi(p_{i})=0$ for $i=1,...,N$, where we
now let $\psi\in C^{\infty}(\mathcal{M})$ and $\partial_s \in \mathfrak{X}(\mathcal{M})$ be as in \ref{spacetimewarp}. Moreover, it suffices to consider the case where $\mathcal{M}_{t_2}$ is connected; then $\mathcal{M}_t$ is connected for all $t\in [t_1,t_2]$. Arguing as in the nonsingular case, we obtain curves $\gamma_{1},...,\gamma_{N}:(t^{\ast},t_{2}]\to\mathcal{M}$,
where $\gamma_{j}(t)\in\mathcal{M}_t$, $\partial_{s}\psi(\gamma_{j}(t))=0$,
and $\mathcal{O}(\gamma_{1}(t)),...,\mathcal{O}(\gamma_N(t))$ do not intersect. Suppose by way
of contradiction that $\gamma_{j}(t)$ has no limit points in $\mathcal{M}$
as $t\searrow t^{\ast}$. By retracing the argument in the nonsingular case, the $0$-completeness of $\mathcal{M}$ implies $R(\gamma_{j}(t))\to\infty$. We know that, for any $\epsilon>0$, $\mathcal{M}$ satisfies the equivariant $\epsilon$-canonical neighborhood assumption at scales $(0,r_0(\epsilon))$, so that after rescaling by $R(\gamma_j(t))$, arbitrarily large backwards parabolic cylinders of $\gamma_j(t)$ are unscathed, with bounded curvature; given any sequence $t_k \searrow t^{\ast}$, we may therefore pass to a further subsequence to get smooth pointed convergence of rescalings of $(\mathcal{M}_{t\leq t_k},\gamma_j(t_k))$ (in the sense of spacetimes) to a $\kappa_0$-solution $(M_{\infty},(g_t^{\infty})_{t\in (-\infty,0]},p_{\infty})$; moreover, if $\psi_{\infty}$ denotes the warping function of $M_{\infty}$, then $\partial_s \psi_{\infty}(p_{\infty})=0$. Thus, $M_{\infty}$ is compact or a shrinking cylinder. If $M_{\infty}$ is compact, then $\lim_{k\to \infty}\operatorname{diam}_{g_{t_k}}(\mathcal{M}_{t_k})=0$, which is impossible.

On the other hand,
we know that whenever $\gamma_{j}$ is differentiable, 
\begin{align*}
\frac{d}{dt}\psi(\gamma_{j}(t))&=\partial_{\mathfrak{t}}\psi(\gamma_{j}(t))+\langle \nabla \psi(\gamma_j(t)),\dot{\gamma}_j(t)\rangle_{g_t}
\\&= \partial_s (\partial_s \psi)(\gamma_j(t))-(n-1)\frac{1}{\psi(\gamma_j(t))}\\ 
&= -2Rc_{\text{tang}}(\gamma_{j}(t))\psi(\gamma_{j}(t)),
\end{align*}
where $Rc_{\text{tang}}(\gamma_{j}(t))$ is the unique eigenvalue
of $Rc$ restricted to $T\left( \mathcal{O}(\gamma_j(t))\right)$.

Because $Rc_{\text{tang}}(\gamma_j(t))\psi^2 (\gamma_j(t))$ converges to a positive constant, there exists $\epsilon>0$
such that 
\[
\frac{D^+}{dt}\psi(\gamma_{j}(t))<0
\]
for $t\in(t^{\ast},t^{\ast}+\epsilon)$. However, $\lim_{t\searrow t^{\ast}}\psi(\gamma_{j}(t))=0$, so by Claim 4, we may apply Hamilton's comparison theorem for ODE \cite{Hamilton86} to obtain that $\psi(\gamma_j(t))\leq 0$ for all $t\in(t^{\ast},t^{\ast}+\epsilon)$; this is a contradiction.

Suppose by way of contradiction that $\text{Bump}(\mathcal{M}_{t})=\infty$
for some $t>0$. Then, for any $N\in\mathbb{N}$, we can find points
$x_{1},...,x_{N} \in \mathcal{M}_t$ with $\partial_{s}\psi(x_j)=0$, and the
above reasoning shows that we can find $N$ curves
$\gamma_{1},...,\gamma_{N}:(t-\epsilon,t)\to \mathcal{M}$ such that $\gamma_j(s)\in \mathcal{M}_s$, $\mathcal{O}(\gamma_j(s))$ are pairwise disjoint, and  $\partial_{s}\psi(\gamma_{j}(s))=0$
for all $s\in(t-\epsilon,t)$. In particular, fixing a regular time
$t_{0}<t$, and letting $\psi_{t_0}:(a,b)\to (0,\infty)$ be the warping function at this time, we see that $\partial_{s}\psi_{t_0}$ has at least $N$ zeros for any $N\in\mathbb{N}$; that is, $\mathcal{M}_{t_0}$ is a regular time with infinitely many bumps, a contradiction.
\end{proof}

\section{Bounding the blowup rate}

In this section, we show that the curvature of a closed rotationally invariant Ricci flow cannot blow up arbitrarily quickly.  Roughly speaking, the first step is to show that if the solution is not Type-I, then the curvature is largest at the singular orbits. We will then show that, at the bump closest to the singular orbit, the warping function is larger than the shrinking cylinder with the same singularity time. 

\begin{lem}
Suppose $(\mathbb{S}^{n+1},(g_{t})_{t\in[0,T)})$ is a rotationally
invariant Ricci flow developing a Type-II singularity at time $T$,
and let $o_{S},o_{N}\in\mathbb{S}^{n+1}$ denote the two singular
orbits. Suppose 
\[
\limsup_{t\nearrow T}\max\left(|Rm|(o_{S},t),|Rm|(o_{N},t)\right)(T-t)^{2}<\infty.
\]
Then 
\[
\limsup_{t\nearrow T}\max_{M}|Rm(\cdot,t)|(T-t)^{2}<\infty.
\]
\end{lem}

\begin{proof}
Suppose by way of contradiction that there exist $(y_{j},s_{j})\in M\times[0,T)$
satisfying $s_{j}\nearrow\infty$ and 
\[
\lim_{j\to\infty}|Rm|(y_{j},s_{j})(T-s_{j})^{2}=\infty.
\]
Set $T_{j}:=\frac{T+s_{j}}{2}\nearrow T$, and choose Choose $(x_{j},t_{j})\in M\times[0,T_{j}]$
maximizing 
\[
M\times[0,T_{j}]\to[0,\infty),\quad(x,t)\mapsto|Rm|(x,t)(T_{j}-t)^{2},
\]
so that $Q_{j}:=|Rm|(x_{j},t_{j})$ satisfy
\[
Q_{j}(T_{j}-t_{j})^{2}\geq|Rm|(y_{j},s_{j})(T_{j}-s_{j})^{2}=\frac{1}{4}|Rm|(y_{j},s_{j})(T-s_{j})^{2}\to\infty
\]
as $j\to\infty$. Set $g_{t}^{j}:=Q_{j}g_{t_{j}+Q_{j}^{-1}t}$ for
$t\in[-Q_{j}t_{j},Q_{j}(T-t_{j})]$, so that for any $(x,t)\in M\times[-Q_{j}t_{j},\frac{1}{2}Q_{j}(T_{j}-t_{j})]$,
we have 
\begin{align*}
  |Rm|_{g^{j}}(x,t)&=Q_j^{-1}|Rm|_g\left(x,t_j+Q_j^{-1}t\right)\le Q_j^{-1}\frac{Q_j(T_j-t_j)^2}{\left(T_j-t_j-Q_j^{-1}t\right)^2}
  \\
  &\leq\frac{(T_{j}-t_{j})^{2}}{\frac{1}{4}(T_j-t_{j})^{2}}\leq4.
\end{align*}
Because $Q_{j}(T_{j}-t_{j})^{2}\to\infty$, we can pass to a subsequence
so that $(M,(g_{t}^{j})_{t\in[-Q_{j}t_{j},\frac{1}{2}Q_{j}(T-t_{j})]},x_{j})$
converge in the equivariant, pointed Cheeger-Gromov-Hamilton sense to a nonflat,
eternal, rotationally invariant $\kappa$-solution $(M_{\infty},(g_{t}^{\infty})_{t\in\mathbb{R}},x_{\infty})$,
which must therefore be the Bryant soliton. Let $x_{\text{Bry}}\in M_{\infty}$
be the tip of the Bryant soliton. There is an $O(n+1)$-invariant,
precompact, open exhaustion $(U_{j})$ of $M_{\infty}$, along with
$O(n+1)$-equivariant diffeomorphisms $\phi_{j}:U_{j}\to M_{j}$ such
that $\phi_{j}^{\ast}g^{j}\to g^{\infty}$ in $C_{loc}^{\infty}(M_{\infty}\times\mathbb{R})$.
However, $\phi_{j}$ must map singular orbits of the $O(n+1)$ action
to singular orbits, so $\phi_{j}(x_{\text{Bry}})\in\{o_{N},o_{S}\}$
for each $j\in\mathbb{N}$; passing to a further subsequence, we can
assume $\phi_{j}(x_{\text{Bry}})=o_{N}$. Then 
\begin{align*}
|Rm|_{g^{\infty}}(x_{\text{Bry}},0)= & \lim_{j\to\infty}|Rm|_{g^{j}}(o_{N},0)=\lim_{j\to\infty}\frac{|Rm|_{g}(o_{N},t_j)}{Q_{j}}\\
\leq & \limsup_{j\to\infty}\frac{C}{Q_{j}(T-t_{j})^{2}}=0,
\end{align*}
a contradiction.
\end{proof}
Now fix a rotationally invariant Ricci flow $(\mathbb{S}^{n+1},(g_{t})_{t\in[0,T)})$
developing a Type-II singularity, and identify
\[
(\mathbb{S}^{n+1}\setminus\{o_{N},o_{S}\},g_{t})=((-1,1)\times\mathbb{S}^{n},\varphi^{2}(x,t)dx^{2}+\psi^{2}(x,t)g_{\mathbb{S}^{n}}),
\]
where $o_{N}$ corresponds to the singular orbit at $x=-1$. By possibly
truncating the time interval $[0,T)$, we may assume that there exists
a function $x_{\ast}:[0,T)\to(-1,1)$ so that $x_{\ast}(t)$ is the
location of the left-most bump of $\psi(\cdot,t),$ and that the number
of bumps and necks does not change for any time in $[0,T)$. 
\begin{lem}\label{upperbound of u}
$\psi(x_{\ast}(t),t)>\sqrt{2(n-1)(T-t)}$ for all $t\in[0,T)$. 
\end{lem}

\begin{proof}
Because $\partial_{s}^{2}\psi(x_*(t),t)<0$, we may compute
\[
\frac{d}{dt}\psi^{2}(x_*(t),t)<2\psi(x_*(t),t)\partial_{x}\psi(x_*(t),t)x_*'(t)-2(n-1)=-2(n-1),
\]
so integrating from $t$ to $T$ gives 
\[
-\psi^{2}(x_*(t),t)<-2(n-1)(T-t),
\]
and the claim follows. 
\end{proof}
Next, we would like to compare the solution $g(t)$ to the Ricci flow ``subsolutions'' constructed in \cite{WuTypeII}. To do so, we must work in the so called ``horizontal coordinates'' defined as follows. Because $\partial_{x}\psi(\cdot,t)>0$ on $(0,x_{\ast}(t))$ for each $t\in[0,T)$, we know
that the inverse $x_{t}(\psi):=\left(\psi(\cdot,t)\right)^{-1}(\psi)$
is well-defined. If we define $z(\psi,t):=(\partial_{s}\psi)^{2}(x_{t}(\psi),t)$,
then it is well-known \cite{WuTypeII} that $z$ satisfies the parabolic equation
\[
\partial_{t}z(\psi,t)=z\partial_{\psi}^{2}z(\psi,t)+\left(n-1-z(\psi,t)\right)\frac{\partial_{\psi}z(\psi,t)}{\psi}+2(n-1)\frac{z(\psi,t)\left(1-z(\psi,t)\right)}{\psi^{2}}-\frac{1}{2}(\partial_{\psi}z)^{2}(\psi,t).
\]
Making the Type-I change of variables $u(\psi,t):=\frac{\psi}{\sqrt{2(n-1)(T-t)}}$
and $\tau(t):=-\log(T-t)$, the equation become $\partial_{\tau} z(u,\tau)=\mathcal{D}[z](u,\tau)$, where
\begin{align*}
\mathcal{D}[z](u,\tau):=&\frac{1}{2(n-1)}\left(z\partial_{u}^{2}z(u,\tau)+\left(n-1-z(u,\tau)\right)\frac{\partial_{u}z(u,\tau)}{u}+2(n-1)\frac{z(u,\tau)\left(1-z(u,\tau)\right)}{u^{2}}-\frac{1}{2}(\partial_{u}z)^{2}(u,\tau)\right)\\&-\frac{1}{2}u
\partial_{u}z(u,\tau).
\end{align*}

A comparison principle for this PDE was proved in \cite{WuTypeII}, which
we recall here for the reader's convenience. 

\begin{lem}
Let $\tau_0>0$ and $\overline{\tau}\in(\tau_{0},\infty)$ be arbitrary, and suppose that
$z^{+}$ is a supersolution and $z^{-}$ is a subsolution of 
\[
\partial_{\tau}z=\mathcal{D}_{u}[z].
\]
Suppose that there exists $K<\infty$ such that either $|\frac{1}{u}\partial_{u}z^{-}|,|\partial_{u}^{2}z^{-}|\leq Ke^{\lambda\tau}$
or $|\frac{1}{u}\partial_{u}z^{+}|,|\partial_{u}^{2}z^{+}|\leq Ke^{\lambda\tau}$
for all $\tau\in[\tau_{0},\overline{\tau}]$. Furthermore, assume
\begin{enumerate}[(C1)]
\item $z^{-}(u,\tau_{0})<z^{+}(u,\tau_{0})$ for all $u\in(0,1)$,
\item $z^{-}(0,\tau)\leq z^{+}(0,\tau)$ and $z^{-}(1,\tau)\leq z^{+}(1,\tau)$
for all $\tau\in[\tau_0,\overline \tau]$.
\end{enumerate}
Then $z^{-}(u,\tau)\leq z^{+}(u,\tau)$ for all $(u,\tau)\in[0,1]\times[\tau_{0},\overline{\tau}]$. 
\end{lem}

We now apply this maximum principle to show that, if $z$ corresponds to the region of an arbitrary rotationally invariant Ricci flow on a closed manifold from a singular orbit to the nearest bump, then an appropriately chosen, rescaled subsolution constructed in \cite{WuTypeII} will be a lower barrier for $z$.

Note that the boundedness conditions on $\partial_{u}z^{+}$ and $\partial^2_{u}z^{+}$
are satisfied if $z^{+}$ is a solution of $\partial_{\tau}z=\mathcal{D}[z]$
corresponding to a smooth solution of Ricci flow. We now recall Wu's construction of barrier functions for rotationally invariant
Ricci flow on $\mathbb{R}^{n+1}$. The following diagram summarizes the order in which we will choose the parameters $A_1,A_2,A_3,\beta,D,\tau_0$ for these solutions.
\[\begin{tikzcd}
	&&& {z_0} \\
	{A_2} & {A_3} & \beta & D & {A_1} & {\tau_0}
	\arrow[from=2-1, to=2-2]
	\arrow[from=2-2, to=2-3]
	\arrow[from=2-3, to=2-4]
	\arrow[from=1-4, to=2-4]
	\arrow[from=2-4, to=2-5]
	\arrow[from=2-5, to=2-6]
\end{tikzcd}\]
Here, $z_0$ is a function determined by our given solution of Ricci flow.

In \cite{WuTypeII}, Wu defines a lower
barrier function by patching together the functions
\[
z_{\text{int}}(u,\tau):=\mathcal{B}(A_{1}e^{\frac{1}{2}\tau}u),
\]
\[
z_{\text{ext}}(u,\tau):=A_{2}e^{-\tau}Z_{1}(u)-e^{-2\tau}A_{3}\zeta(u),
\]
where $\mathcal{B}$ is the $z$-function (in horizontal coordinates)
for the Bryant soliton, $Z_{1}(u):=\frac{1}{u^{2}}(1-u^{2})^{2}$,
and 
\[
\zeta(u):=\frac{1}{u^{4}}(1-u^{2})^{2}\left(1-2u^{2}+2u^{2}(1-u^{2})\left(\log(1-u^{2})-2\log u\right)\right).
\]
An easy computation shows that $z_{\text{int}}(u,\tau)$ is a
subsolution in $\{u>0,\tau\geq0\}$. A less trivial computation shows
that $z_{\text{ext}}(u,\tau)$ is a subsolution in $\left\{ be^{-\frac{1}{2}\tau}\leq u<1,\tau\in[\tau_{0},\infty)\right\} $
for any $A_{2}$, if $A_{3}\geq\underline{A}_{3}(A_{2})$ and $\tau_{0}\geq\underline{\tau}_{0}(A_{2},A_{3})$,
where $b=b(A_{2},A_{3})>0$. 

Wu also notes the following asymptotics:

\[
\mathcal{B}(r)=1-b_{2}r^{2}+O(r^{4}),\qquad\mathcal{B}'(r)=-2b_{2}r+O(r^{3})\quad\text{ as }\quad r\to0,
\]
\[
\mathcal{B}(r)=\frac{1}{r^{2}}+\frac{c_{2}}{r^{4}}+O\left(\frac{1}{r^{6}}\right)\quad\text{ as }\quad r\to\infty,
\]
\[
\zeta(u)=-(1-u^{2})^{2}+O\left((1-u^{2})^{3}\log(1-u^{2})\right)\quad\text{ as }\quad u\to1,
\]
\[
\zeta(u)=\frac{1}{u^{4}}+O\left(\frac{1}{u^{2}}\log u\right)\quad\text{ as }\quad u\to0,
\]
where both $b_2$ and $c_2$ are some numerical constants depending only on $n$. Now fix $\beta\in (1,\infty)$ to be determined, and suppose $r\in[1,\beta]$.
Then 
\[
z_{\text{int}}(e^{-\frac{1}{2}\tau}r,\tau)=\mathcal{B}(A_{1}r),
\]
\[
z_{\text{ext}}(e^{-\frac{1}{2}\tau}r,\tau)=\frac{(1-e^{-\tau}r^{2})}{r^{2}}\left(A_{2}-\frac{A_{3}}{r^{2}}\left(1-2e^{-\tau}r^{2}+2e^{-\tau}r^{2}(1-e^{-\tau}r^{2})\left(\log(1-e^{-\tau}r^{2})-\log(e^{-\tau}r^{2})\right)\right)\right).
\]
Note that when $A_{1}\geq\frac{1}{2}$, we have
\[
\left|z_{\text{int}}(e^{-\frac{1}{2}\tau}r,\tau)-\left(\frac{1}{A_{1}^{2}r^{2}}+\frac{c_{2}}{A_{1}^{4}r^{4}}\right)\right|\leq\frac{C_0}{2r^{6}}
\]
for all $\tau\geq0$ and $r\in[1,\beta]$. Moreover, if $r\in[1,\beta]$
and $\tau\geq\underline{\tau}(\beta,A_{2},A_{3})$, we have
\begin{align*}
\left|z_{\text{ext}}(e^{-\frac{1}{2}\tau}r,\tau)-\left(\frac{A_{2}}{r^{2}}-\frac{A_{3}}{r^{4}}\right)\right|\leq & \left|\frac{(1-e^{-\tau}r^{2})}{r^{2}}\left(-\frac{A_{3}}{r^{2}}\left(-2e^{-\tau}r^{2}+2e^{-\tau}r^{2}(1-e^{-\tau}r^{2})\left(\log(1-e^{-\tau}r^{2})-\log(e^{-\tau}r^{2})\right)\right)\right)\right|\\
 & +\left|e^{-\tau}r^{2}\right|\cdot\left|\frac{A_{2}}{r^{2}}-\frac{A_{3}}{r^{4}}\right|\\
\leq & 4A_{3}\beta^{2}\left(1+\log(\beta^{2})\right)e^{-\tau}+(A_{2}+A_{3})\beta^{2}e^{-\tau}\leq\frac{1}{2r^{6}}.
\end{align*}
In particular, for any $D\in[A_{2}^{\frac{1}{2}}A_{3}^{-\frac{1}{2}},A_{2}^{\frac{1}{2}}A_{3}^{-\frac{1}{2}}\beta]$,
if we set $R_{D}:=D\sqrt{\frac{A_{3}}{A_{2}}}$, then whenever $\tau\geq\underline{\tau}(\beta,A_{2},A_{3},D)$,
we have
\begin{align*}
R_{D}^{2}\left(z_{\text{int}}(e^{-\frac{1}{2}\tau}R_{D},\tau)-z_{\text{ext}}(e^{-\frac{1}{2}\tau}R_{D},\tau)\right)\geq & A_{1}^{-2}-\left(A_{2}-A_{3}\cdot\tfrac{A_{2}}{A_{3}}D^{-2}\right)-C(A_{2},A_{3}) D^{-4}\\
\geq & A_{1}^{-2}-A_{2}\big(1-D^{-2}+C(A_2,A_3)D^{-4}\big),
\end{align*}
\begin{align*}
4R_{D}^{2}\left(z_{\text{int}}(2e^{-\frac{1}{2}\tau}R_{D},\tau)-z_{\text{ext}}(2e^{-\frac{1}{2}\tau}R_{D},\tau)\right)\leq & A_{1}^{-2}+c_2A_1^{-4}\cdot\tfrac{A_2}{4A_3}D^{-2}-\left(A_{2}-A_{3}\cdot\tfrac{A_{2}}{4A_{3}}D^{-2}\right)+C(A_{2},A_{3}) D^{-4}\\
\le & A_{1}^{-2}-A_{2}\left(1-\tfrac{1}{4}\left(1+\tfrac{c_2}{A_1^4A_3}\right)D^{-2}-C(A_2,A_3)D^{-4}\right).
\end{align*}
Choosing $\beta\geq \underline{\beta}(A_2,A_3)$, $D\geq \underline{D}(A_2,A_3,\beta)$, $A_3 \geq \underline{A_3}(A_2)$, and then $A_{1}^{-2}:=A_{2}\left(1-\frac{1}{2}D^{-2}\right)$ thus implies
that $z_{\text{int}}(e^{-\frac{1}{2}\tau}R_{D},\tau)>z_{\text{ext}}(e^{-\frac{1}{2}\tau}R_{D},\tau)$
and $z_{\text{int}}(e^{-\frac{1}{2}\tau}R_{D},\tau)<z_{\text{ext}}(e^{-\frac{1}{2}\tau}R_{D},\tau)$. As in \cite{WuTypeII}, define
\begin{align}\label{zstar}
z_{\ast}(u,\tau):=\begin{cases}
z_{\text{int}}(u,\tau), & u\in[0,e^{-\frac{1}{2}\tau}R_{D}]\\
\max\{z_{\text{int}}(u,\tau),z_{\text{ext}}(u,\tau)\}, & u\in[e^{-\frac{1}{2}\tau}R_{D},2e^{-\frac{1}{2}\tau}R_{D}]\\
z_{\text{ext}}(u,\tau). & u\in[2e^{-\frac{1}{2}\tau}R_{D},1]
\end{cases},
\end{align}
which is a subsolution in the sense of distributions (in particular,
we can apply the aforementioned maximum principle).

Moreover, we can estimate, for all $u\in(e^{-\frac{1}{2}\tau}R_{D},\frac{1}{2}]$,
\begin{align}\label{upperboundofext}
z_{\text{ext}}(u,\tau)= & A_{2}e^{-\tau}\frac{1}{u^{2}}(1-u^{2})^{2}-e^{-2\tau}A_{3}\frac{1}{u^{4}}(1-u^{2})^{2}\left(1-2u^{2}+2u^{2}(1-u^{2})\left(\log(1-u^{2})-2\log u\right)\right)\\\nonumber
\leq & A_{2}R_{D}^{-2}=A_{2}^{2}A_{3}^{-1}D^{-2}\leq D^{-2}.
\end{align}

\noindent The following proposition guarantees that an appropriately chosen subsolution fits underneath $z$.\\ 

\begin{prop}
For any smooth function $z_{0}:[0,1]\to(0,1)$ satisfying $z_{0}(0)=1$
and $z_{0}'(0)=0$, we can choose $D<\infty$ sufficiently large so
that there exists $\underline{\tau}>0$ with $z_{\ast}(\cdot,\tau)\leq z_{0}$
for all $\tau\geq\underline{\tau}$, where $z_*$ is the subsolution defined in (\ref{zstar}).
\end{prop}

\begin{proof}
In fact, we choose $D<\infty$ sufficiently large so that $D^{-2}\leq\text{inf}_{u\in(0,1]}z_{0}$.
We recall that $\mathcal{B}''(0)<0$, hence there exists $\delta>0$
universal such that $\mathcal{B}''(r)\leq-\frac{1}{2}\delta$ for
all $r\in[0,\delta]$. In particular,
\[
(\partial_{u}^{2}z_{\text{int}})(u,\tau)-z_{0}''(u)\leq A_{1}^{2}e^{\tau}\mathcal{B}''(A_{1}e^{\frac{1}{2}\tau}u)+\sup_{r\in[0,\frac{1}{2}]}|z_{0}''|<-\frac{1}{4}A_{1}^{2}e^{\tau}\delta
\]
whenever $\tau\geq\underline{\tau}(A_{1},\delta,z_{0})$ and $u\in[0,A_{1}^{-1}e^{-\frac{1}{2}\tau}\delta]$.
Upon integration, this implies (since $\mathcal{B}'(0)=0=z_{0}'(0)$
and $\mathcal{B}(0)=1=z_{0}(0)$) $\partial_{u}z_{\text{int}}(u,\tau)\leq z_{0}'(u)$
and $z_{\text{int}}(u,\tau)\leq z_{0}(u)$ for all $u\in[0,A_{1}^{-1}e^{-\frac{1}{2}\tau}\delta]$.
On the other hand, because $\mathcal{B}$ is strictly decreasing on
$(0,\infty)$, we can find $\delta_{1}>0$ such that $\inf_{r\in[A_{1}^{-1}\delta,2R_{D}]}\mathcal{B}'(r)<-\delta_{1}$.
Then 
\[
\partial_{u}z_{\text{int}}(u,\tau)-z_{0}'(u)<-\delta_{1}A_{1}e^{\frac{1}{2}\tau}-z_{0}'(u)<0
\]
for all $u\in[A_{1}^{-1}e^{-\frac{1}{2}\tau}\delta,2e^{-\frac{1}{2}\tau}R_{D}]$
whenever $\tau\geq\underline{\tau}(A_{1},\delta,\delta_{1},z_{0})$
is sufficiently large. Upon integration, we get $z_{\text{int}}(\cdot,\tau)\leq z_{0}$
on $[0,2e^{-\frac{1}{2}\tau}R_{D}]$. Furthermore, we have $z_{\text{ext}}(\cdot,t)\leq z_{0}$
on $[e^{-\frac{1}{2}\tau}R_{D},\frac{1}{2}]$ by (\ref{upperboundofext}) and our choice of $D$. In other words, $z_{\ast}(\cdot,\tau)\leq z_{0}$
on $[0,\frac{1}{2}]$ for $\tau$ large. Because $\inf_{u\in[\frac{1}{2},1]}z_{0}>0$,
and $\lim_{\tau\to\infty}\max_{u\in[\frac{1}{2},1]}|z_{\text{ext}}(u,\tau)|=0$,
the proposition follows.
\end{proof}
\begin{cor}
For any closed, rotationally invariant Ricci flow $(M^{n},(g_{t})_{t\in[0,T)})$
developing a singularity at time $T<\infty$, we have
\[
\limsup_{t\nearrow T}\max_{M}|Rm(\cdot,t)|(T-t)^{2}<\infty.
\]
\end{cor}

\begin{proof}
By symmetry, it suffices to show that 
\[
\limsup_{t\nearrow T}\max_{M}|Rm(o_{S},t)|(T-t)^{2}<\infty,
\]
where $o_{S}\in\mathbb{S}^{n+1}$ is the south pole (corresponding
to $x=-1$). We may (by considering the flow on a subinterval $[T',T)$
if necessary) assume that $x_{\ast}(t)$ is a smooth function representing
the position of the left-most bump of the solution, so that $\partial_{x}\psi(x,t)>0$
for all $x\in(-1,x_*(t))$. Then the function $z(u,\tau)$ is well-defined,
smooth, and positive on $\{0\leq u\leq1,\tau\geq0\}$; one may see from Lemma \ref{upperbound of u} that the domain of $z$ indeed contains $\{0\leq u\leq1,\tau\geq0\}$. Moreover, $u$
satisfies the equation 
\[
\partial_{\tau}z= \mathcal{D}_u[z]
\]
in this region. Choose $\tau_{0}<\infty$ according to the previous
lemma, so that $z_{\ast}(u,\tau)$ satisfies 
\[
\partial_{\tau}z_{\ast}< \mathcal{D}_u[z_{\ast}]
\]
on the region $\{0\leq u\leq1,\tau\geq\tau_{0}\}$, as well as $z_{\ast}(u,\tau_{0})<z(u,0)$
for all $u\in(0,1)$, and $z_{\ast}(0,\tau)=1$, $z_{\ast}(1,\tau)=0$
for all $\tau\geq0$. We observe that $z_{\text{comp}}(u,\tau):=z_{\ast}(u,\tau_{0}+\tau)$
then satisfies $z_{\text{comp}}(u,\tau)<z(u,0)$ for all $u\in(0,1)$,
and $z_{\text{comp}}(0,\tau)=1$, $z_{\text{comp}}(1,\tau)=0$ for
all $\tau\geq0$, as well as
\[
\partial_{\tau}z_{\text{comp}}< \mathcal{D}_u[z_{\text{comp}}]
\]
on the region $\{0\leq u\leq1,\tau\geq0\}$. We may thus apply the
maximum principle to obtain $z(u,\tau)\geq z_{\text{comp}}(u,\tau)$
for all $u\in[0,1]$, $\tau\in[0,\infty)$. In particular, the radial
curvature is given by 
\[
\frac{1-z(u,\tau)}{2(n-1)e^{-\tau}u^{2}}\le\frac{1-z_{\text{comp}}(u,\tau)}{2(n-1)e^{-\tau}u^{2}},
\]
so taking $u\searrow0$ tells us that the sectional curvature $K(o_{S},T-e^{-\tau})$
of the Ricci flow solution $(M,(g_{t})_{t\in[0,T)})$ at time $t=T-e^{-\tau}$
at the pole $o_{S}$ is bounded by 
\[
e^{\tau}\lim_{u\searrow0}\frac{1-z_{\text{comp}}(u,\tau)}{2(n-1)u^{2}}=e^{\tau}\lim_{u\searrow0}\frac{1-\mathcal{B}(A_{1}e^{\frac{1}{2}\tau}u)}{2(n-1)u^{2}}=\frac{b_{2}A_{1}^{2}}{2(n-1)}e^{2\tau}.
\]
In other words, 
\[
K(o_{S},t)\leq\frac{b_{2}A_{1}^{2}}{2(n-1)(T-t)^{2}}.
\]
\end{proof}

\begin{cor} If $\mathcal{M}$ is a rotationally invariant Ricci flow spacetime obtained as in \ref{spacetimeexists}, and if $t^{\ast}$ is a singular time such that there exists $\epsilon>0$ with no singular times in $[t^{\ast}-\epsilon,t^{\ast})$, then 
$$\limsup_{t\nearrow t^{\ast}}\max_{\mathcal{M}_t}|Rm|(\cdot)(t^{\ast}-t)^2 <\infty.$$ 
\end{cor}
\begin{proof} By pulling back $\mathcal{M}$ by a (possibly time-dependent) diffeomorphism, we can view $\mathcal{M}_{[t^{\ast}-\epsilon,t^{\ast})}$ as a smooth, closed Ricci flow.
\end{proof}

\section{Appendix: Preparatory results in Section 8 of the Bamler-Kleiner paper}

In this subsection we recall several results from \cite[Section 8]{BamlerKleiner17} which will be used to construct the comparison domain; however, we assume rotational invariance throughout, which allows us to strengthen the conclusions, and ensure that the modified a priori assumptions from section 5.1 hold. Since their proofs are not essentially different from those in \cite{BamlerKleiner17}, we shall not include the detailed proofs. 

\begin{lem}[Lemma 8.10 in \cite{BamlerKleiner17}, bounded curvature at bounded distance]\label{B-K Lemma 8.8}
For every $A<\infty$ there is a constant $C=C(n,A)<\infty$, such that if $$\epsilon_{\operatorname{can}}\leq\overline\epsilon_{\operatorname{can}}(A),$$ then the following holds. 

Let $0<r\leq 1$ and consider an $(n+1)$-dimensional rotationally invariant $(\epsilon_{\operatorname{can}}r,t_0)$-complete Ricci flow spacetime $\mathcal{M}$ that satisfies the equivariant $\epsilon_{\operatorname{can}}$-canonical neighborhood assumption at scales $(\epsilon_{\operatorname{can}}r,1)$. If $x\in\mathcal{M}_{[0,t_0]}$ and $\rho_1(x)\geq r$, then $P_{\mathcal{O}}(x,A\rho_1(x))$ is unscathed, and we have
\begin{eqnarray}
C^{-1}\rho_1(x)\leq\rho_1\leq\rho_1(x)\quad\text{ on }\quad P_{\mathcal{O}}(x,A\rho_1(x)).
\end{eqnarray}
\end{lem}


\begin{lem}[Lemma 8.22 in \cite{BamlerKleiner17}, scale distortion of bilipschitz maps]\label{B-K Lemma 8.22}
There is a constant $10^3<C_{\operatorname{SD}}<\infty$ such that the following holds if $$\eta_{\operatorname{lin}}\leq\overline \eta_{\operatorname{lin}},\quad \delta_{\operatorname{n}}\leq\overline \delta_{\operatorname{n}},\quad\epsilon_{\operatorname{can}}\leq\overline\epsilon_{\operatorname{can}},\quad r_{\operatorname{comp}}\leq \overline r_{\operatorname{comp}}.$$ Let $\mathcal{M}$ and $\mathcal{M}'$ be $(n+1)$-dimensional rotationally invariant $(\epsilon_{\operatorname{can}}r_{\operatorname{comp}},t_0)$-complete Ricci flow spacetimes. Consider a closed product domain $X\subset \mathcal{M}_{[0,t_0]}$ with rotationally invariant time-slices on a time-interval of the form $[t-r^2_{\operatorname{comp}},t]$, $t\geq r^2_{\operatorname{comp}}$, such that the following hold:
\begin{enumerate}[(i)]
    \item $\partial X_t$ consists of embedded $n$-spheres that are centers of $\delta_{\operatorname{n}}$-necks at scale $r_{\operatorname{comp}}$.
    \item Each connected component of $X_t$ contains a $2r_{\operatorname{comp}}$-thick point.
\end{enumerate}
Let $\bar t\in [t-r^2_{\operatorname{comp}},t]$ and $t'>0$, and consider a rotationally equivariant diffeomorphism onto its image $\phi:X_{\bar t}\rightarrow \mathcal{M}'_{t'}$ such that $|\phi^* g'_{t'}-g_{\bar t}|\leq \eta_{\operatorname{lin}}$. We assume that $\mathcal{M}$ satisfies the equivariant $\epsilon_{\operatorname{can}}$-canonical neighborhood assumption at scales $(0,1)$ on $X_{\bar t}$, and that $\mathcal{M}'$ satisfies the equivariant $\epsilon_{\operatorname{can}}$-canonical neighborhood assumption at scales $(0,1)$ on $\phi(X_{\bar t})$. Then for any $x\in X_{\bar t}$, we have
\begin{eqnarray}
C_{\operatorname{SD}}^{-1}\rho_1(x)\leq \rho_1(\phi(x))\leq C_{\operatorname{SD}}\rho_1(x).
\end{eqnarray}
\end{lem}

\begin{lem}[Lemma 8.32 in \cite{BamlerKleiner17}, time-slice necks imply spacetime necks]\label{B-K Lemma 8.32}
If $$\delta_\#>0,\quad 0<\delta\leq\overline \delta(\delta_\#),\quad0<\epsilon_{\operatorname{can}}\leq\overline \epsilon_{\operatorname{can}}(\delta_\#)\quad 0<r\leq\overline r,$$ then the following holds. 

Assume that $\mathcal{M}$ is an $(n+1)$-dimensional rotationally invariant and $(\epsilon_{\operatorname{can}} r,t_0)$-complete Ricci flow spacetime satisfying the equivariant $\epsilon_{\operatorname{can}}$-canonical neighborhood assumption at scales $(\epsilon_{\operatorname{can}}r,1)$. Let $a\in[-1,\frac{1}{4}]$ and consider a time $t\geq 0$ such that $t+ar^2\in[0,t_0]$. Assume that $U\subset \mathcal{M}_{t+ar^2}$ is a equivariant $\delta$-neck at scale $\sqrt{1-3a}r$. In other words, there is a rotationally equivariant diffeomorphism$$\psi_1:\mathbb{S}^n\times(-\delta^{-1},\delta^{-1})\rightarrow U$$ such that
\begin{eqnarray}
\left\|r^{-2}\psi_1^* g_{t+ar^2}-g_a^{\mathbb{S}^n\times\mathbb{R}}\right\|_{C^{[\delta^{-1}]}}<\delta.
\end{eqnarray}
Here $g_t^{cyl}:=ds^2+(\frac{2}{3}-2t)\overline{g}$, $t\in(-\infty,\frac{1}{3})$, is the shrinking round cylinder with $\rho(\cdot,0)=1$, where $\rho=(\frac{2}{3n(n-1)}R)^{-\frac{1}{2}}$ in this case, and the $C^{[\delta^{-1}]}$-norm is taken over the domain of $\psi_1$.

Then there is a product domain $U^*\subset\mathcal{M}_{[t-r^2,t+\frac{1}{4}r^2]\cap[0,t_0]}$ and an $r^2$-time-equivariant, $\partial_{\mathfrak{t}}$-preserving, and rotationally equivariant diffeomorphism $$\psi_2:\mathbb{S}^n\times(-\delta^{-1}_\#,\delta^{-1}_\#)\times[t^*,t^{**}]\rightarrow U^*$$ with $t+t^*r^2=\max\{t-r^2,0\}$ and $t+t^{**}r^2=\min\{t+\frac{1}{4}r^2,t_0\}$, such that $$\psi_2\big\vert_{\mathbb{S}^n\times(-\delta^{-1}_\#,\delta^{-1}_\#)\times\{a\}}=\psi_1\big\vert_{\mathbb{S}^n\times(-\delta^{-1}_\#,\delta^{-1}_\#)}$$ and $$\left\|r^{-2}\psi_2^* g-g^{\mathbb{S}^n\times\mathbb{R}}\right\|_{C^{[\delta^{-1}_\#]}}<\delta_\#.$$ Here the $C^{[\delta^{-1}_\#]}$-norm is taken over the domain of $\psi_2$.
\end{lem}

\begin{lem}[Lemma 8.40 in \cite{BamlerKleiner17}, nearly increasing scale implies Bryant-like geometry]\label{B-K Lemma 8.40}
If $$\alpha,\delta>0,\quad 1\leq J<\infty,\quad\beta\leq\overline\beta(\alpha,\delta,J),\quad\epsilon_{\operatorname{can}}\leq\overline \epsilon_{\operatorname{can}}(\alpha,\delta,J),\quad r\leq\overline r(\alpha),$$ then the following holds.

Let $0<r\leq 1$. Assume that $\mathcal{M}$ is an $(n+1)$-dimensional rotationally invariant and $(\epsilon_{\operatorname{can}}r,t_0)$-complete Ricci flow spacetime that satisfies the equivariant $\epsilon_{\operatorname{can}}$-canonical neighborhood assumption at scales $(\epsilon_{\operatorname{can}}r,1)$. Let $t\in[Jr^2,t_0]$ and $x\in\mathcal{M}_t$. Assume that $x$ survives until time $t-r^2$ and that
\begin{gather*}
    \alpha r\leq\rho(x)\leq\alpha^{-1}r,
    \\
    \rho(x(t-r^2))\leq\rho(x)+\beta r.
\end{gather*}
Let $a\in[\rho(x(t-r^2)),\rho(x)+\beta r]$.

Then $(\mathcal{M}_{t'},x(t'))$ is $\delta$-close to $(M_{\operatorname{Bry}},g_{\operatorname{Bry}},x_{\operatorname{Bry}})$ at scale $a$ for all $t'\in[t-r^2,r]$. Furthermore, there is an $a^2$-time-equivariant, $\partial_{\mathfrak{t}}$-preserving, and rotationally equivariant diffeomorphism onto its image $$\psi:\overline{M_{\operatorname{Bry}}(\delta^{-1})}\times[-J\cdot(ar^{-1})^{-2},0]\rightarrow\mathcal{M}$$ such that $\psi(x_{\operatorname{Bry}},0)=x$ and$$\left\|a^{-2}\psi^*g-g_{\operatorname{Bry}}\right\|_{C^{[\delta^{-1}]}}<\delta,$$ where the norm is taken over the domain of $\psi$.
\end{lem}

\begin{lem}[Lemma 8.41 in \cite{BamlerKleiner17}, Bryant slice lemma] \label{B-K Lemma 8.41}
If 
\begin{gather*}
    \delta_{\operatorname{n}}\leq\overline{\delta}_{\operatorname{n}},\quad 0<\lambda<1,\quad\Lambda\geq\underline{\Lambda},\quad \delta\leq\overline{\delta}(\lambda,\Lambda),
\end{gather*}
then the following holds for some $D_0=D_0(\lambda)<\infty$.

Consider an $(n+1)$-dimensional $O(n+1)$-invariant Ricci flow spacetime $\mathcal M$ and let $r>0$ and $t\geq 0$. Consider a subset $X\subset \mathcal M_t$ such that the following holds
\begin{enumerate}[(i)]
    \item $X$ is a closed and $O(n+1)$-invariant subset (consisting of $O(n+1)$-orbits) with smooth boundaries. 
    \item The boundary components of $X$ are central $O(n+1)$-orbits of $O(n+1)$-equivariant $\delta_{\operatorname{n}}$-necks at scale $r$.
    \item $X$ contains all $\Lambda r$-thick points of $\mathcal M_t$.
    \item Every component of $X$ contains a $\Lambda r$-thick point. 
\end{enumerate}
Consider the image $W$ of an $O(n+1)$-equivariant diffeomorphism
$$\psi:W^*:=\overline{M_{\operatorname{Bry}}(d)}\rightarrow W\subset\mathcal M_t,$$
such that $d\geq\delta^{-1}$ and
$$\left\|(10\lambda r)^{-2}\psi^*g_t-g_{\operatorname{Bry}}\right\|_{C^{[\delta^{-1}]}(W^*)}<\delta.$$
Then $\psi(x_{\operatorname{Bry}})$ is $11\lambda r$-thin. Moreover, if $\mathcal C:=W\setminus\operatorname{Int}X\not=\emptyset$, then
\begin{enumerate}[(a)]
    \item $\mathcal C$ is a $O(n+1)$-invariant $(n+1)$-disk containing $\psi(x_{\operatorname{Bry}})$.
    \item $\mathcal C$ is a component of $\mathcal M_t\setminus \operatorname{Int}X$ and $\partial\mathcal C\subset\partial X$.
    \item $\mathcal C$ is $9\lambda r$-thick and $1.1r$-thin.
    \item $\mathcal C\subset \psi(M_{\operatorname{Bry}}(D_0(\lambda)))\subset\operatorname{Int}W$.
\end{enumerate}
\end{lem}

\begin{lem}[Lemma 8.42 in \cite{BamlerKleiner17}, Bryant slab lemma] \label{B-K Lemma 8.42}
If 
\begin{gather*}
    \delta_{\operatorname{n}}\leq\overline{\delta}_{\operatorname{n}},\quad 0<\lambda<1,\quad\Lambda\geq\underline{\Lambda},\quad \delta\leq\overline{\delta}(\lambda,\Lambda),
\end{gather*}
then the following holds.

Consider an $(n+1)$-dimensional $O(n+1)$-invariant Ricci flow spacetime $\mathcal M$ and let $r>0$ and $t_0\ge 0$. Set $t_1:=t_0+r^2$. For $i=0,1$, let $X_i\subset\mathcal M_{t_i}$ be a closed subset that is a domain with boundary, satisfying conditions (i)--(iv) from Lemma \ref{B-K Lemma 8.41}, and in addition:
\begin{enumerate}[(i)]
    \setcounter{enumi}{4}
    \item $X_1(t)$ is defined for all $t\in[t_0,t_1]$, and $\partial X_1(t_0)\subset\operatorname{Int}X_0$.
\end{enumerate}

Consider a ``$\delta$-good Bryant slab'' in $\mathcal M_{[t_0,t_1]}$, i.e. the image $W$ of a map
$$\psi:W^*=\overline{M_{\operatorname{Bry}}(d)}\times[-(10\lambda)^{-2},0]\rightarrow\mathcal M_{[t_0,t_1]}$$
where $d\ge \delta^{-1}$ and $\psi$ is a $(10\lambda r)^2$-time equivariant and $\partial_{\mathfrak{t}}$-preserving diffeomorphism onto its image, whose time-slices are all $O(n+1)$-equivariant, and
$$\left\|(10\lambda r)^{-2}\psi^*g-g_{\operatorname{Bry}}\right\|_{C^{[\delta^{-1}]}(W^*)}<\delta.$$

Set $\mathcal C_i:=W_{t_i}\setminus\operatorname{Int}X_i\subset\mathcal M_{t_i}$ for $i=0,1$. Then
\begin{enumerate}[(1)]
    \item $\mathcal C_i(t)$ is well-defined and $9\lambda r$-thick for all $t\in[t_0,t_1]$ and
    \item if $\mathcal C_1\neq\emptyset$, then $\mathcal C_0\subset\mathcal C_1(t_0)$ and $\mathcal C_0=\mathcal C_1(t_0)\setminus\operatorname{Int} X_0$.
\end{enumerate}
\end{lem}

\Addresses
\end{document}